\documentclass[11pt]{article}

\newcommand{\tn}{\textnormal}
\newcommand{\mb}{\mathbb}
\newcommand{\mc}{\mathcal}
\newcommand{\lp}{\left(}
\newcommand{\rp}{\right)} \newcommand{\lb}{\left\lbrace}
\newcommand{\rb}{\right\rbrace}

\usepackage[normalem]{ulem}
\usepackage{soul}

\usepackage {epsfig,amsmath,euscript}

\usepackage[latin1]{inputenc}

\usepackage[english]{babel}
\usepackage{color}
\usepackage{subcaption}
\usepackage{amsmath}
\usepackage{amsthm}

\usepackage{amssymb}

\usepackage{ae}
\usepackage{float}
\usepackage[T1]{fontenc}
\usepackage{graphicx}
\usepackage{epstopdf}
\usepackage{longtable}
\usepackage{color}
\definecolor{rred}{rgb}{0.7,0.0,0.2}
\definecolor{bblue}{rgb}{0.2,0.0,0.7}

\newcommand{\secref}[1]{Section \ref{sec:#1}}

\newtheorem{definition}{Definition}
\newtheorem{proposition}{Proposition}
\newtheorem{remark}{Remark}
\newtheorem{theorem}{Theorem}

\newtheorem{lemma}{Lemma}
\newcommand{\seclab}[1]{\label{sec:#1}}

\newcommand{\eqlab}[1]{\label{eq:#1}}
\renewcommand{\eqref}[1]{(\ref{eq:#1})}

\newcommand{\figref}[1]{Figure~\ref{fig:#1}}
\newcommand{\figlab}[1]{\label{fig:#1}}
\newcommand{\propref}[1]{Proposition~\ref{proposition:#1}}
\newcommand{\proplab}[1]{\label{proposition:#1}}

\newcommand{\defnref}[1]{Definition~\ref{definition:#1}}
\newcommand{\defnlab}[1]{\label{definition:#1}}
\newcommand{\lemmaref}[1]{Lemma~\ref{lemma:#1}}
\newcommand{\lemmalab}[1]{\label{lemma:#1}}
\newcommand{\remref}[1]{Remark~\ref{remark:#1}}
\newcommand{\remlab}[1]{\label{remark:#1}}
\newcommand{\thmref}[1]{Theorem~\ref{theorem:#1}}
\newcommand{\thmlab}[1]{\label{theorem:#1}}

\newcommand{\asuref}[1]{Assumption~\ref{assumption:#1}}
\newcommand{\asulab}[1]{\label{assumption:#1}}

\newtheorem{cor}{Corollary}

\newtheorem{asu}{Assumption}

\makeatletter
\newcommand{\tpitchfork}{%
	\vbox{
		\baselineskip\z@skip
		\lineskip-.52ex
		\lineskiplimit\maxdimen
		\m@th
		\ialign{##\crcr\hidewidth\smash{$-$}\hidewidth\crcr$\pitchfork$\crcr}
	}%
}
\makeatother

\title{Bifurcations of mixed-mode oscillations in three-timescale systems: an extended prototypical example}
\date{}
\author{P. Kaklamanos, N. Popovi\'c, and K. U. Kristiansen\footnote{P. Kaklamanos and N. Popovi\'c: School of Mathematics, University of Edinburgh, James Clerk Maxwell Building, King's Buildings, Peter Guthrie Tait Road, Edinburgh, EH9 3FD, United Kingdom (\texttt{p.kaklamanos@sms.ed.ac.uk} and \texttt{nikola.popovic@ed.ac.uk}); K. U. Kristiansen: Department of Applied Mathematics and Computer Science, Technical University of Denmark, Asmussens All\'e, Building 303B, 2800 Kgs. Lyngby, Denmark (\texttt{krkri@dtu.dk})}}

\begin{document}
	

\maketitle


\pagestyle{myheadings}
\thispagestyle{plain}

\begin{abstract}
We study a class of multi-parameter three-dimensional systems of ordinary differential equations that exhibit dynamics on three distinct timescales. We apply geometric singular perturbation theory to explore the dependence of the geometry of these systems on their parameters, with a focus on mixed-mode oscillations (MMOs) and their bifurcations. In particular, we uncover a novel geometric mechanism that encodes the transition from MMOs with single epochs of small-amplitude oscillations (SAOs) to those with double-epoch SAOs. We identify a relatively simple prototypical three-timescale system that realises our mechanism, featuring a one-dimensional $S$-shaped supercritical manifold that is embedded into a two-dimensional $S$-shaped critical manifold in a symmetric fashion. We show that the Koper model from chemical kinetics is merely a particular realisation of that prototypical system for a specific choice of parameters; in particular, we explain the robust occurrence of mixed-mode dynamics with double epochs of SAOs therein. Finally, we argue that our geometric mechanism can elucidate the mixed-mode dynamics of more complicated systems with a similar underlying geometry, such as of a three-dimensional, three-timescale reduction of the Hodgkin-Huxley equations from mathematical neuroscience. 
\end{abstract}


\section{Introduction}\seclab{Introduction}

The Koper model from chemical kinetics  \cite{koper1995bifurcations} is typically written as
\begin{subequations}\eqlab{koper1}
	\begin{align}
	\varepsilon \dot{x} &=ky+3x-x^3-\lambda, \eqlab{koper1-a}\\
	\dot{y} &=x-2y+z, \eqlab{koper1-b}\\
	\dot{z} &=\delta \lp y-z\rp, \eqlab{koper1-c}
	\end{align}
\end{subequations}
with $k,\lambda\in\mb{R}$ and $\varepsilon$ and $\delta$ real and positive parameters. When $ \varepsilon $ is sufficiently small, Equation~\eqref{koper1} exhibits dynamics on two distinct timescales: the variable $x$ is then called the \textit{fast variable}, while the variables $y$ and $z$ are the \textit{slow variables}; correspondingly, Equation~\eqref{koper1-a} is denoted the \textit{fast equation}, whereas \eqref{koper1-b} and \eqref{koper1-c} are called \textit{slow equations}, respectively. On the other hand, when both $ \varepsilon$ and $\delta$ are small, Equation~\eqref{koper1} is a three-timescale system; the variables $x$, $y$, and $z$ are then called the \textit{fast}, \textit{intermediate}, and \textit{slow variables}, respectively. Correspondingly, Equations~\eqref{koper1-a}, \eqref{koper1-b}, and \eqref{koper1-c} are denoted the \textit{fast}, \textit{intermediate}, and \textit{slow equations}, respectively.

Multiple-scale systems of ordinary differential equations frequently feature \textit{mixed-mode oscillations} (MMOs); these are trajectories that are characterised by the alternation of small-amplitude oscillations (SAOs) and large-amplitude excursions (LAOs) in the corresponding time series.
A particularly fruitful approach for the study of mixed-mode dynamics in singularly perturbed slow-fast systems of the type of the Koper model, Equation~\eqref{koper1}, is based on dynamical systems theory, combining Fenichel's geometric singular perturbation theory (GSPT) \cite{fenichel1979geometric} with the desingularisation technique known as ``blow-up"  \cite{krupa2001extending}.
Of particular relevance to that approach are localised, non-hyperbolic singularities (``canard points") on the corresponding critical manifolds which generate SAOs in the resulting MMO trajectories, whereas LAOs arise via a global return mechanism along normally hyperbolic portions of those manifolds. 
A relatively recent, exhaustive review of this so-called ``generalised canard mechanism" for the emergence of MMOs can be found in \cite{desroches2012mixed}.

\begin{figure}[ht!]
	\centering
	\begin{subfigure}[b]{0.45\textwidth}
		\centering
		\includegraphics[scale = 0.45]{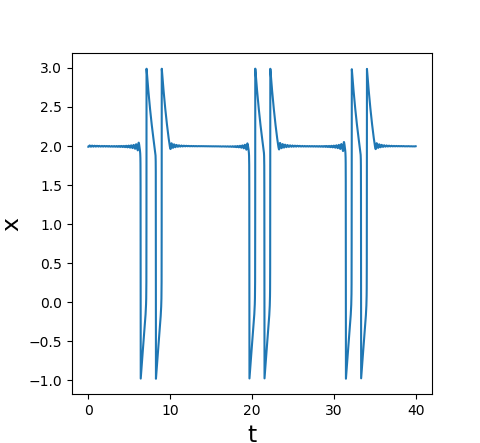}
		\caption{SAOs ``above" ($k = -4.5$, $\lambda = -2.0$).}
	\end{subfigure}
	~
	\begin{subfigure}[b]{0.45\textwidth}
		\centering
		\includegraphics[scale = 0.45]{pics/intro/below}
		\caption{SAOs ``below" ($k = -4.5$, $\lambda = 2.0$).}
	\end{subfigure}
	\\
	\centering
	\begin{subfigure}[b]{0.45\textwidth}
		\centering
		\includegraphics[scale = 0.45]{pics/intro/double}
		\caption{Double epochs of SAOs ($k = -4.0$, $\lambda = 0.0$).}
	\end{subfigure}
	~
	\begin{subfigure}[b]{0.45\textwidth}
		\centering
		\includegraphics[scale = 0.45]{pics/intro/relaxation}
		\caption{Relaxation oscillation ($k = -4.5$, $\lambda = 0.0$).}
	\end{subfigure}
	\caption{Oscillatory dynamics in the Koper model, Equation~\eqref{koper1}, for different values of the parameters $k$ and $\lambda$. (a) MMO trajectory with single epochs of SAOs and Farey sequence $2^{s_1}2^{s_2}2^{s_3}\cdots$; (b) MMO trajectory with single epochs of SAOs and Farey sequence $2_{s_1}2_{s_2}2_{s_3}\cdots$; (c) MMO trajectory with double epochs of SAOs and Farey sequence $1^{s_1}1_{s_2}1^{s_3}1_{s_4}\cdots$; (d) relaxation oscillation.}
	\figlab{fareys}
\end{figure}

Representative MMO trajectories that are realised in the three-timescale Equation~\eqref{koper1} can be seen in \figref{fareys}, where we set $\varepsilon=0.01=\delta$ throughout.
Each such trajectory can be associated with a sequence of the form $\{F_0F_1\ldots\}$, called the \textit{Farey sequence}, which describes the succession of large excursions and small oscillations, where the segments $F_j$ are of the form 
\begin{align*}
F_j=\begin{cases}
{L}^s \tn{ if the segment consists of $L$ LAOs, followed by $s$ SAOs ``above''}; \\
{L}_s \tn{ if the segment consists of $L$ LAOs, followed by $s$ SAOs ``below''}.
\end{cases}
\end{align*}
If a Farey sequence consists of $L^s$-type or $L_s$-type segments only, we say that the corresponding MMO trajectory contains \textit{single epochs} of SAOs, as seen in panels (a) and (b) of \figref{fareys}, respectively; Farey sequences that consist of both $L^s$-type and $L_s$-type segments correspond to MMO trajectories that contain \textit{double epochs} of SAOs, as shown in \figref{fareys}(c). Finally, relaxation oscillation refers to 
oscillatory trajectories that contain large excursions and no SAO segments, i.e., trajectories with associated Farey sequence $\{L^0\}$; cf.~\figref{fareys}(d).

\bigskip



MMOs in the Koper model have been extensively studied in the two-timescale context, i.e., for $\varepsilon>0$ sufficiently small and $\delta=\mathcal{O}(1)$ in Equation~\eqref{koper1} \cite{desroches2012mixed,koper1995bifurcations,kuehn2011decomposing}. However, to our knowledge, there are no equivalent studies in the literature of the three-timescale Koper model, with $\varepsilon$ and $\delta$ small in \eqref{koper1}, which is the scenario we will consider in this article. In the process, we will uncover a geometric mechanism that encodes bifurcations of MMOs and, in particular, the transition from MMOs with single epochs to double epochs of SAOs therein; recall \figref{fareys}(c).

Rather than formulating our mechanism within the framework of the Koper model, Equation\eqref{koper1}, we will first consider the analytically simpler family of slow-fast systems
\begin{subequations}\eqlab{normal}
\begin{align}
\varepsilon\dot{x} &=-y + f_2x^2+f_3x^3=:f(x,y), \eqlab{normal-a} \\
\dot{y} &=\alpha x+\beta y{-z}=:g(x,y,z), \eqlab{normal-b} \\
\dot{z} &=\delta \lp \mu +\phi\lp x, y,z\rp\rp=:\delta h(x,y,z) \eqlab{normal-c}
\end{align}
\end{subequations}
which can be obtained from \eqref{koper1} via a sequence of affine transformations, with 
\begin{subequations}\eqlab{kop_params}
\begin{gather}
\varepsilon = \frac\epsilon{|k|},\quad f_2 = \frac3{|k|}, \quad f_3 = -\frac1{|k|}, \\
\alpha = 1, \quad \beta = -2, \\
\mu = \frac{k+\lambda+2}k,\quad\text{and}\quad \phi(x,y,z) = -y-z.
\end{gather}
\end{subequations}
Our motivation for introducing Equation~\eqref{normal} is two-fold: first, the geometry of the Koper model in \eqref{kop_params} will turn out to be quite restrictive, as variation of the parameter $k$ in \eqref{koper1-a} affects both the associated invariant manifolds and the reduced flow thereon. In \eqref{normal}, on the other hand, the effect of the corresponding parameters $f_2$, $f_3$, and $\mu$ on the geometry can be studied independently. Second, it will become apparent that the geometric mechanism described here is generic, in that it transcends the Koper model proper; correspondingly, we propose Equation~\eqref{normal}, with $f_2>0$, $f_3<0$, $\alpha$, $\beta$, and $\mu$ real parameters and $\varepsilon$ and $\delta$ sufficiently small, as a ``prototypical", normal form-type model which encapsulates our mechanism.

Mixed-mode dynamics in three-timescale slow-fast systems of the type in \eqref{normal} has been studied before; see, e.g., \cite{desroches2012mixed,nan2014dynamical,de2014three,de2016sector} for specific examples and further references. In particular, ``prototypical" models akin to the one in Equation~\eqref{normal} have been considered by Krupa et al. in \cite{krupa2008mixed} and by Letson et al. in \cite{letson2017analysis}; the corresponding systems of equations are given by
\begin{subequations}\eqlab{prototypical}
\begin{align}
\varepsilon\dot{x} &=-y+f_2x^2+f_3x^3, \eqlab{prototypical-a}\\
\dot{y} &=x-z, \eqlab{prototypical-b} \\
\dot{z} &=\varepsilon \lp \mu+\phi(x,y,z)\rp \eqlab{prototypical-c}
\end{align}
\end{subequations}
and by
\begin{subequations}\eqlab{canonic}
\begin{align}
\varepsilon\dot{x} &=y+x^2, \eqlab{canonic-a} \\
\dot{y} &=-\alpha^2 x+\beta y+z, \\
\dot{z}&=\delta,
\end{align}
\end{subequations}
respectively.

However, it is worth emphasising that our prototypical model, Equation~\eqref{normal}, is substantively different from both Equations~\eqref{prototypical} and \eqref{canonic}, in spite of the evident similarities between the three systems. Specifically, Equation~\eqref{prototypical} refers to the special case of $\alpha = 1$, $\beta=0$ and $\delta=\varepsilon$ in \eqref{normal}. As will become clear in the following, the absence of a linear $y$-term in \eqref{prototypical-b} makes a crucial difference geometrically, as it is precisely this term in \eqref{prototypical-b} which generates MMOs with double SAO epochs in the three-timescale regime. The canonical form in \eqref{canonic}, on the other hand, does capture local phenomena and properties of SAOs for $\beta\neq 0$ therein; however, as no cubic $x$-term is present in \eqref{canonic-a}, it does not allow for LAO-type dynamics via a global return mechanism, nor does it admit true equilibria. In that sense, Equation~\eqref{normal} combines aspects of both \eqref{prototypical} and \eqref{canonic}, yielding rich oscillatory behaviour which has, to the best of our knowledge, not previously been classified in the three-timescale context.

Correspondingly, our principal aim in this article is a classification of the mixed-mode dynamics in our ``generalised prototypical model", Equation~\eqref{normal}. Then, we will apply that classification to the realisation thereof that is provided by the Koper model, Equation~\eqref{koper1}, in the three-timescale scenario where $\varepsilon$ and $\delta$ are sufficiently small, on the basis of Fenichel's geometric singular perturbation theory (GSPT) \cite{fenichel1979geometric}; the resulting bifurcation diagram, in terms of the parameters $k$ and $\lambda$, is shown in \figref{kl-plane0}.  In particular, we will explain the robust occurrence of mixed-mode dynamics with double epochs of SAOs in the three-timescale Koper model; by contrast, double-epoch MMOs have only been observed in very narrow parameter regimes in the two-timescale case \cite{desroches2012mixed}.
Throughout, we will focus on the novel singular geometry of Equations~\eqref{normal} and \eqref{koper1}, i.e., on the double singular limit of $\varepsilon=0=\delta$ therein, as well as on perturbations off that limit in either $\varepsilon$ or $\delta$. Subsequently, we will comment on the qualitative mixed-mode dynamics which is expected to result from a full two-parameter perturbation analysis, as is also evidenced by numerical simulation.

Finally, we will argue that the geometric mechanism described here is ``generic'', in the sense that it allows for the classification of complex mixed-mode dynamics in more complicated systems with similar geometric properties, such as in a three-dimensional reduction of the Hodgkin-Huxley equations from mathematical neuroscience \cite{HHmain,doi2001complex,rubin2007giant}.

\begin{figure}[ht!]
	\centering
	\includegraphics[scale=0.34]{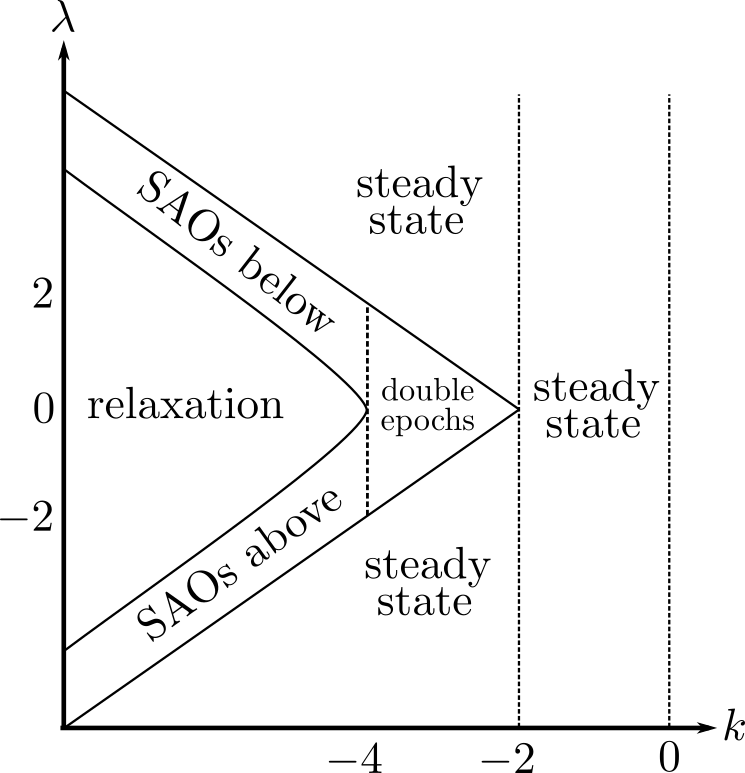}
	\caption{Two-parameter bifurcation diagram of the three-timescale Koper model, Equation~\eqref{koper1}, to leading order in $\varepsilon$ and $\delta$; see \secref{koper} for details.}
	\figlab{kl-plane0}
\end{figure}

The article is organised as follows. In \secref{singgeom}, we describe the geometry of the three-time-scale Equation~\eqref{normal} in the double singular limit of $\varepsilon=0=\delta$: we define critical and supercritical manifolds; then, we construct families of singular cycles which form the basis for MMO trajectories of Equation~\eqref{normal}.
In \secref{singpert1}, we study the singularly perturbed system in \eqref{normal} for $\varepsilon$ and $\delta$ sufficiently small; we classify the mixed-mode dynamics of \eqref{normal}, as illustrated in \figref{fareys}, by establishing a correspondence with the cycles constructed in \secref{singgeom}. In \secref{koper}, we apply our results to the Koper model from chemical kinetics, Equation~\eqref{koper1}, and we elucidate in detail the structure of the two-parameter bifurcation diagram in \figref{kl-plane0}. We conclude in \secref{conclusion} with a discussion, and an outlook to future research; in particular, we indicate how our analysis can be extended to to a three-dimensional reduction of the Hodgkin-Huxley equations derived by Rubin and Wechselberger \cite{rubin2007giant} which generalises our extended prototypical example, Equation~\eqref{normal}. Finally, in Appendix~\ref{mechs}, we provide additional detail on SAO-generating mechanisms in the three-timescale context considered here.

\section{The double singular limit: geometry and singular cycles}
\seclab{singgeom}

In this section, we study the double singular limit of $\varepsilon=0=\delta$ in Equation~\eqref{normal}. To that end, we first describe the singular geometry for $\varepsilon=0$; then, we consider the resulting flow in the limit of $\delta\to 0$. Finally, we construct singular cycles which will form the basis of MMO trajectories for Equation~\eqref{normal} when $\varepsilon$ and $\delta$ are sufficiently small, as considered in \secref{singpert1} below.

\subsection{The critical manifold $\mathcal{M}_1$}

For $\varepsilon$ sufficiently small and $\delta=\mathcal{O}(1)$ fixed, Equation~\eqref{normal} is singularly perturbed with respect to the small parameter $\varepsilon$; in particular, \eqref{normal} describes the dynamics in terms of the \textit{intermediate} time $t$. Rewriting the governing equations in the fast time 
$\tau=t/\varepsilon$, we have
\begin{subequations}\eqlab{norm12fast}
\begin{align}
{x}' &= -y + f_2x^2+f_3x^3, \\
{y}' &= \varepsilon \lp \alpha x+\beta y{-z}\rp, \\
{z}' &=  \varepsilon \delta \lp \mu +\phi\lp x, y,z\rp\rp, 
\end{align}
\end{subequations}
which is a two-timescale system with one fast variable $x$ and two slow variables $y$ and $z$. The reduced problem of the above is obtained by setting $\varepsilon=0$ in \eqref{normal},
\begin{subequations}\eqlab{norm12red}
\begin{align}
0 &=-y+f_2x^2+f_3x^3, \\
\dot{y} &=\alpha x+\beta y{-z}, \\
\dot{z} &=\delta \lp \mu +\phi\lp x,y,z\rp\rp,
\end{align}
\end{subequations}
while the layer problem is found for $\varepsilon=0$ in \eqref{norm12fast}:
\begin{subequations}\eqlab{norm12lay}
\begin{align}
{x}' &=-y + f_2x^2+f_3x^3, \eqlab{norm12lay-a} \\
{y}' &=0, \eqlab{norm12lay-b}\\
{z}' &=0.	\eqlab{norm12lay-c}
\end{align}
\end{subequations}
We will refer to the flow that is induced by the one-dimensional vector field in Equation~\eqref{norm12lay} as the \textit{fast flow}; the corresponding trajectories will be denoted as the \textit{fast fibres}. The critical manifold $\mathcal{M}_1$ for \eqref{normal} is a set of equilibria for \eqref{norm12lay}, and is given by
\begin{align}\eqlab{M1}
\mathcal{M}_1 := 
\lb\lp x,y,z\rp \in\mb{R}^3~\big\lvert ~ f(x,y) =0 \rb =
\lb\lp x,y,z\rp \in\mb{R}^3~\big\lvert ~ y = F(x) \rb,
\end{align}
where we define
\begin{align}
F(x) = f_2x^2+f_3x^3 \eqlab{Fx}.
\end{align}
The manifold $\mathcal{M}_1$ can be written as $ \mathcal{M}_1 = \mathcal{S}^a \cup \mathcal{S}^r \cup \mathcal{F}_{\mathcal{M}_1} $, 
where
\begin{align*}
\mathcal{S}^a = \lb \lp x,y,z\rp\in\mathcal{S} ~\Big|~ \frac{\partial f}{\partial x}(x,y) < 0\rb\quad\text{and}\quad \mathcal{S}^r =\lb\lp x,y,z\rp\in\mathcal{S} ~\Big|~ \frac{\partial f}{\partial x}(x,y)> 0\rb
\end{align*}
are normally attracting and normally repelling, respectively, whereas
$\mathcal{F}_{\mathcal{M}_1}$ is degenerate due to a loss of normal hyperbolicity:
\begin{align}
\mathcal{F}_{\mathcal{M}_1} := \lb\lp x,y,z\rp\in\mathcal{M}_1 ~\Big|~\frac{\partial f}{\partial x}(x,y) =0\rb
=\lb\lp x,y,z\rp \in\mathcal{M}_1~\big\lvert ~ x(2f_2+3f_3x) =0 \rb. 
\end{align}
In particular, we may write $\mathcal{F}_{\mathcal{M}_1} = \mathcal{L}^-\cup\mathcal{L}^+$, where
\begin{align}\eqlab{els}
\mathcal{L}^{-} = \lb \lp x,y,z\rp\in\mb{R}^3~\big|~x =0=y\rb\quad\text{and}\quad \mathcal{L}^+ = \lb \lp x,y,z\rp \in\mathbb{R}^3 ~\bigg|~ x =-\frac{2}{3}\frac{f_2}{f_3}\ \text{and}\ y = \frac{4}{27}\frac{f^3_2}{f^2_3}\rb;
\end{align}	
hence, it follows that $\mathcal{S}^a = \mathcal{S}^{a^-}\cup\mathcal{S}^{a^+}$, with
\begin{align}\eqlab{sa}
\mathcal{S}^{a^-} = \lb \lp x,y,z\rp\in\mathcal{S} ~|~ x<0\rb\quad\text{and}\quad
\mathcal{S}^{a^+} = \lb \lp x,y,z\rp\in\mathcal{S} ~\bigg|~ x>-\frac{2}{3}\frac{f_2}{f_3}\rb,
\end{align}
while
\begin{align}\eqlab{sr}
\mathcal{S}^{r} = \lb \lp x,y,z\rp\in\mathcal{S} ~\Big|~ 0<x<-\frac{2}{3}\frac{f_2}{f_3}\rb. 
\end{align}

The normally hyperbolic portion $\mathcal{S}$ of  $\mathcal{M}_1$ therefore consists of a repelling middle sheet $S^r$ and two attracting sheets $S^{a^\mp}$ that meet $\mathcal{S}^r$ along $\mathcal{L}^\pm$, respectively; see \figref{em1}(a).
From the above, it is apparent that $\mathcal{L}^-$ always coincides with the $z$-axis, whereas variation in $f_2$ and $f_3$ translates $\mathcal{L}^+$, therefore ``stretching'' or ``compressing'' $\mathcal{M}_1$. (Clearly, variation in $\alpha$, $\beta$, and $\mu$ has no effect on the geometry of $\mathcal{M}_1$.) Finally, the elements of the sets $\mathcal{Q}^\mp$ defined by
\begin{align*}
\mathcal{Q}^\mp=\lb (x,y,z) \in\mathcal{L}^\mp~\lvert~ f(x,y)=0=g(x,y,z)\rb
\end{align*}
are called the \textit{folded singularities} of $\mathcal{M}_1$ on $\mathcal{L}^\mp$, respectively \cite{szmolyan2001canards}; for \eqref{normal}, these sets are the singletons
$\mathcal{Q}^- = \lb q^-\rb$ and $\mathcal{Q}^+ = \lb q^+\rb$, with $q^\mp= (x_q^\mp,y_q^\mp,z_q^\mp)$  located at
\begin{gather}\eqlab{foldsing}
\begin{gathered}
x_q^-  = 0, \quad y_q^- =0, \quad\text{and}\quad z_q^-=0,\quad\text{as well as at} \\
x_q^+ =-\frac{2f_2}{3f_3},\quad y_q^+ = \frac{4f_2^3}{27f_3^2},\quad\text{and}\quad z_q^+ = {\frac{4\beta f_2^3}{27f_3^3}-\frac{2\alpha f_2}{3f_3}}, 
\end{gathered}
\end{gather}
respectively.
\begin{figure}[ht!]
	\centering
	\begin{subfigure}[b]{0.45\textwidth}
		\centering
		\includegraphics[scale = 0.40]{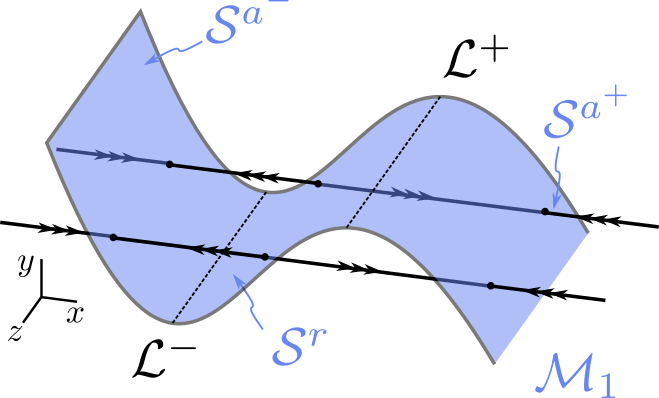}
		\caption{}
	\end{subfigure}
	~
	\begin{subfigure}[b]{0.45\textwidth}
		\centering
		\includegraphics[scale = 0.40]{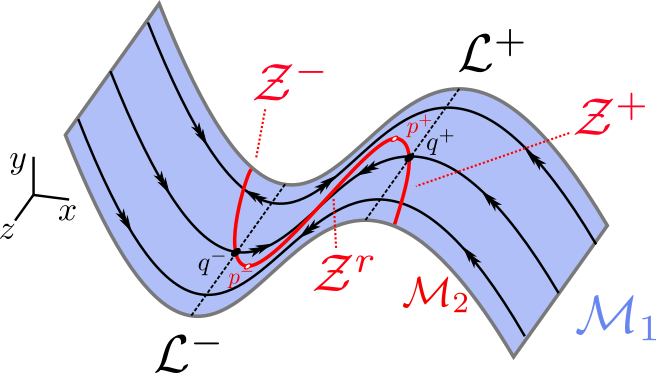}
		\caption{}
	\end{subfigure}
	\caption{(a) The critical manifold $\mathcal{M}_1$ as the set of equilibria for the fast flow of \eqref{norm12lay}; the fast fibres are parallel to the $x$-direction. (b) The supercritical manifold $\mathcal{M}_2$ as the set of equilibria for the intermediate flow of \eqref{intermediate}; the intermediate fibres are confined to $\mathcal{M}_2$ and evolve on planes with $z$ constant.}
	\figlab{em1}
\end{figure}

Finally, we consider the reduced problem on $ \mathcal{M}_1$, as given by \eqref{norm12red}, 
	with $\delta$ sufficiently small; Equation~\eqref{norm12red} is then singularly perturbed with respect to the small parameter $\delta$, written in the intermediate time $t$. To classify the folded singularities $q^\mp$ of $\mathcal{M}_1$, we project the flow of \eqref{norm12red} onto $\mathcal{M}_1$ \cite{szmolyan2001canards}:
recalling that $\mathcal{M}_1$ is defined by $ f(x,y) =0 $, we can apply the chain rule to find
\begin{align*}
-f_x\dot x = f_y\dot y,
\end{align*}
where $ f_x = 2f_2x+3f_3x^2 =F'(x)$ and $ f_y =-1 $, from \eqref{norm12red}. We therefore obtain
\begin{subequations}\eqlab{classy}
	\begin{align}
	-F'(x)\dot x &=-\alpha x-\beta F(x){+z}, \\
	\dot z &=\delta\lp\mu +\phi\lp x, F(x),z\rp\rp
	\end{align}
\end{subequations}
or
\begin{subequations}\eqlab{classy-des}
	\begin{align}
	\dot x & =-\alpha x-\beta F(x){+z}, \\
	\dot z &=-\delta F'(x) \lp\mu +\phi\lp x, F(x),z\rp \rp
	\end{align}
\end{subequations}
after a rescaling of time which introduces a factor of $ -F'(x) $ on the right-hand sides in \eqref{classy}, reversing the direction of the flow on $ \mathcal{S}^r $. The folded singularities of Equation~\eqref{normal} then correspond to equilibria for \eqref{classy-des}; specifically, for $\delta$ positive, the folded singularities $q^\mp$ are \emph{folded nodes} \cite{szmolyan2001canards,letson2017analysis}. Their strong and weak stable manifolds define ``funnel regions'' on the corresponding sheets $\mathcal{S}^{a^\mp}$, which essentially determine the basins of attraction to $q^\mp$ on $ \mathcal{S}^{a^\mp} $. Here and in the following, we focus on the flow of Equation~\eqref{normal} in the vicinity of the fold line $\mathcal{L}^-$; with regard to the strong stable manifold of the folded node $q^-$, we hence have the following result:
\begin{lemma}
Let 
\begin{gather}
\mathcal{G}\lp x_0,x_1;z_0;\mu\rp = \int_{x_0}^{x_1} \frac{ F'(\sigma)\lp \mu+\phi\lp \sigma,F(\sigma),z_0\rp\rp}{\alpha\sigma+\beta F(\sigma)-z_0} \tn{d}\sigma, 
\eqlab{GiF}
\end{gather}
where $F$ is defined as in \eqref{Fx}.
Then, for $\delta$ sufficiently small, the strong stable manifold of the origin for Equation~\eqref{classy-des} can be written as the graph
\begin{gather}
z = \delta \mathcal G(0,x;0,\mu) + \mathcal{O}(\delta^2)\qquad\text{for }x\in I,
\end{gather}
where $I$ is an appropriately defined, fixed interval about $x=0$.
\lemmalab{strong}
\end{lemma}
\begin{proof}
Given a trajectory of \eqref{classy-des} with initial condition $ (x_0, y_0,z_0) $ on $ \mathcal{S}^{a^\mp} $, i.e., with $y_0=F(x_0)$, let $ s $ denote the displacement in the $ x $-direction of that trajectory under the corresponding flow. Then, in a first approximation, the displacement in the $ z $-direction is given by $ \delta\mathcal{G}\lp x_0,x_0+s;z_0;\mu\rp $, where $ \mathcal{G} $ is defined as in \eqref{GiF}; see \cite{krupa2008mixed} for details. The result is obtained by setting $x_0=0=z_0$ in the resulting expression, which corresponds to the unique trajectory of \eqref{classy-des} that passes through the origin.
\end{proof}
An analogous representation can be obtained for the strong stable manifold of the folded node $ q^+ $. From the above, we conclude in particular that the funnels of the folded singularities $q^\mp$ are ``stretched'' as $\delta$ decreases. In the limit of $\delta = 0$, $q^\mp$ are (degenerate) \emph{folded saddle-nodes}; see again \cite{szmolyan2001canards,letson2017analysis} for details. For future reference, we note that the associated strong manifolds (``strong canards") correspond to the unique intermediate fibres on $\mathcal{S}^{a^\mp}$ that cross $q^\mp$, respectively, while the corresponding weak manifolds (``weak canards") can be locally approximated by the supercritical manifold $ \mathcal{M}_{2}$ which is introduced in the following subsection.

\subsection{The supercritical manifold $\mathcal{M}_2$}

We can view the differential-algebraic systems in \eqref{norm12red} and \eqref{classy-des} as slow-fast vector fields on $\mathcal{M}_1$. The layer problem corresponding to \eqref{norm12red} therefore reads
\begin{subequations}\eqlab{intermsing}
	\begin{align}
	0 &=-y+F(x), \eqlab{intermsing-a}\\
	\dot{y} &=\alpha x+\beta y{-z}, \eqlab{intermsing-b} \\
	\dot{z} &=0\eqlab{intermsing-c}
	\end{align}
\end{subequations}
or
\begin{subequations}\eqlab{intermediate}
	\begin{align}
	0 &=-y+F(x), \\
	-F'(x)\dot x &=-\alpha x-\beta F(x){+z}, \\
	\dot{z} &=0;
	\end{align}
\end{subequations}
we will refer to the above as the \textit{intermediate flow}, and to the corresponding trajectories as the \textit{intermediate fibres}; see panel (b) of \figref{em1}. We emphasise that the intermediate flow is not defined on the fold lines $\mc{L}^\mp$, whereon $F'(x)=0$.

Rewriting Equation~\eqref{normal} in the slow time $s=\delta t$, we have
\begin{subequations}\eqlab{M1slowtime}
	\begin{align}
	\varepsilon \delta{x}' &= -y +F(x), \\
	\delta{y}' &= \alpha x+\beta y{-z}, \\
	{z}' &=  \mu +\phi\lp x, y,z\rp; 
	\end{align}
\end{subequations}
the reduced system that is obtained from \eqref{M1slowtime} is given by
\begin{subequations}\eqlab{slow}
	\begin{align}
	0 &=-y{+F(x)}, \\
	0 &=\alpha x+\beta F(x){-z}, \\
	z' &=\mu +\phi\lp x, F(x),z\rp,
	\end{align}
\end{subequations}
which we will refer to as the \textit{slow flow} of Equation~\eqref{normal}.
Away from $\mathcal{L}^\mp$, the supercritical manifold $\mathcal{M}_2$ is the set of equilibria for \eqref{intermediate}, 
and is given by
\begin{align}
\begin{aligned}
\mathcal{M}_2 :&= \lb\lp x,y,z\rp \in\mb{R}^3~\big\lvert ~ f(x,y) =0=g(x,y,z)\rb 
= \lb\lp x,y,z\rp \in\mathcal{M}_1~\big\lvert ~ {z=G(x)}\rb,
\end{aligned}\eqlab{M2}
\end{align}
where we define
\begin{align}
G(x) = \alpha x + \beta F(x) \eqlab{Gx};
\end{align}
in that notation, the coordinates of the folded singularities in \eqref{foldsing} can be expressed as $y_{q}^\mp = F(x_{q}^\mp)$ and $z_{q}^\mp = G(x_{q}^\mp)$.
The manifold $\mathcal{M}_2$ can be written as the union $ \mathcal{M}_2 = \mathcal{Z}\cup \mathcal{F}_{\mathcal{M}_2}$, where
\begin{align}
\mathcal{Z}=\lb\lp x,y,z\rp\in\mathcal{M}_2 ~\Big|~\frac{{dg}}{{d}x}\lp x,F(x),G(x)\rp \neq0\rb
\end{align}
is normally hyperbolic 
and the set
\begin{align}
\begin{aligned}
\mathcal{F}_{\mathcal{M}_2} := \lb \lp x,y,z\rp\in\mathcal{M}_2 ~\Big|~\frac{{dg}}{{d}x}\lp x,F(x),G(x)\rp =0\rb
=\lb\lp x,y,z\rp\in\mathcal{M}_2 ~\big\lvert~\alpha + 2\beta f_2x+3\beta f_3x^2 =0\rb
\end{aligned}\eqlab{nhZ}
\end{align}
is degenerate. Equation~\eqref{nhZ} yields $\mathcal{F}_{\mathcal{M}_2} = \lb p^-,p^+\rb$, with
\begin{align}
p^\mp=\lb\lp x,y,z\rp\in\mathcal{M}_2 ~\big\lvert~ x = x^{\mp}_p\rb, \eqlab{folp}
\end{align}	
where
\begin{align}\eqlab{Zfolds}
x^{\mp}_p=\frac{-\beta f_2{\pm}\sqrt{\beta^2f_2^2-3\alpha\beta f_3}}{3\beta f_3},\quad y^{\mp}_p = F\lp x^{\mp}_p\rp,\quad\text{and}\quad z^{\mp}_p = G\lp x^{\mp}_p\rp.
\end{align}
The points $ p^\mp $ are called the fold points of $ \mathcal{M}_2 $. Equation~\eqref{Zfolds} immediately implies
\begin{proposition}
	The manifold $\mathcal{M}_2$ admits
	\begin{enumerate}
		\item exactly two fold points if and only if $\beta^2f_2^2-3\alpha\beta f_3>0$; 
		\item exactly one fold point if and only if $\beta^2f_2^2-3\alpha\beta f_3 = 0$; and
		\item no fold points if and only if $\beta^2f_2^2-3\alpha\beta f_3<0$. 
	\end{enumerate}
	\proplab{zfoldsprop}
\end{proposition}

\begin{remark}
Under the conditions stated in \propref{zfoldsprop}, the fold points $p^\mp$ of $\mathcal{M}_2$ are ``inherited'' from the fold lines $ \mathcal{L}^\mp $ of $\mathcal{M}_1$, in the sense that $G(x)$ in \eqref{M2} is a cubic polynomial because $F(x)$ in \eqref{M1} is.
\end{remark}

We note that a necessary (but not sufficient) condition for $\mathcal{M}_2$ to have two fold points is the requirement that $\beta\neq0$. If $\mathcal{M}_2$ admits two fold points, then the normally hyperbolic portion $\mathcal{Z}$ of $\mathcal{M}_2$ consists of three branches:
$\mathcal{Z}=\mathcal{Z}^-\cup \mathcal{Z}^{0}\cup\mathcal{Z}^{+}$,
where
\begin{gather}
\begin{gathered}
\mathcal{Z}^{-}=\lb\lp x,y,z\rp\in\mathcal{Z} ~\lvert~ x<x_p^{-}\rb,\quad \mathcal{Z}^{0}=\lb\lp x,y,z\rp\in\mathcal{Z} ~\lvert~ x_p^{-}<x<x_p^{+}\rb,\quad\text{and} \\
\mathcal{Z}^{+}=\lb\lp x,y,z\rp\in\mathcal{Z} ~\lvert~ x>x_p^{+}\rb.
\end{gathered}\eqlab{zbranches}
\end{gather}

A rescaling of time in \eqref{intermediate} by a factor of $-F'(x)$, as was done in \eqref{classy-des}, reverses the orientation on $\mc{S}^r$, whereas it preserves it on $\mc{S}^{a^\mp}$. If the supercritical manifold $\mathcal{M}_2$ admits two fold points, i.e., if $\beta^2f_2^2-3\alpha\beta f_3>0$ by \propref{zfoldsprop}, then the portion of the middle branch $\mathcal{Z}^0\cap \mc{S}^r$ of $\mathcal{M}_2$ in \eqref{zbranches} is attracting, respectively repelling, under the flow of the desingularised Equation~\eqref{intermediate} for  $\beta<0$, respectively $\beta>0$. That is, the stability properties of $\mc{Z}$ within $\mc{S}$ -- i.e., within $\mc{M}_1$ and away from $\mc{L}^\mp$ -- are determined from the scalar equation $\dot{x} =-\alpha x-\beta F(x){+z}$,  with the stability of $\mc{Z}^0$ being reversed on $\mc{S}^r$, cf. \figref{m2bif}. It follows that $\mc{Z}^0$ could potentially be separated into attracting portions in $\mc{S}^r$ and repelling ones in $\mc{S}^{a^\mp}$, cf. \figref{m2bif}(a), or vice versa, see \figref{m2bif}(d). A similar argument applies to the outer branches $\mc{Z}^\mp$, as seen in panels (c) and (f) of \figref{m2bif}.

\begin{remark}
We remark that the above discussion of the stability of $\mc{Z}$ in the double singular limit of $\varepsilon=0=\delta$ is alternative to the approach outlined in \cite{cardin2017fenichel}, where $x$ is expressed as a function of $y$ in \eqref{intermsing-b} via the algebraic constraint in \eqref{intermsing-a}, as well as to that in \cite{letson2017analysis}, where only the stability of the partially perturbed counterpart of $\mc{Z}$ is investigated; see Appendix~\ref{mechs} for an extension of the latter within the framework of Equation~\eqref{normal}.
\end{remark}

\begin{remark}
By the above, the folded singularities $q^\mp$ of $\mathcal{M}_1$ are located at the intersections between $\mathcal{M}_2$ and $\mathcal{L}^\mp$. In the double singular limit of $ \varepsilon=0=\delta$, the points $q^\mp$ coincide with the folded singularities of $ \mathcal{M}_1 $ for $\varepsilon =0$ and $\delta=\mathcal{O}(1)$, i.e., in the two-timescale limit, which stems from the fact that the fast and intermediate Equations~\eqref{normal-a} and \eqref{normal-b} do not depend on $ \delta$ in our case.
\end{remark}

\subsection{Relative geometry}

In this subsection, we describe the position of the folded singularities $q^\mp$ of $\mathcal{M}_1$ relative to each other, as well as of the fold points $p^\mp$ of $ \mathcal{M}_2 $ -- assuming that a pair of such points exists -- relative to the fold lines $\mathcal{L}^\mp$.

\begin{proposition}
	Assume that $\mathcal{M}_2$ admits two fold points, i.e., that $\beta^2f_2^2-3\alpha\beta f_3>0$, by \propref{zfoldsprop}. 
	\begin{enumerate}
		\item If $\alpha\beta<0$, then both fold points of  $\mathcal{M}_2$ lie on $\mathcal{S}^r$; 
		\item if $\alpha\beta>0$, then one fold point of $\mathcal{M}_2$ lies on $\mathcal{S}^{a^-}$, while the other fold point lies on $\mathcal{S}^{a^+}$; and
		\item if $\alpha=0$, then one fold point of $\mathcal{M}_2$ lies on $\mathcal{L}^-$, while the other fold point lies on $\mathcal{L}^+$.
	\end{enumerate}
	\proplab{relfold}
\end{proposition}
\begin{proof}
	The result follows from a comparison of the values of $x^{\mp}$ in \eqref{Zfolds} with the $x$-coordinates of $\mathcal{L}^\mp$ in the three cases where $\alpha\beta<0$, $\alpha\beta>0$, and $\alpha=0$, respectively.
\end{proof}

\propref{relfold} is again summarised in \figref{m2bif}. We remark that the symmetry described in \propref{relfold} breaks down when $\mathcal{O}(x^2)$-terms are included in the intermediate Equation~\eqref{normal-b}; see \cite{desroches2018spike} for an example. If $\beta=0$, then the projection of the critical manifold $\mathcal{M}_2$ onto the $(x,z)$-plane is a straight line. That case has been studied in \cite{krupa2008mixed,de2014three,de2016sector}; recall also Equation~\eqref{prototypical}.

\begin{figure}[ht!]
	\centering
	\begin{subfigure}[b]{0.3\textwidth}
		\centering
		\includegraphics[scale = 0.3]{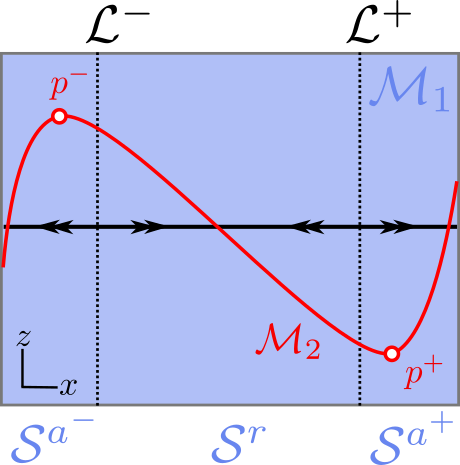}
		\caption{$\beta<0$, $\alpha<0$.}
	\end{subfigure}
	~
	\begin{subfigure}[b]{0.3\textwidth}
		\centering
		\includegraphics[scale = 0.3]{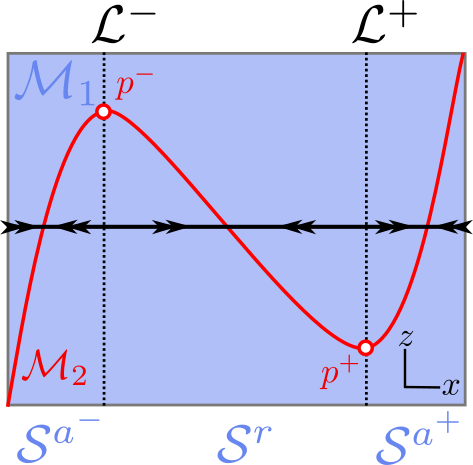}
		\caption{$\beta<0$, $\alpha=0$.}
	\end{subfigure}
	~
	\begin{subfigure}[b]{0.3\textwidth}
		\centering
		\includegraphics[scale = 0.3]{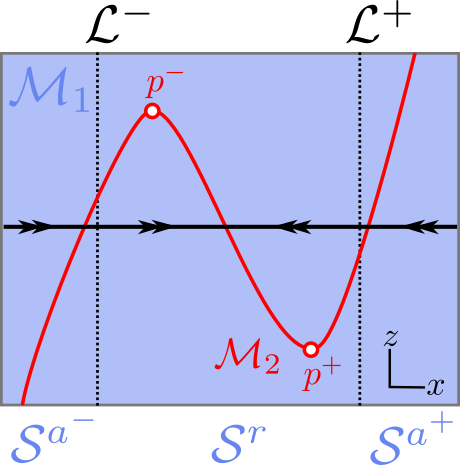}
		\caption{$\beta<0$, $\alpha>0$.}
	\end{subfigure} \\
	\vspace*{0.2in}
	\begin{subfigure}[b]{0.3\textwidth}
		\centering
		\includegraphics[scale = 0.3]{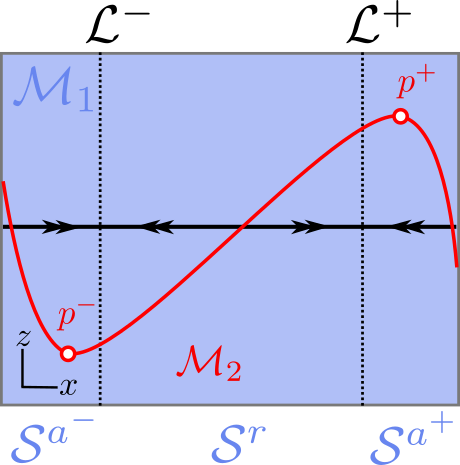}
		\caption{$\beta>0$, $\alpha<0$.}
	\end{subfigure}
	~
	\begin{subfigure}[b]{0.3\textwidth}
		\centering
		\includegraphics[scale = 0.3]{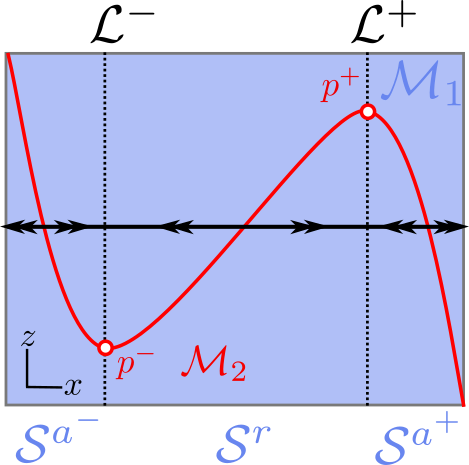}
		\caption{$\beta>0$, $\alpha=0$.}
	\end{subfigure}
	~
	\begin{subfigure}[b]{0.3\textwidth}
		\centering
		\includegraphics[scale = 0.3]{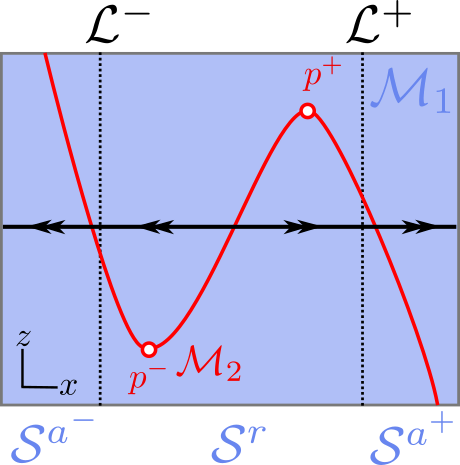}
		\caption{$\beta>0$, $\alpha>0$.}
	\end{subfigure}
	\caption{Projection of the supercritical manifold $\mathcal{M}_2$ and of the fold lines $\mathcal{L}^\mp$ of the critical manifold $\mathcal{M}_1$ onto the $(x,z)$-plane: in dependence on the parameters $\alpha$ and $\beta$, the pair of fold points $ p^{\mp} $ of $\mathcal{M}_2$ lies either on $\mathcal{S}^r$ (panels (c) and (f)), on $\mathcal{S}^{a^\mp}$ (panels (a) and (d)), or on $\mathcal{L}^\mp$ (panels (b) and (e)).}
	\figlab{m2bif}
\end{figure}

We now turn our attention to the location of the folded singularities of $\mathcal{M}_1$ relative to each other and with respect to the fast and intermediate fibres defined previously; recall \figref{em1}. We first define planes that contain the folded singularities and that are perpendicular to the fold lines $\mathcal{L}^-$ and $\mathcal{L}^+$, as follows. 
\begin{definition}
	Denote by $\mathcal{P}^\mp$ the planes 
	$\mathcal{P}^\mp = \lb (x,y,z)\in\mb{R}^3~\lvert~ z = z_q^\mp\rb$,
	where $z_q^\mp$ are the $z$-coordinates of the folded singularities $q^\mp$ of $\mathcal{M}_1$ on $\mathcal{L}^\mp$, respectively. We will refer to $\mathcal{P}^\mp$ as \emph{normal planes} in the following.
\end{definition}

\begin{definition}
	The folded singularities $q^\mp$ of $\mathcal{M}_1$ are said to be 
	\begin{enumerate}
		\item aligned if $\mathcal{P}^-\equiv \mathcal{P}^+$;
		\item connected if they are not aligned and if $\mathcal{P}^\mp\cap\mathcal{Z}^\pm \neq \emptyset$; or 
		\item remote if they are neither aligned nor connected, i.e., if $\mathcal{P}^-\not\equiv \mathcal{P}^+$ and $\mathcal{P}^\mp\cap\mathcal{Z}^\pm = \emptyset$. 
	\end{enumerate}
	\defnlab{singalign}
\end{definition}

In dependence of the parameters $\alpha$, $\beta$, $f_2$, and $f_3$ in Equation~\eqref{normal}, we have the following result on the position of $q^-$ and $q^+$ relative to each other:
\begin{proposition}
Recall the classification in \defnref{singalign} above.
	\begin{enumerate}	
		\item For $\alpha\beta<0$, the folded singularities $q^\mp$ of $\mathcal{M}_1$ are aligned if 
		$\frac{\alpha}{\beta}=\frac{2f_2^2}{9f_3}$, connected if 
		$\frac{\alpha}{\beta}>\frac{2f_2^2}{9f_3}$, and remote if
		$\frac{\alpha}{\beta} < \frac{2f_2^2}{9f_3}$.
		\item For $\alpha\beta\geq0$ with $ \beta\neq0 $, the folded singularities $q^\mp$ are connected.
		\item For $ \beta = 0 $ with $ \alpha\neq0 $, the folded singularities $q^\mp$ are remote.
	\end{enumerate}
	\proplab{relatives}
\end{proposition}
\begin{proof}
	The statements follow from Equation~\eqref{foldsing} and the properties of $ G(x) $ in \eqref{Gx}; see panels (c) and (d) of \figref{m2bif} for cases corresponding to the first statement, and panels (a), (b), (e), and (f) for cases corresponding to the second statement.
\end{proof}

In what is to come, we will restrict our attention to the case that is illustrated in panel (c) of \figref{m2bif}:
\begin{asu}
In the following, we assume that $\alpha>0$ and $\beta<0$ in Equation~\eqref{normal}. 
\asulab{albeta}
\end{asu}
\asuref{albeta} is made for three reasons. First, it is consistent with the Koper model, Equation~\eqref{koper1}, after transformation to the prototypical Equation~\eqref{normal}. (In particular, it follows that the scenarios illustrated in panels (b) and (e) of \figref{m2bif} cannot be realised in \eqref{koper1}.) Second, given $\beta\neq 0$, remote singularities can only be present when $\alpha\beta<0$. Third, given \asuref{albeta}, the outer branches $ \mc{Z}^\mp\cap\mathcal{S}^{a\mp}$ of $\mc{Z}$ are attracting, while the middle branch is repelling, which allows for the construction of closed singular periodic orbits (``cycles") which will serve as templates for the corresponding MMO trajectories, as will become apparent in the following subsection. In particular, since $\mc{Z}^0$ is entirely contained in $\mc{S}^r$, we will write 
$$\mathcal{Z}=\mathcal{Z}^{-}\cup\mathcal{Z}^{r}\cup\mathcal{Z}^{+}$$ in the following.

\begin{figure}[ht!]
	\centering
	\begin{subfigure}[b]{0.3\textwidth}
		\centering
		\includegraphics[scale = 0.28]{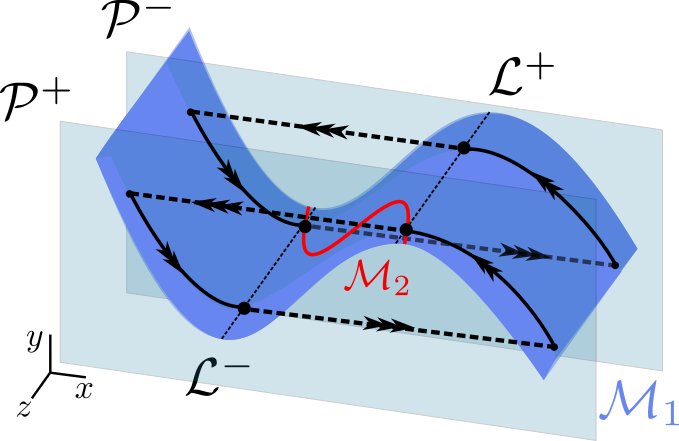}
	\end{subfigure}
	~~
	\begin{subfigure}[b]{0.3\textwidth}
		\centering
		\includegraphics[scale = 0.28]{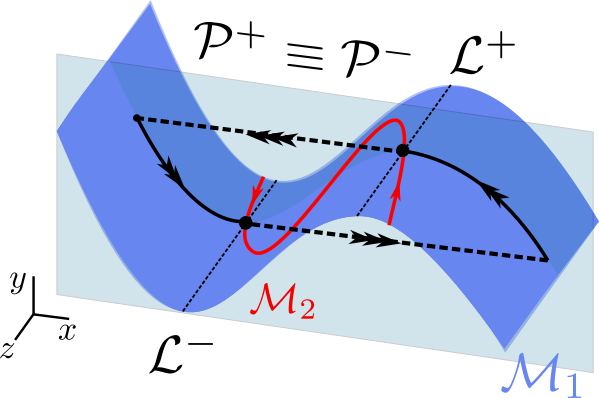}
	\end{subfigure}
	~~
	\begin{subfigure}[b]{0.3\textwidth}
		\centering
		\includegraphics[scale = 0.28]{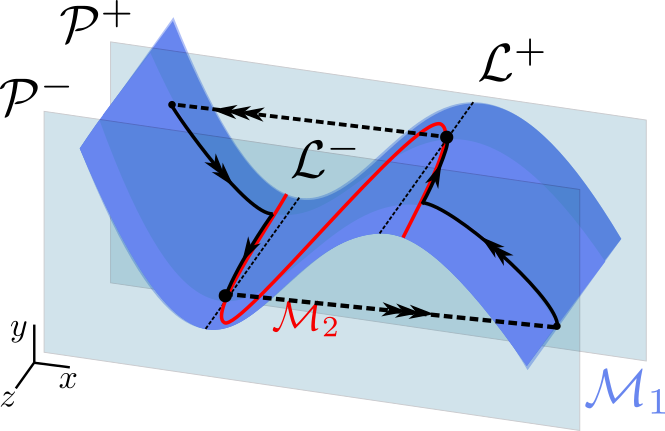}
	\end{subfigure}
	\\
	\vskip10pt
	\begin{subfigure}[b]{0.3\textwidth}
		\centering
		\includegraphics[scale = 0.3]{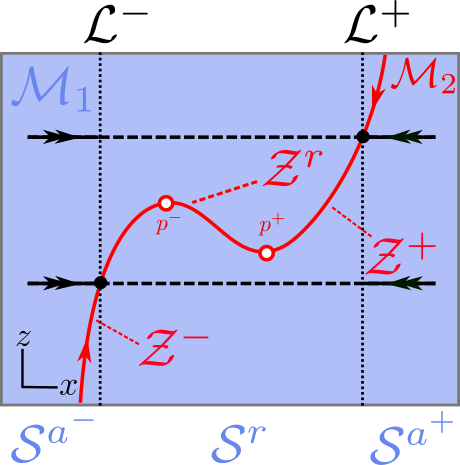}
		\caption{Remote singularities.}
	\end{subfigure}
	~
	\begin{subfigure}[b]{0.3\textwidth}
		\centering
		\includegraphics[scale = 0.3]{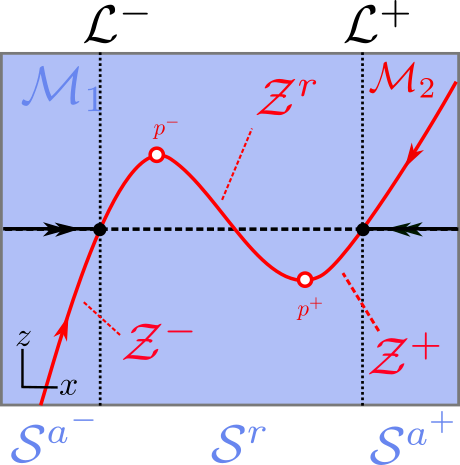}
		\caption{Aligned singularities.}
	\end{subfigure}
	~
	\begin{subfigure}[b]{0.3\textwidth}
		\centering
		\includegraphics[scale = 0.3]{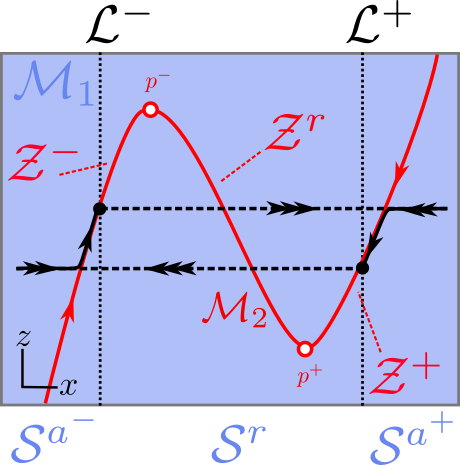}
		\caption{Connected singularities.}
	\end{subfigure}
	\caption{Relative geometry of the folded singularities $q^\mp$ of $\mathcal{M}_1$ according to \defnref{singalign} (top row); bifurcation of the resulting singular cycles, as described in \propref{double} (bottom row).}
	\figlab{singcycle}
\end{figure}

\subsection{Singular cycles}

We now consider the reduced flow on $\mathcal{M}_2$. We impose the following assumption on the function $\phi(x,y,z)$ in the slow Equation~\eqref{normal-c}:
\begin{asu}
The function $\phi(x,y,z)$ in Equation~\eqref{normal-c} is such that $\phi(x_q^-,y_q^-,z_q^-)=0$, $\phi(x,F(x),G(x))>0$ for $x<x_q^-$, $\phi(x_q^+,y_q^+,z_q^+)\leq0$, and $\phi(x,F(x),G(x))<0$ for $ x>x_q^+$.
\asulab{redflow}
\end{asu}


\asuref{redflow} ensures that the reduced flow on the portions $\mc{Z}^{\mp}$ of $\mathcal{M}_2$ is directed towards the folded singularities $q^\mp$ of $\mathcal{M}_1$. We emphasise that the properties of that flow therefore crucially depend on $ \mu $: in particular, we have that for $ \mu = 0 $, a true global equilibrium of Equation~\eqref{normal} coincides with $q^-$; see \secref{singpert1}.

\asuref{albeta} and \asuref{redflow} together imply the existence of singular cycles in Equation~\eqref{normal}, the properties of which are determined by the relative position of the folded singularities $q^\mp$ of $\mathcal{M}_1$, as classified in \propref{relatives}. (Clearly, the choice of $\phi(x,y,z)$ in \eqref{normal-c} does not affect that classification; correspondingly, we do not specify it here.) These cycles are defined as the concatenation of singular orbits for the corresponding limiting systems in \eqref{norm12lay}, \eqref{intermediate}, and \eqref{slow}, respectively.
\begin{proposition}
Assume that \asuref{albeta} and \asuref{redflow} hold. 
\begin{enumerate}
\item If the folded singularities $q^\mp$ of $\mathcal{M}_1$ are remote, then there exist a singular cycle evolving on $\mathcal{P}^-$, a singular cycle evolving on $\mathcal{P}^+$, and a family of singular cycles in between; each of the cycles in that family evolves on a plane parallel to $\mathcal{P}^{\mp}$ that lies between $\mathcal{P}^{-}$ and $\mathcal{P}^{+}$. These cycles are ``two-scale", in the sense that the singular dynamics on them alternates between the fast and the intermediate timescale (on $ \mathcal{M}_1 \backslash  \mathcal{M}_2 $). 
\item If $q^\mp$ are aligned, then there there exists exactly one singular cycle that evolves on the plane $\mathcal{P}:=\mathcal{P}^-\equiv \mathcal{P}^+$. This cycle is ``two-scale", in the sense that the singular dynamics on it alternates between the fast and the intermediate timescale (on $ \mathcal{M}_1 \backslash \mathcal{M}_2 $).  
\item If $q^\mp$ are connected, then there exists exactly one singular cycle that evolves on a subset of $\mathcal{P}^-\cup\mathcal{Z}^{-}\cup\mathcal{P}^+\cup\mathcal{Z}^{+}$. This cycle is ``three-scale", in the sense that the singular dynamics on it alternates between the fast, the intermediate (on $ \mathcal{M}_1 \backslash \mathcal{M}_2 $), and the slow timescale (on $ \mathcal{M}_2 $). 
\end{enumerate}
\proplab{double}
\end{proposition}

\defnref{singalign} and \propref{double} are summarised in \figref{singcycle}, where we recall that the fast, intermediate, and slow dynamics are given by the limiting systems in \eqref{norm12lay}, \eqref{intermediate}, and \eqref{slow}, respectively. 

\section{Singular perturbation}\seclab{singpert1}

In this section, we discuss the correspondence between the families of singular cycles constructed in \propref{double} and the MMO trajectories which perturb from those cycles for $\varepsilon$ and $\delta$ positive, but sufficiently small, in Equation~\eqref{normal}. In the process, we give a qualitative characterisation of the resulting mixed-mode dynamics in dependence of system parameters. 

By standard GSPT \cite{fenichel1979geometric,krupa2001extending,cardin2017fenichel}, we obtain slow manifolds $ \mathcal{S}_{\varepsilon\delta}^{a,r} $ as perturbations of $ \mathcal{S}^{a,r} $ away from the fold lines $\mathcal{L}^\mp$, for $ \varepsilon,\delta>0 $ sufficiently small. The dynamics on these locally invariant manifolds is itself slow-fast with respect to the singular perturbation parameter $\delta$, which implies the existence of super-slow locally invariant manifolds $\mathcal{Z}_{\varepsilon\delta}^{\mp,r}$ as perturbations of $\mathcal{Z}^{\mp,r}$ away from the fold points $p^\mp$. In particular, the manifolds $ \mathcal{Z}_{\varepsilon\delta}^{\mp} $ locally approximate the weak canards of the folded singularities $ q^\mp $, respectively \cite{letson2017analysis}; recall \eqref{classy-des}. 

First, we remark on the transition between steady-state behaviour and oscillatory dynamics in Equation~\eqref{normal} in dependence of $\mu$. For $ \varepsilon=0=\delta $, true equilibria of \eqref{normal} cross the folded singularities $ q^\mp $ at 
\begin{align}
\mu_{q}^-=-\phi(x_{q}^-, y_{q}^-, z_{q}^-)\quad\text{and}\quad
\mu_{q}^+=-\phi(x_{q}^+, y_{q}^+, z_{q}^+), \eqlab{muqmp}
\end{align}
where we recall that $(x_{q}^-, y_{q}^-, z_{q}^-)=(0,0,0)$ and, hence, that $ 0=\mu_{q}^-<\mu_{q}^+$, by \asuref{redflow}.

By \cite{guckenheimer2008singular,krupa2008mixed,desroches2012mixed}, Hopf bifurcations occur $\mc{O}(\varepsilon,\delta)$-close to $\mu_q^{\mp}$ for $ \varepsilon,\delta >0$ sufficiently small in Equation~\eqref{normal}; these generate small-amplitude limit cycles in the vicinity of the folded singularities $q^\mp$ which correspond to MMOs with signature $0^k$. The mixed-mode dynamics that may arise due to those Hopf bifurcations does not seem to be well-understood in the three-scale regime considered here \cite{desroches2012mixed}. However, MMO trajectories that contain both SAO and LAO segments have been shown to emerge either in a ``slow passage through a canard explosion" that occurs $\mc{O}(\varepsilon,\delta)$-away from the corresponding Hopf points \cite{krupa2008mixed,letson2017analysis} or via a delayed Hopf-type phenomenon on $\mathcal{Z}_{\varepsilon\delta}^\mp$ \cite{de2016sector,letson2017analysis}; the underlying mechanisms are briefly addressed in Appendix~\ref{mechs} in the context of Equation~\eqref{normal}. Crucially, for fixed $\mu\in (\mu_{q}^-,\mu_{q}^+)$, one can find $\varepsilon,\delta>0$ sufficiently small such that \eqref{normal} features global mixed-mode dynamics, which we will study in dependence of the parameters $f_2$, $f_3$, $\alpha$, and $\beta$ therein. For future reference, we will write
\begin{align}
    M:= (\mu_{q}^-,\mu_{q}^+)\quad\text{and}\quad\overline{M}:= \big[\mu_{q}^-,\mu_{q}^+\big].
\end{align}

For $\varepsilon$ and $\delta$ positive and sufficiently small in \eqref{normal}, trajectories can hence be composed from components that evolve close to fast, intermediate, and slow segments of the corresponding singular cycles introduced in \secref{singgeom}. In a first approximation, where the fast and intermediate segments are approximated as straight lines -- the latter in the 
$ (x,z) $-plane -- trajectories are attracted to the vicinities of both folded singularities $q^\mp$ if these are aligned or connected, whereas they tend to either $q^-$ or $q^+$ if the singularities are remote, as can be seen from \figref{singcycle}. (In~\secref{singgeom}, we showed that the funnels of the folded nodes $q^\mp$ for Equation~\eqref{normal} expand with decreasing $ \delta $; in the three-timescale limit as $ \delta $ approaches zero, these funnels can be viewed as having been ``stretched" in one direction.) From the well-established theory of two-timescale singular perturbations, it is known that SAOs arise in the passage past folded singularities under the perturbed flow, see \cite{desroches2012mixed, de2016sector,wechselberger2005existence}; the underlying local two-timescale mechanisms are well understood. Again, we discuss their three-timescale analogues in \secref{local} below. In particular, we conclude that SAOs are observed ``above" or ``below", in the language of \figref{fareys}, depending on which folded singularity of Equation~\eqref{normal} trajectories are attracted to; double epochs of SAOs can occur when trajectories are attracted to both folded singularities $q^\mp$. The mixed-mode dynamics of Equation~\eqref{normal} can hence naturally be classified according to whether the folded singularities $q^\mp$ are aligned, connected, or remote; cf.~\figref{muplane}. 

\begin{figure}[ht!]
	\centering
	\begin{subfigure}[b]{0.45\textwidth}
		\centering
		\includegraphics[scale = 0.27]{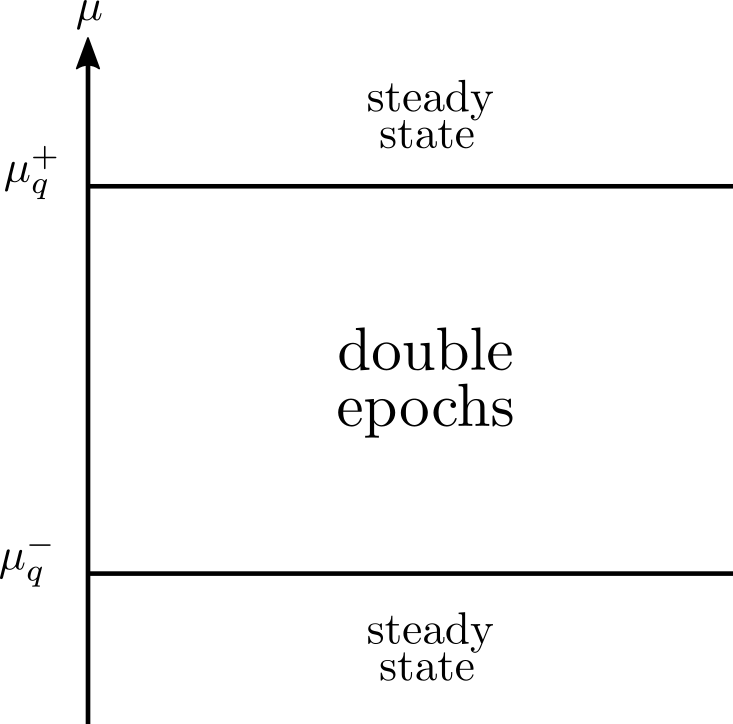}
		\caption{Aligned or connected singularities $\big(\frac{\alpha}{\beta}\geq\frac{2f_2^2}{9f_3}\big)$.}
	\end{subfigure}
	~
	\begin{subfigure}[b]{0.45\textwidth}
		\centering
		\includegraphics[scale = 0.27]{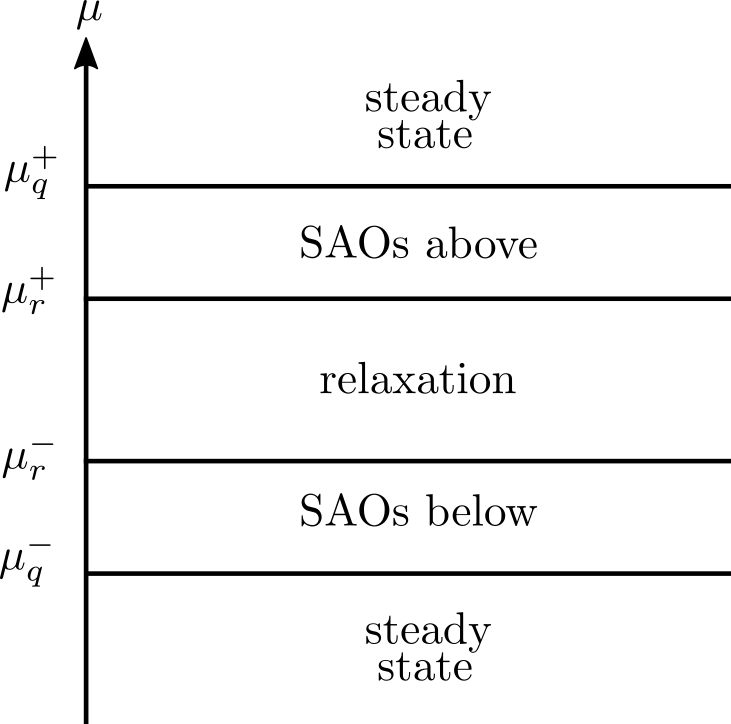}
		\caption{Remote singularities $\big(\frac{\alpha}{\beta}<\frac{2f_2^2}{9f_3}\big)$.}
	\end{subfigure}
	\caption{Dynamics of Equation~\eqref{normal} with $ f_2$, $f_3$, $\alpha$, and $\beta$ fixed and $\varepsilon,\delta>0$ sufficiently small: the $\mu$-values $\mu_{q}^\mp$ distinguish between oscillatory dynamics and steady-state behaviour. (a) When $\frac{\alpha}{\beta}\geq\frac{2f_2^2}{9f_3}$, the singular geometry of \eqref{normal} is such that double epochs of perturbed slow dynamics occur for $ \varepsilon, \delta >0$ sufficiently small. (b) When $\frac{\alpha}{\beta}<\frac{2f_2^2}{9f_3}$, there exist two values $ \mu_r^\mp $ which separate MMO trajectories with single epochs of SAOs from relaxation oscillation, in dependence of the properties of $ \phi(x,y,z) $ in \eqref{normal-c}.}
	\figlab{muplane}
\end{figure}



\subsection{Aligned or connected singularities}
When the folded singularities of $\mathcal{M}_1$ are aligned or connected, ``double epochs" of slow dynamics are observed in \eqref{normal} for $ \varepsilon,\delta>0 $ sufficiently small and $ \mu \in (\mu_{q}^-,\mu_{q}^+) $; see \figref{muplane} and \figref{conp}. 

\begin{theorem}
	Assume that \asuref{albeta} and \asuref{redflow} hold, that the folded singularities of Equation~\eqref{normal} are aligned or connected in the sense of \defnref{singalign}, i.e., that $\frac{\alpha}{\beta}\geq\frac{2f_2^2}{9f_3}$ in \eqref{normal}, and fix $\mu\in M$. Then, there exist $ \varepsilon_0,\delta_0 >0$ sufficiently small  such that \eqref{normal} features MMOs with double epochs of perturbed slow dynamics for all $ \lp \varepsilon,\delta\rp \in (0,\varepsilon_0)\times(0,\delta_0)$.
	\thmlab{conne}
\end{theorem}
\begin{proof}
We need to show that, under the stated assumptions, Equation~\eqref{normal} admits oscillatory trajectories that consist of repeat sequences of fast, intermediate, and slow segments near $\mc{S}^{a^-}$, followed by fast, intermediate, and slow segments near $\mc{S}^{a^+}$; see \figref{conp}. To that end, we define the sections
\begin{align*}
\Delta^+ &= \lb \lp x,y,z\rp \in \mb{R}^3\ \bigg\lvert \ x = \frac{x_{q}^++x_{q}^-}{2}, \ y > \frac{y_{q}^++y_{q}^-}{2},\ \text{and}\ z > \frac{z_{q}^++z_{q}^-}{2}\rb, \\
\Delta^- &= \lb \lp x,y,z\rp \in \mb{R}^3\ \bigg\lvert \ x = \frac{x_{q}^++x_{q}^-}{2}, \ y < \frac{y_{q}^++y_{q}^-}{2},\ \text{and}\ z < \frac{z_{q}^++z_{q}^-}{2}\rb, \\
\Sigma^- &= \lb \lp x,y,z\rp \in \mb{R}^3\ \lvert \ x<x_{q}^-\ \text{and}\ y = y_{q}^-+\rho_1, \ \text{with}\ \rho_1>0 \textnormal{ small}\rb,\quad\text{and} \\
\Sigma^+ &= \lb \lp x,y,z\rp \in \mb{R}^3\ \lvert \ x>x_{q}^-\ \text{and}\ y = y_{q}^+- \rho_2, \ \text{with}\ \rho_2>0 \textnormal{ small}\rb, 
\end{align*}
as well as the corresponding transition maps
\begin{align*}
\pi_{\rm out}^- : \Sigma^- \to \Delta^-, \quad
\pi_{\rm in}^+ : \Delta^- \to \Sigma^+,\quad 
\pi_{\rm out}^+ : \Sigma^+ \to \Delta^-,\quad\text{and}\quad 
\pi_{\rm in}^- : \Delta^+ \to \Sigma^-
\end{align*} 
that are induced by the flow of \eqref{normal} for $\mu\in M$ and $\varepsilon,\delta>0$ sufficiently small; see \figref{conp}. We hence need to prove that the return map $ \pi = \pi_{\rm in}^-\circ \pi_{\rm out}^+ \circ \pi_{\rm in}^+ \circ \pi_{\rm out}^-: \Sigma^-\to \Sigma^-$ is well-defined. 

We first consider the case where the folded singularities $q^-$ and $q^+$ are connected.
Then, the maps $ \pi^\mp_{\rm in} $ are well-defined by standard GSPT; in particular, the constants $\rho_1$ and $\rho_2$ are sufficiently small such that trajectories with initial conditions on $ \Delta^+\lvert_{\{z>z_{q}^+\}} $, respectively $ \Delta^-\lvert_{\{z<z_{q}^-\}} $, are attracted exponentially close to $ \mc{Z}^{-}_{\varepsilon\delta} $, respectively $ \mc{Z}^{+}_{\varepsilon\delta} $, before crossing $ \Sigma^- $, respectively $ \Sigma^+ $. The well-definedness of $ \pi_{\rm out}^\mp $ follows from \cite{letson2017analysis}, where it was shown that trajectories which approach the vicinities of $ \mc{L}^+ $ and $ \mc{L}^- $ exponentially close to $ \mc{Z}^{+}_{\varepsilon\delta} $ and $ \mc{Z}^{-}_{\varepsilon\delta} $, respectively, diverge exponentially from the latter at most at a \textit{buffer point} that is bounded between $ q^+ $ and $ p^+ $, respectively between $ q^- $ and $ p^- $.
(According to \cite{hayes2016geometric}, these buffer points lie $ o\lp1\rp $-close to $ \mc{L}^\mp $, respectively.)

We may hence conclude that the map $\pi:\ \Sigma^-\to\Sigma^-$ is well-defined for all $\varepsilon,\delta>0$ sufficiently small. It follows that orbits of \eqref{normal} are attracted to either $ \mc{Z}^{-}_{\varepsilon\delta} $ or $ \mc{Z}^{+}_{\varepsilon\delta} $ irrespective of initial condition, possibly after a jump; then, they follow the slow flow until they have to jump and are attracted to either $ \mc{Z}^{-}_{\varepsilon\delta} $ or $ \mc{Z}^{+}_{\varepsilon\delta} $, which shows the existence of an attractor with double epochs of perturbed slow dynamics, as claimed. 

Finally, the above argument also holds for the case of aligned singularities: while the folded singularities $q^\mp$ have the same $z$-coordinates in the singular limit of $ \varepsilon=0=\delta $, the corresponding buffer points in the perturbed Equation~\eqref{normal}, with $ \varepsilon,\delta>0 $ sufficiently small, are still bounded between $ q^\mp $ and $ p^\mp $, respectively. Hence, orbits jump at points near $ \mc{L}^\pm $ and then cross $ \Delta^\pm $, and are therefore attracted to both $ \mc{Z}^{-}_{\varepsilon\delta} $ and $ \mc{Z}^{+}_{\varepsilon\delta} $, as in the connected case.

\begin{figure}[ht!]
		\begin{subfigure}[b]{0.45\textwidth}
		\centering
		\includegraphics[scale = 0.40]{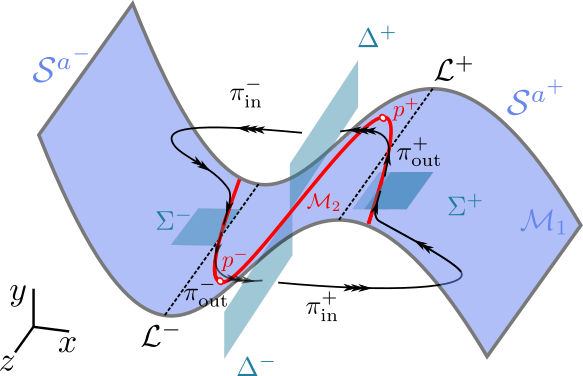}
		\caption{}
	\end{subfigure}
	~
	\begin{subfigure}[b]{0.45\textwidth}
		\centering
		\includegraphics[scale = 0.45]{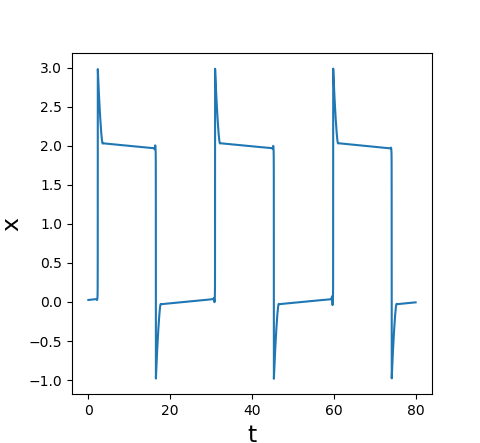}
		\caption{}
	\end{subfigure}
	\caption{Schematic illustration of the emergence of MMO trajectories with double epochs of perturbed slow dynamics in \eqref{normal}, for the case of aligned or connected singularities: (a) singular geometry and transition maps; (b) corresponding time series for the Koper model, as defined in \secref{koper}.}
	\figlab{conp}
\end{figure}

\end{proof}


We remark that \thmref{conne} guarantees the existence of an \textit{attractor} for Equation~\eqref{normal} with double epochs of perturbed slow dynamics; however, a more specific characterisation, such as of its periodicity or chaoticity, is dependent on the properties of the function $ \phi(x,y,z) $ in \eqref{normal-c} and hence requires a case-by-case study. Moreover, we note that the ``double epoch" regime can be further divided into subregimes where SAOs occur ``above" and ``below''; SAOs are seen ``above" with SAO-less slow dynamics below, or vice versa; or ``three-timescale" relaxation oscillation is found, with the flow alternating between fast, intermediate, and slow SAO-less dynamics. Again, a precise characterisation requires careful consideration of the given function $ \phi(x,y,z) $ in \eqref{normal-c}. (Thus, for instance, we typically observe double-epoch MMO trajectories with SAOs only before relaxation in the Koper model from chemical kinetics studied in \secref{koper}; cf.~\figref{fareys} and \figref{relax}.) 

%

\subsection{Remote singularities}
In the case where the folded singularities of $\mathcal{M}_1$ in \eqref{normal} are remote, recall \defnref{singalign} and \figref{singcycle}, the perturbed flow of Equation~\eqref{normal} with $ \varepsilon,\delta>0 $ sufficiently small can exhibit MMOs with single epochs of SAOs, or ``two-timescale" relaxation oscillation where the flow alternates between the fast and the intermediate dynamics for $\mu\in M$; see again \figref{muplane} and \eqref{muqmp}. 

First, to show the existence of relaxation oscillation in a $ \mu $-subregime of $(\mu_{q}^-,\mu_{q}^+)$, we combine the approaches of \cite{krupa2008mixed} and \cite{szmolyan2004relaxation}. To leading order in $ \varepsilon $ and $ \delta $, the $\mu$-values which separate the corresponding parameter regimes can be determined by requiring that the intermediate flow on $ \mathcal{S}^{a^-}$ is ``balanced'' by that on $ \mathcal{S}^{a^+} $. To that end, we consider the singular limit of $ \varepsilon=0 $ with $ \delta>0 $ sufficiently small in \eqref{normal}; in other words, we restrict to the flow on $ \mc{S}^{a^\mp} $, neglecting the $ z $-displacement due to the fast flow (in $\varepsilon$) away from $ \mc{S}^{a^\mp} $, see \cite{krupa2008mixed} for details. We begin by defining 
\begin{align}
x_0 := -\frac{f_2}{f_3}, \quad x_{\rm max} := -\frac{2f_2}{3f_3}, \quad\text{and}\quad x_{\rm max}^\ast:=\frac{f_2}{3f_3},  \eqlab{exmaxes}
\end{align}
where $ x_0 $ is the $ x $-coordinate of $ P(\mc{L}^-) $,  $ x_{\rm max}^\ast $ is the $ x $-coordinate of $ P(\mc{L}^+) $, and $ x_{\rm max} $ is the $ x $-coordinate of $ \mc{L}^+ $, cf. \figref{ro} below; here, we recall that $ P(\mc{L}^\mp)\subset\mc{S}^{a^\mp} $ denotes the projection of $ \mc{L}^\mp $ onto $ \mc{S}^{a^\mp} $ along the fast fibres of \eqref{norm12lay-a}. 

We then define
\begin{subequations}\eqlab{Res}
\begin{gather}
\mathcal{G}^-_0\lp z, \mu\rp := \mathcal{G}\lp x_\tn{max}^\ast,0;z;\mu\rp, \qquad
\mathcal{G}^+_0\lp z, \mu\rp := \mathcal{G}\lp x_0,x_\tn{max};z;\mu\rp, \eqlab{gis}\\
\text{and}\quad \mc{R}(z,\mu) := \mc{G}_0^-(z,\mu)+\mc{G}_0^+(z,\mu), \eqlab{Ris}
\end{gather}
\end{subequations}
where we recall that  
\begin{align*}
\mathcal{G}\lp x_0,x_1;z_0;\mu\rp = \int_{x_0}^{x_1} \frac{ F'(\sigma)\lp \mu+\phi \lp \sigma,F(\sigma),z_0\rp\rp}{\alpha\sigma+\beta F(\sigma)-z_0} \tn{d}\sigma
\end{align*}
is obtained by eliminating time in the reduced flow on $ \mc{S}^{a^\mp} $ under \eqref{classy-des} and integrating; cf.~\lemmaref{strong}. Finally, for future reference, we also write  
\begin{align}
I:= (z_{q}^-,z_{q}^+)\quad\text{and}\quad\bar I:= \big[z_{q}^-,z_{q}^+\big].
\end{align}
	\begin{figure}[ht!]
	\centering
	\begin{subfigure}[b]{0.45\textwidth}
		\centering
		\includegraphics[scale = 0.40]{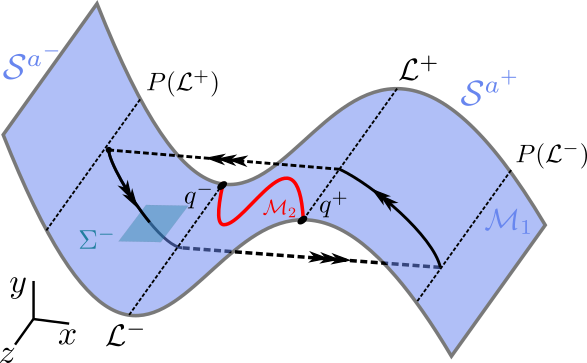}
		\caption{}
	\end{subfigure}
	~
	\begin{subfigure}[b]{0.45\textwidth}
		\centering
		\includegraphics[scale = 0.45]{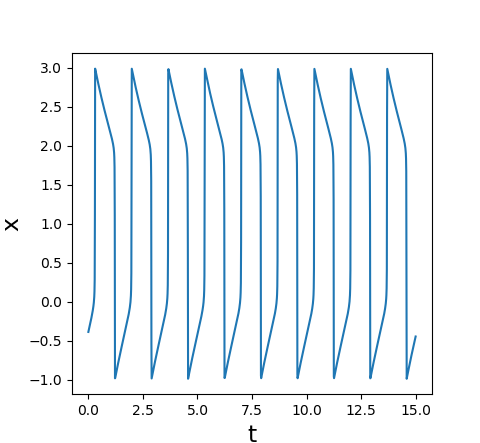}
		\caption{}
	\end{subfigure}
	\caption{Schematic illustration of the emergence of relaxation oscillation in \eqref{normal}, for the case of remote singularities: (a) singular geometry; (b) corresponding time series for the Koper model studied in \secref{koper}.}
	\figlab{ro}
\end{figure}
To leading order in $ \delta $ and for appropriately restricted $ \mu $-values, as specified below, the singular (in $\varepsilon$) trajectory through a point $ (0,0,z) \in \mc{L}^-\lvert_{I}$ returns to a point $ (0,0,\hat{z}) $ on $ \mc{L}^-\lvert_{I} $ , where
\begin{align}
\hat{z} = z+\delta\mc{R}(z,\mu)+\mc{O}(\delta^2). \eqlab{zret}
\end{align}
In the following, we will show that this map is well defined; see \thmref{relax} below.

We will say that the flow on $ \mathcal{S}^{a^-}$ is balanced by that on $ \mathcal{S}^{a^+} $ at a point with $ z=z^\ast $ for $ \delta>0 $ sufficiently small if $ \hat{z} = z^\ast $, i.e., if $ \mc{R}(z^\ast,\mu)=0 $. 
We will require the following technical result.

\begin{lemma}
	If the folded singularities $q^\mp$ of $\mathcal{M}_1$ are remote, i.e., if $ \frac{\alpha}{\beta} <\frac{2f_2^2}{9f_3} $, then
	\begin{align}
	\int_{x_{0}}^{x_{\rm max}} \frac{ F'(\sigma)}{\alpha\sigma+\beta F(\sigma)-z_{q}^\mp} \tn{d}\sigma>0 \quad\text{and}\quad 
	\int_{x_{\rm max}^\ast}^{0} \frac{ F'(\sigma)}{\alpha\sigma+\beta F(\sigma)-z_{q}^\mp} \tn{d}\sigma>0,
	\eqlab{intord}
	\end{align}
	as well as
	\begin{align}
	\mc{G}_0^-(z_{q}^-,\mu)>0 \quad\text{and}\quad 	\mc{G}_0^+(z_{q}^+,\mu)<0,
	\eqlab{inGs}
	\end{align}
	for $\mu\in\overline{M}$.
	\lemmalab{denomsign}
\end{lemma}
\begin{proof}
	The assertions in \eqref{intord} and \eqref{inGs} follow immediately from the below:
	\begin{enumerate}
		\item $ F'(x) <0$ for $ x\in (x^\ast_{\rm max}, 0)\cup(x_0,x_{\rm max}) $, i.e., on $ \mc{S}^a $, recall Equation~\eqref{sa};
		\item $ \alpha x+\beta F(x)-z<0 $ for $ (x,z) \in (x^\ast_{\rm max}, 0)\times\bar I$, recall Equation~\eqref{M2} and \figref{singcycle}(a);
		\item $ \alpha x+\beta F(x)-z>0 $ for $ (x,z) \in (x_0,x_{\rm max})\times\bar I$, again by \eqref{M2} and	\figref{singcycle}(a);
		\item $ \mu+\phi \lp \sigma,F(\sigma),z\rp>0 $ for $ (x,z,\mu) \in (x^\ast_{\rm max}, 0)\times\bar I\times \overline{M}$, recall \asuref{redflow};
		\item $ \mu+\phi \lp \sigma,F(\sigma),z\rp<0 $ for $ (x,z,\mu) \in (x_0,x_{\rm max})\times\bar I\times \overline{M}$, again by \asuref{redflow}.
	\end{enumerate}
\end{proof}

The values $\mu_r^- $ and  $\mu_r^+$ for which the reduced flow on $ \mathcal{S}^{a^-}$ is balanced by that on $ \mathcal{S}^{a^+} $ at $ q^-$ and $ q^+ $, respectively, to leading order in $ \delta>0 $ are found by solving $ \mc{R}(z_{q}^-,\mu)=0$ and $ \mc{R}(z_{q}^+,\mu)=0$, respectively. The relative position of $\mu_r^- $ and  $\mu_r^+$ on the real line depends on the properties of the function $ \phi(x,y,z) $ in \eqref{normal-c}; we therefore make the following assumption. 

\begin{asu}
Denote by $ \mu_r^\mp $ the $ \mu $-values for which
\begin{align}
\mc{R}(z_{q}^-,\mu_r^-)=0 \quad\text{and}\quad \mc{R}(z_{q}^+,\mu_r^+)=0\eqlab{mump}
\end{align}
hold, i.e., define
\begin{align}
\mu_r^\mp &:= -\frac{\int_{x_{\rm max}^\ast}^{0} \frac{F'(\sigma)\phi \lp \sigma,F(\sigma),z_{q}^\mp\rp}{\alpha\sigma+\beta F(\sigma)-z_{q}^\mp} \tn{d}\sigma+\int_{x_{0}}^{x_{\rm max}} \frac{F'(\sigma)\phi \lp \sigma,F(\sigma),z_{q}^\mp\rp}{\alpha\sigma+\beta F(\sigma)-z_{q}^\mp} \tn{d}\sigma}{\int_{x_{\rm max}^\ast}^{0} \frac{F'(\sigma)}{\alpha\sigma+\beta F(\sigma)-z_{q}^\mp} \tn{d}\sigma+\int_{x_{0}}^{x_{\rm max}} \frac{F'(\sigma)}{\alpha\sigma+\beta F(\sigma)-z_{q}^\mp} \tn{d}\sigma}.
\eqlab{muclose}
\end{align}
Then, we assume that  
\begin{align*}
\mu_{q}^-<\mu_r^-<\mu_r^+<\mu_{q}^+.
\end{align*}
\asulab{mr}
\end{asu}

Clearly, \lemmaref{denomsign} now implies that the denominator in \eqref{muclose} is non-zero. \asuref{mr} is essentially an assumption on the properties of the function $ \phi(x,y,z) $ in \eqref{normal-c}, which is made for consistency with the Koper model, Equation~\eqref{koper1}, after transformation to the prototypical Equation~\eqref{normal}. We now make an additional assumption on $ \phi(x,y,z)$, which is also consistent with the Koper model.
\begin{asu}
	We assume that $ \partial_z\phi(x,y,z)\leq0 $ for $ z\in \bar I $. 
	\asulab{phiz}
\end{asu}
We remark that \asuref{phiz} is sufficient, but not necessary, for the occurrence of relaxation oscillation in \eqref{normal}; see also the discussion following \thmref{relax} below.
We now introduce a final preliminary technical result. 
\begin{lemma}
	Assume that \asuref{phiz} holds and that the folded singularities of Equation~\eqref{normal} are remote in the sense of \defnref{singalign}, i.e., that $\frac{\alpha}{\beta}<\frac{2f_2^2}{9f_3}$ in \eqref{normal}. Then, the following holds for 	$ z\in \bar I\times \left[\mu_r^-,\mu_r^+\right] $:
	\begin{align}
	\begin{aligned}
	\partial_\mu \mc{R}(z,\mu)<0\\
	\end{aligned}
	\eqlab{partzInt}
	\end{align}
and 	
\begin{align}
\begin{aligned}
\partial_z\mc{R}(z,\mu)>0.
\end{aligned}
\eqlab{partRmu}
\end{align}
	\lemmalab{partzInt}
\end{lemma}
\begin{proof}
The proof is similar to that of \lemmaref{denomsign}.
\end{proof}

We now state our main result in this section:
\begin{theorem}
Assume that \asuref{albeta} through \asuref{phiz} hold, that the folded singularities of Equation~\eqref{normal} are remote in the sense of \defnref{singalign}, i.e., that $\frac{\alpha}{\beta}<\frac{2f_2^2}{9f_3}$ in \eqref{normal}, and fix $\mu\in(\mu_r^-, \mu_r^+)$. Then, there exist $ \varepsilon_0, \delta_0>0 $ sufficiently small such that Equation~\eqref{normal} admits a stable relaxation oscillation orbit for all $\lp\varepsilon,\delta\rp \in (0,\varepsilon_0)\times (0,\delta_0) $.
\thmlab{relax}
\end{theorem}

\begin{proof}
	The proof in based on showing that the assumptions of \cite[Theorem 4]{szmolyan2004relaxation} are satisfied.

	\begin{itemize}
		\item \underline{Assumption 1 in \cite[Theorem 4]{szmolyan2004relaxation}}. By construction, the manifold $ \mc{M}_1 $ is $ S $-shaped, with two attracting sheets $ \mc{S}^{a^\mp} $ separated by a repelling sheet $ \mc{S}^r $;  recall \eqref{els}, \eqref{sa}, and \eqref{sr}. 
		
		\item \underline{Assumption 2 in \cite[Theorem 4]{szmolyan2004relaxation}}. The tranversality condition is satisfied on $ \mc{L}^\mp\lvert_{I} $.
		
		\item \underline{Assumptions 3 and 4 in \cite[Theorem 4]{szmolyan2004relaxation}}. In the singular limit of $ \varepsilon=0=\delta $, one can construct a $1$-dimensional map
		\begin{align}
		\sigma^-:\ P(\mc{L}^+)\lvert_{I} \ \to \ \mc{L}^-,  \qquad z\ \mapsto \ z, \eqlab{sigmin}
		\end{align}
		where we recall that $ P:\mc{L}^\mp\to \mc{S}^{a^\pm} $ is the projection of a point on $ \mc{L}^\mp$ to a point on $ \mc{S}^{a^\pm} $ along the fast fibres of \eqref{norm12lay-a}.  
		For $ \delta>0 $ sufficiently small, the map $ \sigma^- $ perturbs smoothly to
		\begin{gather}
		\begin{aligned}
		\sigma^-_\delta:\ P(\mc{L}^+)\lvert_{I} \ \to \ \mc{L}^-, \qquad
		z\ \mapsto \ z+\mc{O}(\delta); 
		\end{aligned}
		\end{gather}
		in particular, by eliminating time in \eqref{classy-des}, we may approximate the map $ \sigma^-_\delta $ by
		\begin{align}
		z\ \mapsto \ z+\delta \mc{G}_0^-(z,\mu)+\mc{O}(\delta^2). \eqlab{sigdel}
		\end{align}
		Moreover, since
		\begin{align}
		\begin{aligned}
		\mc{G}_0^-(z_{q}^-,\mu) &>0
		\end{aligned} 
		\eqlab{G0-}
		\end{align}
		for $ \mu \in M$, by \lemmaref{partzInt}, the map $ \sigma^-_\delta: P(\mc{L}^+)\lvert_{I} \ \to \ \mc{L}^-\lvert_{\{z>z_{q}^-\}}$ is well-defined, with $ \sigma_\delta^-(z)\to \sigma^-(z) $ as $ \delta\to0 $ uniformly in $ z$. In particular, \eqref{G0-} implies that the point on $ P(\mc{L}^+) $  with $ z=z_{q}^- $ is mapped to a point on $ \mc{L}^- $ with $ z>z_{q}^- $. Since the map $ \sigma_\delta^- $ is induced by the reduced flow on $ \mc{S}^{a^-} $, it follows by existence and uniqueness of solutions that all points on $ P(\mc{L}^+) $  with $ z>z_{q}^- $ are mapped to points on $ \mc{L}^- $ with $ z>z_{q}^- $. Similarly, one can construct a map $ \sigma^+_\delta: P(\mc{L}^-)\lvert_{I} \ \to \ \mc{L}^+\lvert_{\{z<z_{q}^+\}}$, under which points on $ P(\mc{L}^-) $  with $ z\leq z_{q}^+ $ are mapped to points on $ \mc{L}^+ $ with $ z<z_{q}^- $. 
		
		The composition of $ \sigma^-_\delta$ and $\sigma^+_\delta $ defines the return map
		\begin{gather}\eqlab{ellmin}
		\begin{aligned}
		\pi^-:=\sigma^-_\delta \circ  \sigma^+_\delta:	\ \mc{L}^-\lvert_I \ \to \ \mc{L}^-\lvert_I, \qquad
		z\ \mapsto \ z + \delta\mc{R}(z,\mu). 
		\end{aligned}
		\end{gather}
		Since
		\begin{align*}
		\begin{aligned}
		\partial_\mu \mc{R}(z,\mu)>0 \quad\text{for all }(z,\mu) \in I\times (\mu_{r}^-, \mu_{r}^+),
		\end{aligned}
		\end{align*}
		i.e., since $ \mc{R}(z,\mu) $ is an increasing function of $ \mu $ by \lemmaref{partzInt}, and since, moreover,
		\begin{align*}
		\mc{R}(z_{q}^-,\mu_r^-)=0\quad\text{and}\quad\mc{R}(z_{q}^+,\mu_r^+)=0
		\end{align*} 
		by \asuref{mr}, it correspondingly follows that 
		\begin{align*}
		\mc{R}(z_{q}^-,\mu)>0\quad\text{and}\quad		\mc{R}(z_{q}^+,\mu)<0 \quad\text{for all } \mu\in(\mu_{r}^-, \mu_{r}^+),
		\end{align*} 
		 and the map $ \pi^-:	\mc{L}^-\lvert_I \ \to \ \mc{L}^-\lvert_I $ is therefore well defined. By the intermediate value theorem, for any $ \mu\in(\mu_{r}^-, \mu_{r}^+) $, there exists $ z^\ast \in I $ such that $ \mc{R}(z^*,\mu)=0 $; therefore, $ z^\ast $ is a fixed point of $ \pi^- $. Moreover, from 	\eqref{partzInt}, it follows that $ \partial_z \mc{R}(z,\mu)<0 $ for all $ \lp z,\mu\rp \in\bar I\times\left[\mu_{r}^-, \mu_{r}^+\right]$; therefore, the fixed point $ z^\ast \in I $ is unique, which implies the existence of a unique singular cycle for any $ \delta>0 $ sufficiently small.
		
		\item \underline{Assumption 5 in \cite[Theorem 4]{szmolyan2004relaxation}}. Consider the section 
		\begin{align*}
		\Sigma^- &= \lb \lp x,y,z\rp \in \mb{R}^3\ \lvert \ x<x_{q}^-\ \text{and}\ y = y_{q}^-+\rho_1, \ \text{with}\ \rho_1>0 \rb;
		\end{align*}
		see \figref{ro}. The return map $ \pi^-:	\mc{L}^-\ \to \ \mc{L}^-$  in \eqref{ellmin} induces a return map $ \Pi^-:\Sigma^-\to \Sigma^- $. Correspondingly, fixed points of $ \pi^- $ on $ \mc{L}^- $ are connected to fixed points of $ \Pi^- $ on $ \Sigma^- $ via trajectories of the reduced flow in \eqref{classy-des} on $ \mc{S}^{a^-} $. In particular, a fixed point of $ \Pi^- $ is hyperbolic and attracting, respectively repelling, if and only if the corresponding fixed point of $ \pi^- $ is hyperbolic and attracting, respectively repelling. 
		Since $ \partial_z \lp z+\delta\mc{R}(z,\mu)\rp  <1 $ for all $ z\in\bar I $, by \eqref{partzInt}, $ z^\ast $ is attracting in our case, as is the corresponding fixed point of $ \Pi^- $ on $ \Sigma^- $, which shows the existence of a hyperbolic and attracting singular periodic orbit.
	\end{itemize}
The above assumptions are established in the singular limit of $ \varepsilon=0 $, with $ \delta>0 $ sufficiently small. For $ \varepsilon, \delta>0 $ small and $ z\in I $, a trajectory with initial condition $ (x^\ast_{\rm max}, y_{q}^+, z) $ on $ \mc{S}^{a^-}_{\varepsilon\delta} $ -- i.e., at the height of $ P(\mc{L}^+) $ -- is mapped to a point on $ \mc{S}^{a^-}_{\varepsilon\delta} $ under the flow of \eqref{normal} and following a large excursion: 
\begin{align}
z\mapsto z+\delta\mc{R}(z,\mu)+\mc{O}(\delta^2,\delta\varepsilon\ln\varepsilon). \eqlab{reter}
\end{align} 
The above map is smooth in $ z $ and $ \mu $. To see that \eqref{reter} holds, we note that in the double singular limit of $ \varepsilon=0=\delta $, the return map is the identity which, for $ \varepsilon=0 $ and $ \delta>0$ small, then perturbs to \eqref{ellmin}. By \cite[Theorems 3 and 4]{szmolyan2004relaxation}, an $ \mc{O}(\varepsilon\ln\varepsilon) $-contribution arises through the perturbation of the functions $ \mc{G}^\mp_0 $ for $ \varepsilon,\delta>0 $ sufficiently small which is multiplied by a factor of $ \delta $ in our case, since the $ z $-displacement needs to be zero for $ \delta=0 $ regardless of $ \varepsilon $, as can be seen by setting $\delta=0$ in \eqref{norm12fast}. Hence, \eqref{reter} follows.

\end{proof}

We reiterate that \asuref{phiz} is sufficient, but not necessary for stable relaxation oscillation to occur, as described in \thmref{relax}. Specifically, \lemmaref{partzInt} guarantees the existence \textit{and uniqueness} of a singular cycle in the limit of $ \varepsilon=0 $, with $ \delta>0 $ small; see the third bullet point in the proof of \thmref{relax}. In general, the function $ \phi(x,y,z)$ could be such that more than one singular cycle exists in that limit. The stability of each of those would be studied individually, which  would result in the existence of more than one relaxation oscillation orbit for $ \varepsilon,\delta>0 $ small, some of which could potentially be stable, while others could be unstable. We reiterate that \asuref{phiz} is satisfied for the Koper model from chemical kinetics which is studied in \secref{koper}. 

We now state our main result on the existence of MMOs with single epochs of perturbed slow dynamics. 

\begin{theorem}
Assume that \asuref{albeta} through \asuref{phiz} hold, that the folded singularities of Equation~\eqref{normal} are remote in the sense of \defnref{singalign}, i.e., that $\frac{\alpha}{\beta}<\frac{2f_2^2}{9f_3}$ in \eqref{normal}, and fix $\mu\in(\mu_{q}^-, \mu_r^-)$. Then, there exist $ \varepsilon_0, \delta_0>0 $ sufficiently small such that Equation~\eqref{normal} features MMOs with single SAO epochs for all $\lp\varepsilon,\delta\rp \in (0,\varepsilon_0)\times (0,\delta_0) $.
	\thmlab{mmo}
\end{theorem}

\begin{proof}
	Consider the section
	\begin{align*}
	\Sigma^- &= \lb \lp x,y,z\rp \in \mb{R}^3\ \lvert \ x<x_{q}^-\ \text{and}\ y = y_{q}^-+\rho_1, \ \text{with}\ \rho_1>0 \tn{ small}\rb.
	\end{align*} 
	As in \thmref{conne} and \thmref{relax}, the return map $ \pi^-:\Sigma^-\to\Sigma^- $ induced by the flow of \eqref{normal} is well defined for $ \mu \in ( \mu_{q}^-,\mu_r^- ) $ and $ \varepsilon, \delta>0 $ sufficiently small.
	
	\begin{figure}[ht!]
		\centering
		\begin{subfigure}[b]{0.45\textwidth}
			\centering
			\includegraphics[scale = 0.40]{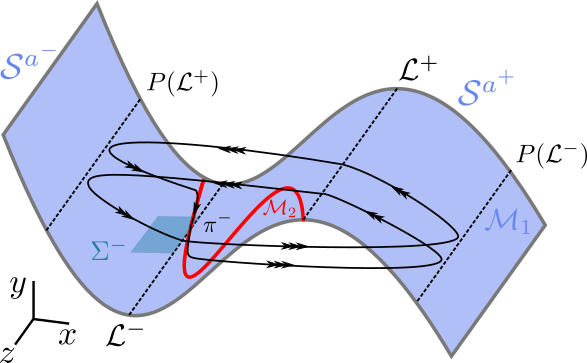}
			\caption{}
		\end{subfigure}
		~
		\begin{subfigure}[b]{0.45\textwidth}
			\centering
			\includegraphics[scale = 0.45]{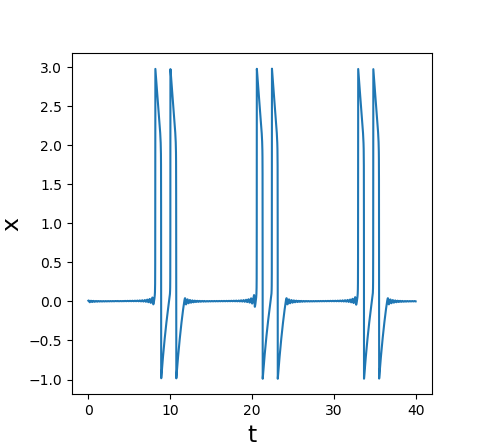}
			\caption{}
		\end{subfigure}
		\caption{Schematic illustration of the emergence of MMO trajectories with single epochs of perturbed slow dynamics in \eqref{normal}, for the case of remote singularities: (a) singular geometry and return map; (b) corresponding time series for the Koper model studied in \secref{koper}.}
		\figlab{mmo}
	\end{figure}

	Consider now a point $ (x^\ast_{\rm max}, y_{q}^+, z) $ on $ \mc{S}^{a^-}_{\varepsilon\delta} $, i.e., at the height of $ P(\mc{L}^+) $, with $ z\in I $. The corresponding trajectory returns to $ \mc{S}^{a^-}_{\varepsilon\delta} $, after a large excursion, at a point with $ z+\delta\mc{R}(z,\mu)+\mc{O}(\delta^2,\delta\varepsilon\ln\varepsilon) $; recall the proof of \thmref{relax}. Moreover, there holds that  $\mc{R}(z,\mu)<0$ for all $ (z,\mu)\in I\times (\mu_{q}^-,\mu_r^-) $, which follows from \eqref{partzInt} and \eqref{partRmu}. Therefore, the trajectory ``drifts'' towards the negative $ z $-direction until it reaches a point $ (x^\ast_{\rm max}, y_{q}^+, z) $ on $ \mc{S}^{a^-}_{\varepsilon,\delta} $ with $ z<z_{q}^- $, i.e., until it enters the funnel of $ q^- $. We reiterate that, in a first approximation, the funnel of $ q^- $ is the area in $ \mc{S}^{a^-} $ bounded by $ \mc{Z}^{-} $ and the intermediate fibre of \eqref{intermediate} that crosses $ q^- $, which is given by $ \lb z= z_{q}^-\rb$. 
	
	Points $ (x^\ast_{\rm max}, y_{q}^+, z) $ on $ \mc{S}^{a^-}_{\varepsilon\delta} $ with $ z<z_{q}^- $ are attracted  to the vicinity of $ q^- $ and undergo SAOs. According to \cite{krupa2010local,hayes2016geometric}, the buffer point beyond which all trajectories have to diverge exponentially from $ \mc{Z}^{-} $ lies $ o(1) $-close to $q^-$. Therefore, there exist $ \varepsilon_0,\delta_0>0 $ sufficiently small, which satisfy in particular $\varepsilon_0<\lp z_{q}^+-z_{q}^-\rp^2 $ and $ \delta_0< z_{q}^+-z_{q}^-$, such that, for all $ (\varepsilon,\delta)\in (0,\varepsilon_0)\times(0,\delta_0) $, trajectories that diverge exponentially from $ \mc{Z}^{-}_{\varepsilon\delta} $ undergo a slow drift towards the negative $z$-direction, without interacting with $ \mc{Z}^{+}_{\varepsilon\delta} $, until they enter the funnel of $ q^- $. The above implies the existence of MMO trajectories with single epochs of perturbed slow dynamics, as claimed. 	
\end{proof}

The requirement that $\varepsilon_0<\lp z_{q}^+-z_{q}^-\rp^2 $ and $ \delta_0< z_{q}^+-z_{q}^-$ is a sufficient condition which guarantees that MMOs with single epochs of perturbed slow dynamics exist, regardless of how close the folded singularities are in the $ z$-direction. As will become apparent in \secref{local} below, the buffer point near $ \mc{L}^- $ lies $ \mc{O}(\sqrt{\varepsilon},\delta) $-close to the fold line $ \mc{L}^- $. Hence, trajectories which diverge exponentially from $ \mc{Z}^{-} $ are not able to reach $ \mc{Z}^{+} $. If the singularities $ q^\mp $ are remote, but sufficiently close in the $ z $-direction, then trajectories can potentially interact with both $ \mc{Z}^{\mp} $ for (sufficiently large) values of $ \varepsilon $ and $ \delta $; correspondingly, Equation~\eqref{normal} can feature MMOs with double epochs of perturbed slow dynamics, or even more exotic patterns where the two epochs are separated by LAOs, in that case; see \cite{hhgspt2019} for an example of the latter in the context of the multi-timescale Hodgkin-Huxley equations from mathematical neuroscience. We remark that we have not been able to find such ``exotic'' behaviour in the context of the Koper model from chemical kinetics; cf.~\secref{koper}.

Finally, in regard to the number of LAOs that can occur between SAO segments, we have the following result. 
\begin{cor}
Assume that \asuref{albeta} through \asuref{phiz} hold, that the folded singularities of Equation~\eqref{normal} are remote in the sense of \defnref{singalign}, i.e., that $\frac{\alpha}{\beta}<\frac{2f_2^2}{9f_3}$ in \eqref{normal}, fix $ \mu\in\lp \mu_{q}^-, \mu_r^-\rp $ and consider $ \varepsilon,\delta>0 $ sufficiently small. Denote by $p_{\rm out} = \lp x_{\rm out}, y_{\rm out}, z_{\rm out}\rp $ the point at which a given trajectory diverges exponentially from $ \mathcal{Z}_{\varepsilon\delta}^{-} $, with $ z_{\rm out}>0 $, and denote by $L$ the number of large excursions that follow before the trajectory is again attracted exponentially close to $ \mathcal{Z}_{\varepsilon\delta}^{-} $. Then, the following holds.
	\begin{enumerate}	
		\item If $z_{\rm out}+\delta\mc{R}(0,\mu) <0$, then $L=1$; 
		\item if $0<z_{\rm out}+\delta\mc{R}(0,\mu)  <z_{\rm out}$, then
		\begin{align}
		L = 1+\left\lfloor \frac{z_\textnormal{out}}{\delta\mc{R}(0,\mu)} \right\rfloor, 
		\eqlab{ell}
		\end{align}
		where $\lfloor~\rfloor$ denotes the floor function.
	\end{enumerate}
	\proplab{LAObounds}
\end{cor}
\begin{proof}
	Both statements follow immediately from \thmref{mmo}.
\end{proof}

\subsection{Summary}\seclab{summary}
In summary, the emergence of mixed-mode dynamics in \eqref{normal} can thus be understood as follows. By standard GSPT \cite{fenichel1979geometric},
the normally hyperbolic portions $ \mathcal{S}^{a^\mp}$ and $ \mc{Z}^{\mp}$ of $\mathcal{M}_1$ and
$\mathcal{M}_2$, respectively, perturb to $ \mathcal{S}^{a^\mp}_{\varepsilon\delta} $ and $ \mc{Z}^{\mp}_{\varepsilon\delta} $, respectively.
Given an initial point $ (x,y,z)\in \mathcal{S}^{a^-}_{\varepsilon\delta} $, the corresponding trajectory will follow the intermediate flow on $\mathcal{S}^{a^-}_{\varepsilon\delta}$ until it is either attracted to $\mathcal{Z}^{-}_{\varepsilon\delta} $ or until it reaches the vicinity of $ \mathcal{L}^- $. In the former case, the trajectory then follows the slow flow on $\mathcal{Z}^{-}_{\varepsilon\delta} $ and can undergo SAOs; in the latter case, no slow dynamics occurs, and the trajectory jumps near $ \mathcal{L}^- $ to the opposite attracting sheet $\mathcal{S}^{a^+}_{\varepsilon\delta} $, resulting in a {large excursion}. The above sequence then begins anew; see \figref{conp}, \figref{ro}, and \figref{mmo} for schematic illustrations: depending on the relative geometry of the folded singularities $q^\mp$ of $\mathcal{M}_1$, oscillatory trajectories with single, double, or no epochs of slow dynamics can occur, as indicated in \figref{muplane}. 

We emphasise that the ``double epoch'' regime in panel (b) of \figref{muplane} does not necessarily imply mixed-mode dynamics with two epochs of SAOs but, rather, with double epochs of \textit{perturbed slow dynamics} of the corresponding singular cycles. That is, MMO trajectories are attracted to the vicinity of both branches $ \mc{Z}^{\mp}_{\varepsilon\delta} $ and hence exhibit slow dynamics; however, whether SAOs will occur depends on which region on $ \mathcal{Z} $ trajectories enter: by \lemmaref{delayedH} in Appendix~\ref{mechs}, they may experience either focal or nodal attraction. In particular, if a trajectory is attracted to the focal region on both $ \mathcal{Z}^{-}_{\varepsilon\delta} $ and $ \mc{Z}^{+}_{\varepsilon\delta} $, then two epochs of SAOs are observed. On the other hand, trajectories that are first attracted to the focal region on, say, $ \mathcal{Z}^{-} $ before being attracted to and repelled from the nodal region on $ \mc{Z}^{+} $ feature SAOs below and mere slow dynamics above. (The corresponding segment of the associated Farey sequence would be $ 1_s1^0 $, with $ s>0 $.) Similarly, a trajectory that is attracted to and repelled from nodal regions on both $ \mathcal{Z}^{-} $ and $ \mc{Z}^{+} $ features no SAOs at all and is hence a relaxation oscillation with fast, intermediate, and slow components; the associated Farey sequence would be $ 1^01_0 $. In the transition between remote and connected singularities, exotic MMO trajectories may occur which contain segments of two-timescale relaxation oscillation, SAOs above, and SAOs below. (The associated Farey sequence would be $ 1^sL_k $, with $ L,s, k>0 $.) Moreover, we postulate that chaotic mixed-mode dynamics may be possible. However, the above characterisation depends substantially on the particular form of the function $ \phi $ in \eqref{normal-c}; it is hence not feasible to further subdivide that region in \figref{muplane} on the basis of system parameters alone. Rather, a case-by-case study is required.

Finally, we remark on the role of the ratio between the scale separation parameters $ \varepsilon $ and $ \delta $ for the dynamics of Equation~\eqref{normal}. Locally, in order for the system to exhibit three timescales and for the iterative reduction from the fast via the intermediate to the slow dynamics to be accurate, $ \varepsilon $ and $\delta$ need to be sufficiently small, which is akin to asking ``When is $ \varepsilon $ small enough?" in a two-timescale system. We recall that the resulting SAOs will be either of sector type or of delayed Hopf type; again by \lemmaref{delayedH}, the width of the corresponding regions on $ \mathcal{Z} $ is either $ \mathcal{O}(\varepsilon) $ or $ \mathcal{O}(\sqrt{\varepsilon})$. Correspondingly,. By \cite{krupa2008mixed} and \lemmaref{strong}, the ``step'' in the $ z $-direction taken by trajectories after a large excursion and re-injection is $ \mathcal{O}(\delta) $; it therefore follows that if $ \delta =\mathcal{O}(\varepsilon^c)$ for $ 0<c<1 $, then trajectories will typically not undergo sector-type dynamics,
since the width of the latter regime is $ \mathcal{O}(\varepsilon) $. Hence, delay-type SAOs are expected to dominate in that case; see \figref{sechof}.

\section{The Koper model revisited}
\seclab{koper}
In this section, we revisit the Koper model from chemical kinetics, Equation~\eqref{koper1}, which we reiterate to be a particular realisation of Equation~\eqref{normal} for
\begin{subequations}\eqlab{kop_params1}
\begin{gather}
\varepsilon = \frac\epsilon{|k|},\quad f_2 = \frac3{|k|}, \quad f_3 = -\frac1{|k|}, \\
\alpha = 1, \quad \beta = -2, \\
\mu = \frac{k+\lambda+2}k,\quad\text{and}\quad \phi(x,y,z) = -y-z,
\end{gather}
\end{subequations}
after the transformation
$\lp x,y,z,\lambda,k,t\rp \to \lp x+1,y+\frac{2+\lambda}{\left|k\right|},-z-1+\frac{2(2+\lambda)}{\left|k\right|},\lambda,k,t\rp$. Henceforth, we will refer to \eqref{normal} with the above choice of parameters as the Koper model; here, $ k<0 $ and $ \lambda\in\mb{R} $ will be our bifurcation parameters. 

From \secref{singgeom}, it is apparent that the effect of the parameter $ k $ on the dynamics is more substantial than that of $ \lambda $, since variation in $k$ simultaneously affects
the timescale separation (through $ \varepsilon $) and the singular geometry (through $ f_2 $ and $ f_3 $), as well as the slow flow and the global return (through $ \mu $). Given $k<0$ fixed, on the other hand, variation in $ \lambda $ only affects the slow flow and the global return (through $ \mu $). It is therefore the parameter $ k $ that determines whether the folded singularities in the Koper model are remote, aligned, or connected, and whether the model can exhibit single or double epochs of SAOs. For given $ k<0 $, the parameter $ \lambda $ can differentiate between steady-state and oscillatory behaviour, as well as between mixed-mode dynamics and relaxation oscillation in the case of remote singularities.

\begin{remark}
Alternatively, the Koper model can be written in the symmetric form
\begin{align*}
\epsilon \dot{x} &= y-x^3+3x, \\
\dot{y} &=kx-2\lp y+\lambda\rp+z, \\
\dot{z} &=\delta(\lambda+y-z),
\end{align*}
which is invariant under the transformation $\lp x,y,z,\lambda,k,t\rp \to \lp -x,-y,-z,-\lambda,k,t\rp$ \cite{desroches2012mixed}. 
\remlab{symmetry}
\end{remark}

In the following, we will restrict to the case where $\lambda>0$ in \eqref{koper1}. Moreover, we will investigate the dynamics near $\mathcal{L}^-$ only: by \remref{symmetry}, the flow near $\mathcal{L}^+$ for $\lambda<0$ can then be inferred by symmetry; cf.~also panels (a) and (b) in  \figref{fareys}, where the corresponding time series are seen to be symmetric about  the $t$-axis for $ k $ fixed and $ \lambda\to-\lambda $.

\subsection{Singular geometry}
The critical and supercritical manifolds $\mathcal{M}_1$ and $\mathcal{M}_2$, respectively, for the Koper model are given by
\begin{align*}
\mathcal{M}_1 &= \lb  \lp x,y,z\rp \in\mb{R}^3~\Big\lvert ~ y = x^2\frac{3-x}{|k|} \rb\quad\text{and} \\
\mathcal{M}_2 &= \lb  \lp x,y,z\rp \in\mathcal{M}_1~\Big\lvert~ z = x-2x^2\frac{3-x}{|k|} \rb;
\end{align*}
see \secref{singgeom}. The normally hyperbolic portion $\mathcal{S}$ of the critical manifold $ \mathcal{M}_1 $ can be written as
\begin{align}
\begin{aligned}
\mathcal{S} &= \mathcal{S}^{a^-}\cup \mathcal{S}^r\cup\mathcal{S}^{a^+},
\end{aligned}
\end{align}
where
\begin{gather*}
\mathcal{S}^{a^-} = \lb  \lp x,y,z\rp \in\mathcal{M}_1~\big\lvert ~ x<0\rb,\quad  
\mathcal{S}^r = \lb  \lp x,y,z\rp \in\mathcal{M}_1~\big\lvert ~ 0<x<2\rb,\quad\text{and} \\
\mathcal{S}^{a^+} = \lb  \lp x,y,z\rp \in\mathcal{M}_1~\big\lvert ~ x>2\rb
.
\end{gather*}

The fold lines of $ \mathcal{M}_1 $ are located at
\begin{align}
\mathcal{L}^{-} = \lb \lp x,y,z\rp\in\mb{R}^3~\big\lvert~x = 0,~y=0\rb\quad\text{and}\quad\mathcal{L}^+ = \lb \lp x,y,z\rp \in\mathbb{R}^3 ~\Big\lvert~ x  =2,~y = \frac4{|k|}\rb;
\end{align}	
the corresponding folded singularities $q^\mp$ are found at
\begin{align}
q^-=(0,0,0)\quad\text{and}\quad
q^+ = \lp 2,~ \frac{4}{|k|},~2-\frac{8}{|k|}\rp.
\eqlab{sings}
\end{align}
We have the following result on the relative position of the singularities $q^\mp$:
\begin{proposition}
Let $\varepsilon=0=\delta$. Then, the folded singularities of the Koper model are aligned for $k=-4$, connected when $-4<k<0$, and remote for $k<-4$. 
\proplab{kopal}
\end{proposition}
\begin{proof}
The statement follows from \propref{relatives} and \eqref{kop_params}, or by comparison of the $ z $-coordinates of $ q^- $ and $ q^+ $.
\end{proof}

The supercritical manifold $\mathcal{M}_2$ is normally hyperbolic everywhere except at the fold points $p^\mp$, where
\begin{gather}
\begin{gathered}
x_p^\mp = 1\pm\sqrt{1-\tfrac{|k|}6},\quad y_p^\mp = \frac{\Big(2\pm\sqrt{1-\tfrac{|k|}{6}}\Big)\Big(1\mp\sqrt{1-\tfrac{|k|}{6}}\Big)^2}{|k|},\quad\text{and} \\
z_p^\mp = 1\mp\sqrt{1-\tfrac{|k|}{6}}-2\frac{\Big(2\pm\sqrt{1-\tfrac{|k|}{6}}\Big) \Big(1\mp\sqrt{1-\tfrac{|k|}{6}}\Big)^2}{|k|}.
\end{gathered}
\eqlab{foldloc}
\end{gather}
Based on the above, we have the following
\begin{proposition}\proplab{kopfolds}
If $-6<k<0$, then $\mathcal{M}_2$ admits two fold points which are located between the points of intersection of $\mathcal{M}_2$ with $\mathcal{L}^{\mp}$, i.e., on the repelling sheet of $\mathcal{M}_1$. If $k<-6$, then $\mathcal{M}_2$ admits no fold points.
\end{proposition}

We reiterate that, due to $ k<0 $, the fold points $ p^\mp $ in the Koper model cannot cross $ \mathcal{L}^\mp $, and that the corresponding singular geometry is therefore as depicted in \figref{m2bif}(c).
\begin{remark}
In \cite[Example 4.3]{cardin2017fenichel}, the manifold $ \mathcal{M}_2 $ is characterised as normally hyperbolic everywhere, in spite of its graph being $ S $-shaped. \propref{kopfolds} above shows that $ \mathcal{M}_2 $ can, in fact, admit two fold points at which normal hyperbolicity is lost. 
\end{remark}

\subsection{Classification of three-timescale dynamics}
Here, we classify the dynamics of the Koper model in the three-timescale context for various choices of the parameters $ k$ and  $\lambda $ in Equation~\eqref{koper1}. In particular, we hence construct the
two-parameter bifurcation diagram shown in \figref{kl-plane}; recall \figref{kl-plane0}. (A 
two-timescale analogue of \figref{kl-plane}, for the case of one fast and two slow variables in \eqref{koper1}, is presented in \cite{desroches2012mixed}.) 
Given the definition of $\mu$ in \eqref{kop_params}, we consider $\lambda$ as a function of $k$ here when retracing the analysis from \secref{singpert1}, in particular in relation to the classification in \figref{muplane}; the requisite calculations are simplified due to the symmetry of \eqref{koper1}, by \remref{symmetry}.

\begin{figure}[ht!]
	\centering
	\begin{subfigure}[b]{0.45\textwidth}
	\includegraphics[scale=0.34]{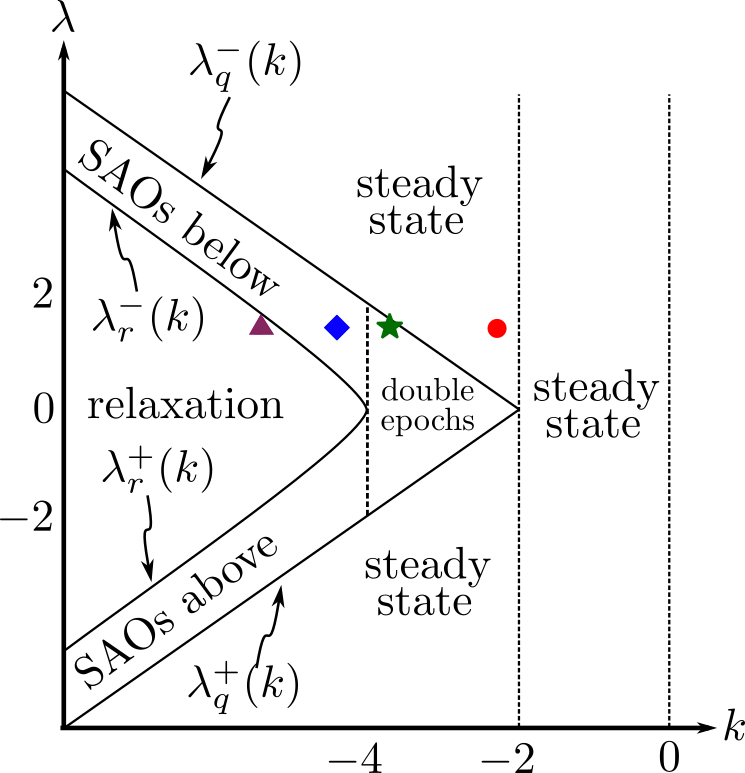}
	\caption{}
	\end{subfigure}
	\begin{subfigure}[b]{0.45\textwidth}
	\includegraphics[scale=0.38]{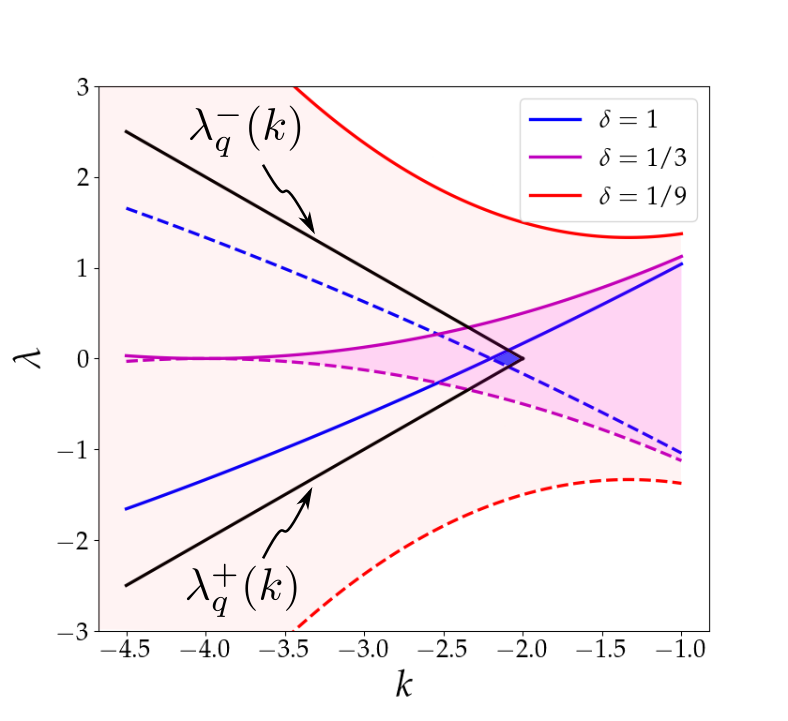}
	\caption{}
	\end{subfigure}
	\caption{(a) Two-parameter bifurcation diagram for the three-timescale Koper model, Equation~\eqref{koper1}: oscillatory dynamics is restricted to the triangular region of the $ (k,\lambda) $-plane that is bounded by $ \lambda_{q}^\mp(k) $; mixed-mode dynamics is separated from relaxation oscillation by the curves $ \lambda_r^\mp(k) $; {to leading order,} the mixed-mode regime is subdivided into regions of either single or double epochs of SAOs at $ k = -4 $. Numerical verification of the cases corresponding to the coloured shapes is given in \figref{relax}. (b) Classification of the folded singularities $q^\mp$ in dependence of $\delta$: dashed and solid curves correspond to $q^-$ and $q^+$, respectively, being of folded degenerate node type. For $\delta$ fixed, shading indicates parameter regimes in the $(k,\lambda)$-plane where both $q^\mp$ are folded nodes.
	With decreasing $ \delta $, these regimes stretch until the curves ``detach'' at $ \delta=\frac13 $; see \cite[Figure~16]{desroches2012mixed} for comparison.}
	\figlab{kl-plane}
\end{figure}

In a first step, we note that the boundary between steady-state behaviour and oscillatory dynamics in the Koper model is marked by curves that are $\mc{O}(\varepsilon,\delta)$-close to the lines given by $\lambda_{q}^\mp(k) = \mp(2+k)$; these are found by making use of \eqref{kop_params1} in \eqref{muqmp}, and solving for $ \lambda $.

It hence follows that oscillatory dynamics is restricted to the triangular area illustrated in \figref{kl-plane}. A further subdivision of that area is obtained by noting that mixed-mode dynamics is separated from relaxation oscillation by two curves $ \lambda_r^-(k) $ and $ \lambda_r^+(k)=-\lambda_r^-(k)$; {these are} found by substituting \eqref{kop_params1} into \eqref{muclose} and solving for $\lambda$. While analytical expressions for $ \lambda_r^\mp(k) $ can be obtained by direct integration, they are quite involved algebraically, and are hence not included here. These expressions imply that, for $ \varepsilon=0=\delta $ and $k<-4$, $\lambda_{q}^+(k)<\lambda_r^+(k) <\lambda_r^-(k)<\lambda_{q}^-(k)$, as well as that $ \lambda_r^\mp $ are asymptotically parallel to $ \lambda_{q}^\mp$, respectively, for $ \left|k\right| $ sufficiently large; moreover, the curves $\lambda_r^\mp(k)$ connect tangentially at $k=-4$. 
(Numerically, one finds that, for $ \varepsilon = \mathcal{O}(10^{-4}) $ and $ \delta = \mathcal{O}(10^{-2})$, the transition between mixed-mode dynamics and relaxation occurs at $\lambda_r^\mp(k) +\mathcal{O}(\delta)$, as is to be expected from \eqref{muclose}.)

Finally, the resulting, chevron-shaped region in which MMOs are observed is further divided into subregions in which either single or double epochs of SAOs are found; {to leading order in $\varepsilon$ and $\delta$, that division occurs at $k=-4$.} Geometrically, the division is due to the fact that the folded singularities $q^\mp$ in the Koper model are remote  for $ k<-4 $, while they are connected when $ -4<k<0 $. We emphasise that, in the two-timescale context of $\varepsilon$ sufficiently small and $ \delta =\mathcal{O}(1) $, MMOs with double epochs of SAOs occur in a very narrow region of the $(k,\lambda)$-plane, as shown in \figref{kl-plane}(b) for $\delta =1$ (shaded blue). That region corresponds to the regime where both folded singularities $q^\mp$ are of folded node type and trajectories are attracted to both of them through the associated funnels, by \cite{desroches2012mixed}; these funnels stretch as $\delta$ decreases, recall \lemmaref{strong}. Hence, in the three-timescale context, trajectories can reach both folded singularities $q^\mp$ as long as they are attracted to $\mathcal{M}_2$ on both $\mathcal{S}^{a^\mp}$, i.e., as long as $q^-$ and $q^+$ are aligned or connected. 

\begin{remark}
Comparing \figref{muplane} with \figref{kl-plane}, we note that the two panels in the former are combined in the latter, as one-parameter diagrams (in $\mu$) are merged into one two-parameter diagram in $ (k,\lambda) $; correspondingly, parallel lines with $\mu$ constant in \figref{muplane} are ``bent", and hence intersect, in \figref{kl-plane}. (Here, we reiterate that $ k $ determines the singular geometry of the Koper model, while $ \lambda $ affects the resulting flow.)
\end{remark}

\subsection{Numerical verification}

In this subsection, we verify our classification of the three-timescale dynamics of the Koper model for various representative choices of the parameters $k$ and $\lambda$, as indicated in \figref{kl-plane}. We initially fix $ \varepsilon = 0.01 = \delta$ and $ \lambda = 1.5 $, and we vary $ k $. We recall that the Koper model is symmetric in $ \lambda $, and that it hence suffices to consider positive $\lambda$-values; cf.~again \remref{symmetry} and \figref{fareys}.

For $ k=-2.2 $ (red circle), the flow of the Koper model converges to steady state; see panel (a) of \figref{relax}. For $ k=-3.6 $ (green asterisk), we observe mixed-mode dynamics with double epochs of SAOs, since the folded singularities $q^\mp$ are connected in that regime; the points at which these trajectories ``jump" are estimated in \propref{nodesc} of Appendix~\ref{mechs}. We note that the dynamics on $ \mathcal{Z}^{-} $ differs from that on $ \mc{Z}^{+} $ due to the definition of $ \phi(x,y,z) $ as given in \eqref{kop_params}, in spite of the singular geometry being symmetric; see \figref{relax}(c). For $ k=-4.4$ (blue diamond), the Koper model exhibits mixed-mode dynamics with single epochs of SAOs, as illustrated in panel (e) of \figref{relax}. Finally, for $ k=-5.4 $ (purple triangle), we observe relaxation oscillation; see \figref{relax}(c).

\begin{figure}[ht!]
	\centering
	\begin{subfigure}[b]{0.3\textwidth}
		\centering
		\includegraphics[scale = 0.4]{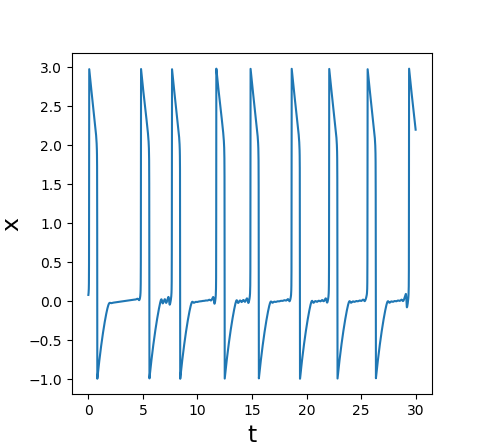}
		\caption{$ \delta  = 0.1 = \mathcal{O}(\sqrt{\varepsilon})$.}
	\end{subfigure}
	~
	\begin{subfigure}[b]{0.3\textwidth}
		\centering
		\includegraphics[scale = 0.4]{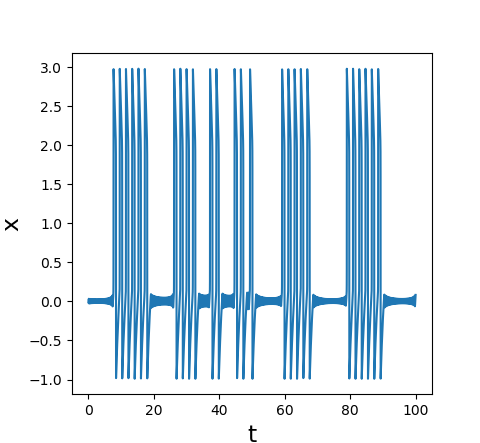}
		\caption{$ \delta  = 0.001 = \mathcal{O}({\varepsilon}^{\frac{3}{2}})$.}
	\end{subfigure}
	~
	\begin{subfigure}[b]{0.3\textwidth}
		\centering
		\includegraphics[scale = 0.4]{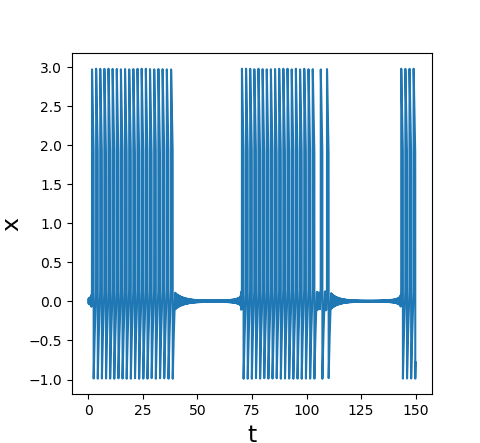}
		\caption{$ \delta  = 0.0003 = \mathcal{O}({\varepsilon}^2)$.}
	\end{subfigure}
	\caption{Mixed-mode time series in the Koper model for $\varepsilon=0.01$ fixed and varying $\delta$: as $ \delta $ decreases, the number of LAOs between SAO segments typically increases; additionally, for these particular parameter values, the model seems to exhibit sector-delayed-Hopf-type dynamics \cite{de2016sector}, as is particularly apparent in panel (b).}
	\figlab{sechof}
\end{figure}

\begin{figure}[h!t]
	\centering
	\includegraphics[scale=0.6]{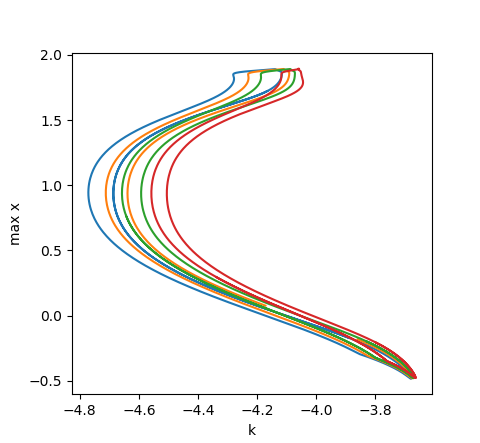}
	\caption{Numerical continuation of periodic orbits in the Koper model with \texttt{auto-07p} \cite{doedel2007auto} for $ \lambda = 1.5 $ and $ \varepsilon = 0.1 = \delta $: one observes coexistence of multiple periodic orbits, as evidenced by the overlap between the corresponding $k$-intervals.}
	\figlab{auto}
\end{figure}

It was shown in \cite{de2016sector} that for $ \delta = \mathcal{O}(\varepsilon^2) $, their prototypical model, Equation~\eqref{prototypical}, can admit MMO trajectories which contain SAO segments that are the product of bifurcation delay alternating with sector-type dynamics. In \figref{sechof}, we present an example that indicates sector-delayed-Hopf-type dynamics in the Koper model; as indicated in \secref{summary}, a crude requirement for the existence of such mixed dynamics is that $ \delta = \mathcal{O}(\varepsilon^c) $ for $ c\geq1 $. We remark that sector-type SAOs cease to exist when $ k=-4 $, as the corresponding regions on $\mathcal{Z}^\mp$ vanish then, 
which follows by substitution of \eqref{kop_params} into \eqref{zcd} below; see Appendix~\ref{mechs} for details. 

We emphasise again that the MMO trajectories described here cannot be viewed, strictly speaking, as perturbations of individual singular cycles, as defined in \secref{singgeom}. Rather, we have shown that if the folded singularities of \eqref{normal} are remote, then there exist $ \varepsilon$ and $\delta$ positive and sufficiently small such that the Koper model exhibits MMOs with single epochs of SAOs; correspondingly, we observe double epochs of SAOs if those singularities are aligned or connected. The above statement is corroborated by numerical continuation, as illustrated in \figref{auto}, where multiple periodic orbits seem to coexist for $ k $, $ \lambda $, $ \varepsilon $, and $ \delta $ fixed. (A similar observation was made in the context of the two-timescale Koper model, i.e., for $\delta =1$ in Equation~\eqref{koper1-c} \cite[Figure~19]{desroches2012mixed}.) An in-depth study of the properties of these periodic orbits in relation to the mixed-mode dynamics of Equation~\eqref{normal} is left for future work. 

\begin{figure}[p!ht]
	\centering
	\begin{subfigure}[b]{0.48\textwidth}
		\centering
		\includegraphics[scale = 0.42]{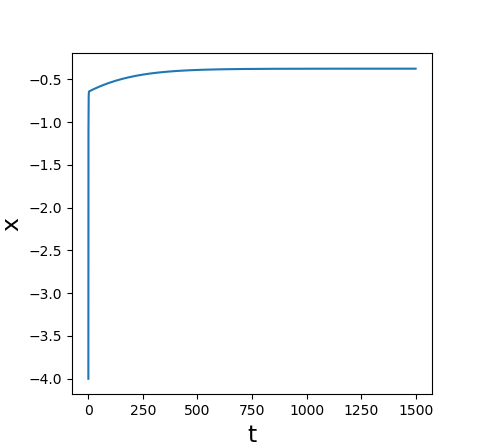}
		\caption{$k=-2.2$.}
	\end{subfigure}
	~
	\begin{subfigure}[b]{0.48\textwidth}
		\centering
		\includegraphics[scale = 0.14]{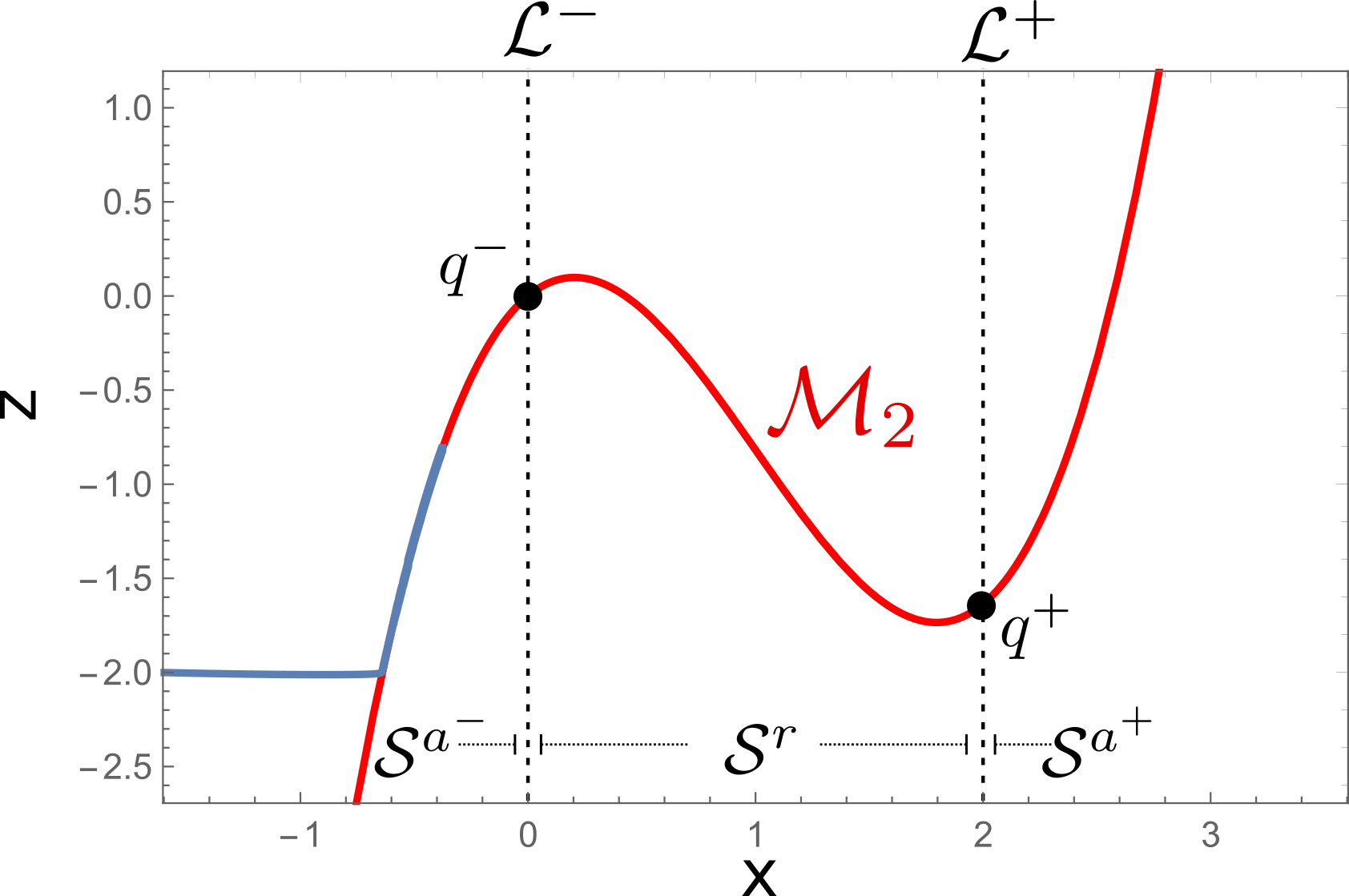}
		\caption{$k=-2.2$.}
	\end{subfigure}
	\\
	\begin{subfigure}[b]{0.48\textwidth}
		\centering
		\includegraphics[scale = 0.42]{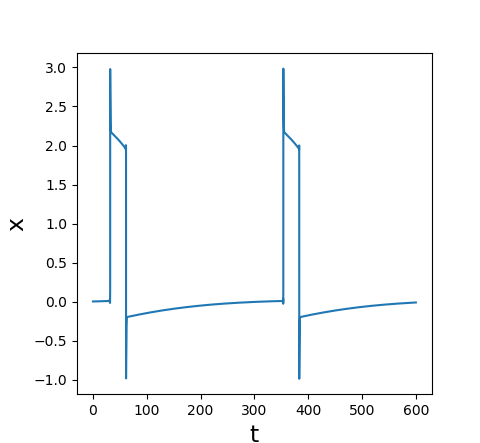}
		\caption{$k=-3.6$.}
	\end{subfigure}
	~
	\begin{subfigure}[b]{0.48\textwidth}
		\centering
		\includegraphics[scale = 0.14]{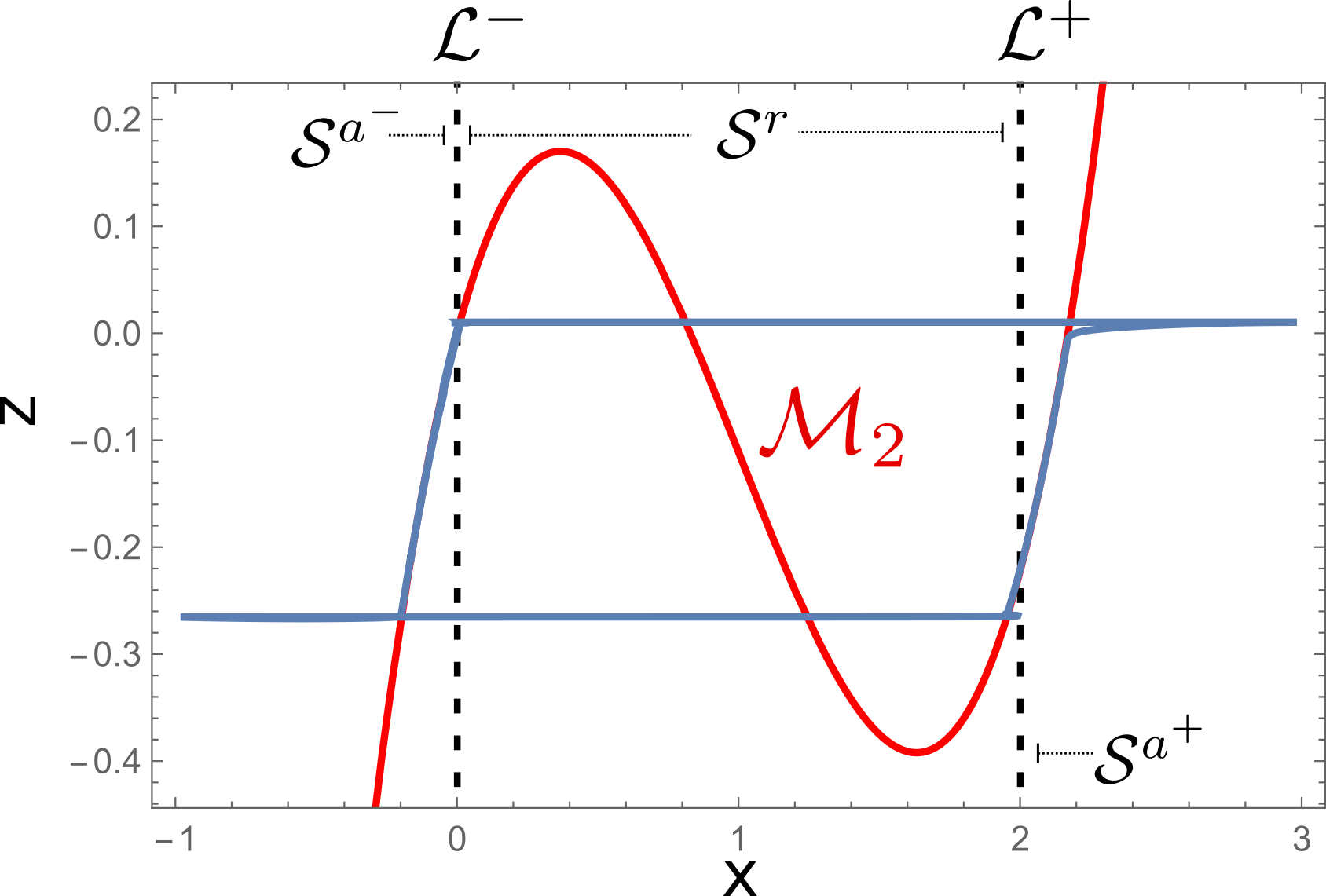}
		\caption{$k=-3.6$.}
	\end{subfigure}
	\\
	\begin{subfigure}[b]{0.48\textwidth}
		\centering
		\includegraphics[scale = 0.42]{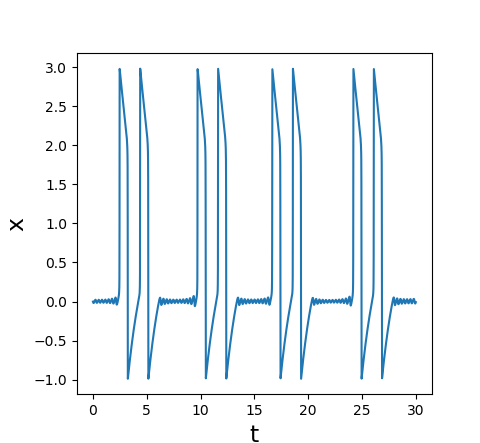}
		\caption{$k=-4.4$.}
	\end{subfigure}
	~
	\begin{subfigure}[b]{0.48\textwidth}
		\centering
		\includegraphics[scale = 0.14]{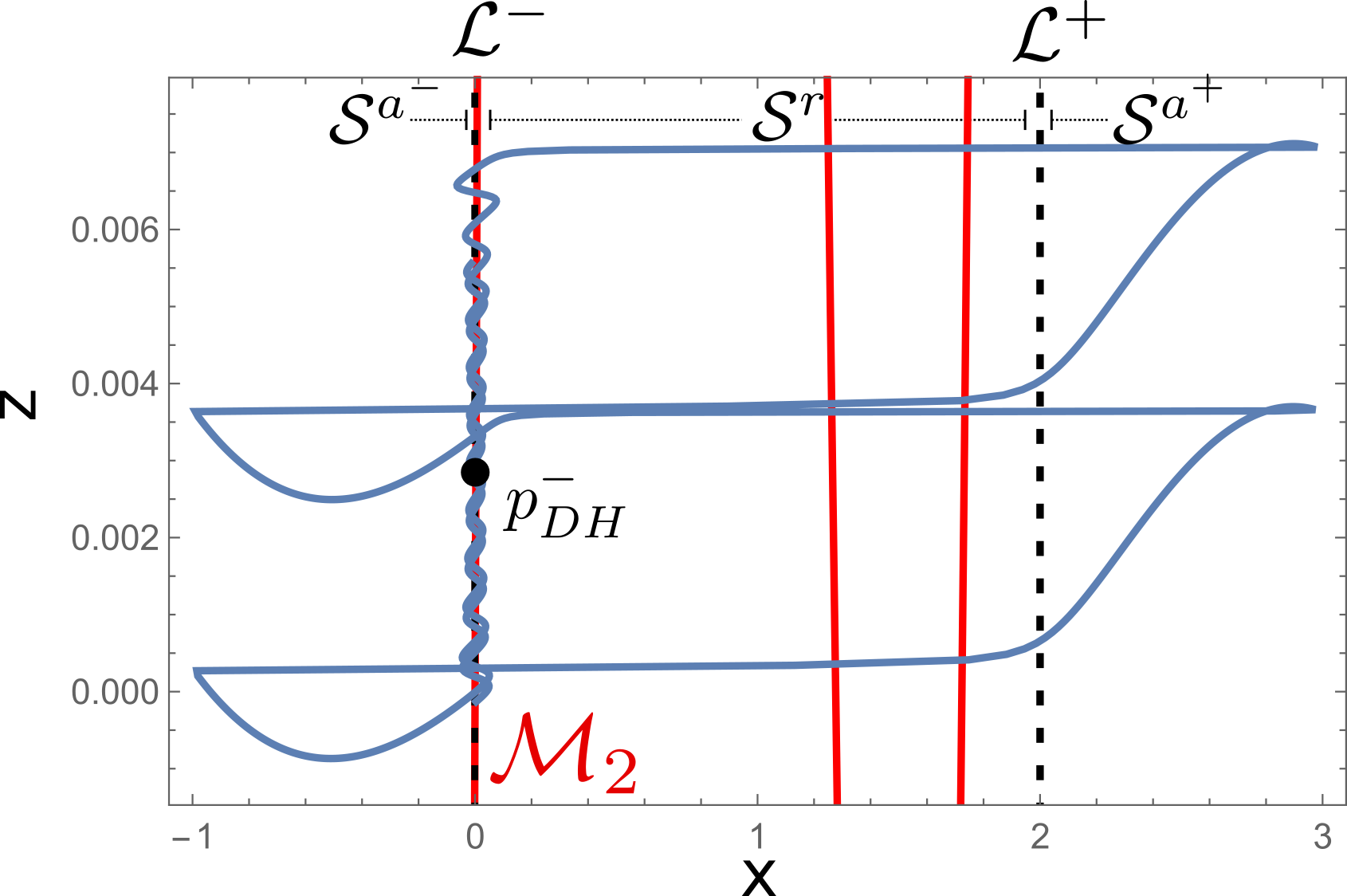}
		\caption{$k=-4.4$.}
	\end{subfigure}
	\\
	\begin{subfigure}[b]{0.48\textwidth}
		\centering
		\includegraphics[scale = 0.42]{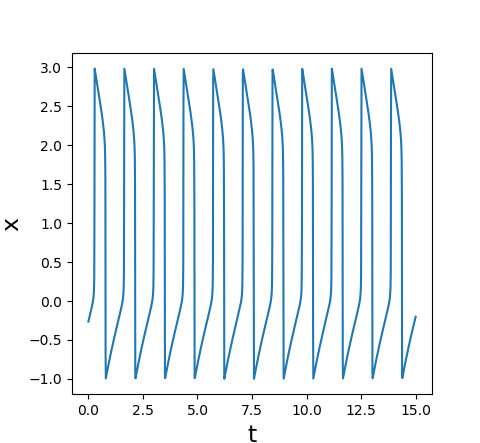}
		\caption{$k=-5.4$.}
	\end{subfigure}
	~
	\begin{subfigure}[b]{0.48\textwidth}
		\centering
		\includegraphics[scale = 0.14]{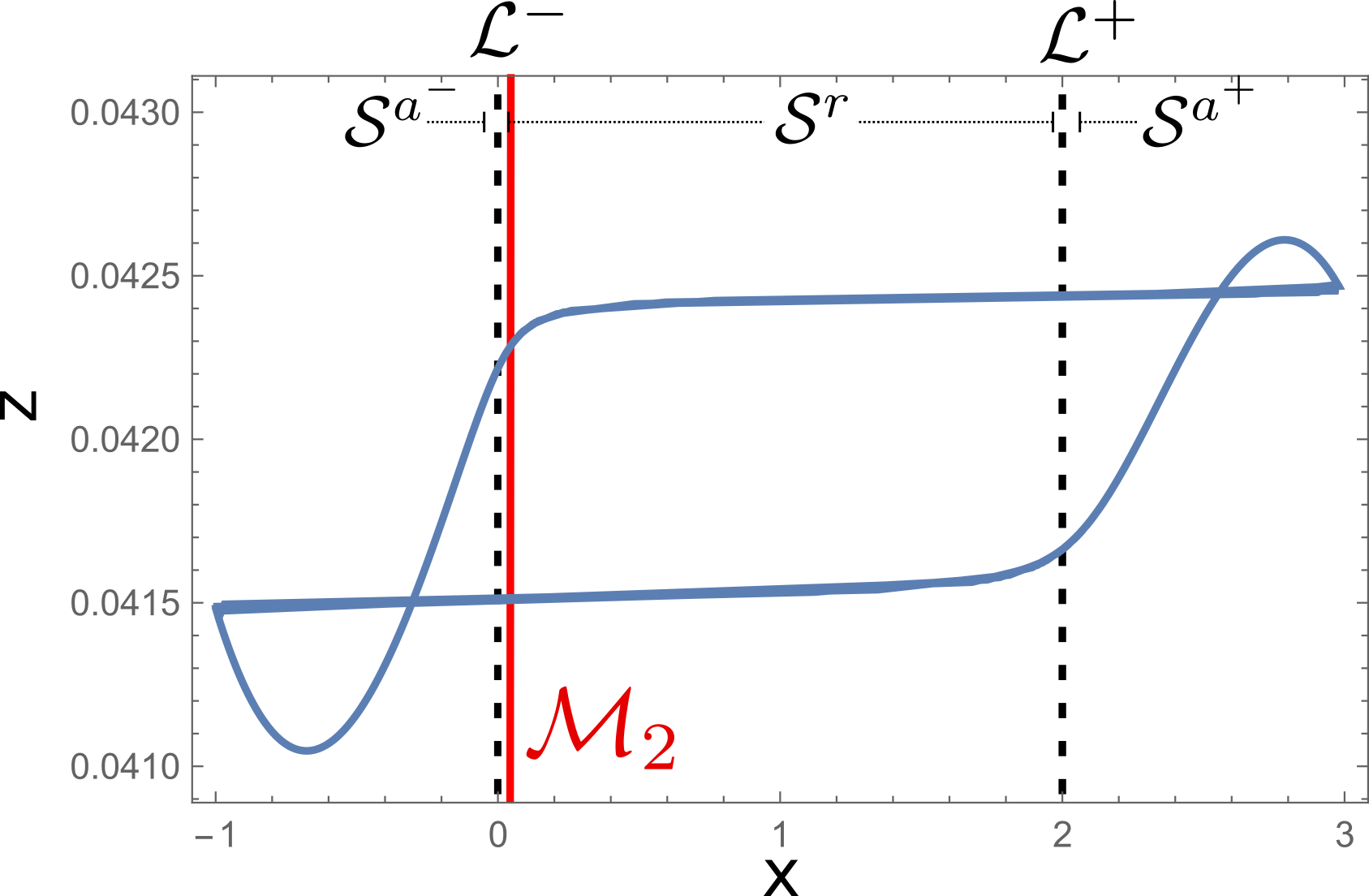}
		\caption{$k=-5.4$.}
	\end{subfigure}
	\caption{Verification of the bifurcation diagram in \figref{kl-plane} for representative choices of $k$, with $\lambda=1.5$ and $\varepsilon=0.01=\delta$ fixed: as $k$ decreases, one observes a transition from (a) steady-state behaviour via (c) MMO trajectories with double epochs of SAOs and (e) single epochs of SAOs to (g) relaxation oscillation. The corresponding singular geometry in phase space is shown in panels (b), (d), (f), and (h), respectively.}
	\figlab{relax}
\end{figure}

\section{Conclusions}
\seclab{conclusion}
In the present article, we have introduced an extended prototypical example of a three-dimensional, three-timescale system, Equation~\eqref{normal}. We have classified the mixed-mode dynamics of that system in dependence of its parameters, thus relating bifurcations of MMO trajectories to the underlying singular geometry. In particular, in \secref{singpert1}, we identified the geometric mechanism that is responsible for the transition from MMOs with single epochs of SAOs to those with double epochs, and we argued that the latter are robust in the three-timescale context. Specifically, we showed that, if the folded singularities of \eqref{normal} are remote, then there exist $\varepsilon$ and $\delta$ sufficiently small such that our system exhibits either MMOs with single epochs of SAOs or two-timescale relaxation oscillation, whereas double epochs of SAOs can be observed if the singularities are aligned or connected; cf.~\propref{relatives}. In \secref{koper}, we demonstrated our results for the Koper model from chemical
kinetics \cite{koper1995bifurcations}, which represents one particular realisation of \eqref{normal}; in particular, we constructed the two-parameter bifurcation diagram in \figref{kl-plane} on the basis of results obtained in \secref{singpert1}, thus classifying in detail the mixed-mode dynamics of the three-timescale Koper model. 

A~posteriori, it is evident that the local dynamics of our extended prototypical model, Equation~\eqref{normal}, is similar to that of the canonical system, Equation~\eqref{canonic}, proposed in \cite{letson2017analysis}; however, due to the absence of a cubic $x$-term in \eqref{canonic-a}, the latter can yield SAO-type dynamics only due to the lack of an LAO-generating global return mechanism. The prototypical system in Equation~\eqref{prototypical}, on the other hand, does not allow for the supercritical manifold to be cubic-like due to the $y$-term being absent in \eqref{prototypical-b} and can hence only exhibit MMOs with single epochs of SAOs, as opposed to our extended Equation~\eqref{normal}; recall \propref{relatives}. Hence, we postulate that the extended prototypical model in Equation~\eqref{normal} represents the simplest general example within that given class of systems that can encapsulate the geometric mechanism described in \secref{singpert1}.
Additionally, we remark that the singular geometry considered here is relatively specific due to its symmetry properties. In particular, we have not considered explicitly the scenario where the fold points $p^\mp$ of $\mathcal{M}_2$ can cross the fold lines $\mathcal{L}^\mp$ of $\mathcal{M}_1$; recall, in particular, panels (b) and (e) of \figref{m2bif}. While that scenario is not realised in the Koper model, Equation~\eqref{koper1}, it has been shown to give rise to interesting local dynamics through the interaction of $p^\mp$ with $\mathcal{L}^\mp$; a recent, relevant example can be found in \cite{desroches2018spike}.

Our analysis in \secref{singpert1} shows that, in parameter regimes where both $ \mathcal{M}_1 $ and $ \mathcal{M}_2 $ are normally hyperbolic, standard GSPT \cite{fenichel1979geometric} implies that an iterative reduction of timescales can be applied. In the fully perturbed Equation~\eqref{normal} with $\varepsilon$ and $\delta$ sufficiently small, it follows that the manifolds $ \mathcal{Z}_{\varepsilon\delta}^{\mp,r} $ lie $ \mathcal{O}(\delta) $-close to their unperturbed counterparts $ \mathcal{Z}^{\mp,r}$, respectively, since $ \mathcal{Z}^{\mp,r}_{\varepsilon0}$ are $\varepsilon$-independent. Since, moreover, the manifolds $ \mathcal{S}_{\varepsilon0}^{a,r} $ lie $ \mathcal{O}(\varepsilon) $-close to $ \mathcal{S}^{a,r}$ \cite{fenichel1979geometric}, any fibers of $\mathcal{Z}_{\varepsilon\delta}^{\mp,r}$ that lie on $\mathcal{S}_{\varepsilon\delta}^{a,r}$ are $ \mathcal{O}(\varepsilon+\delta) $-close to $ \mathcal{S}^{a,r}$. (That estimate is in disagreement with \cite{cardin2017fenichel}; however, we note that, away from $\mathcal{Z}_{\varepsilon\delta}^{\mp,r}$, $\mathcal{S}_{\varepsilon\delta}^{a,r}$ are $\mathcal{O}(\varepsilon)$-close to $ \mathcal{S}^{a,r}$.) Under \asuref{redflow}, trajectories that are attracted to $ \mathcal{Z}_{\varepsilon\delta}^{\mp,r} $ follow the slow flow of \eqref{classy-des} and potentially undergo SAOs. In the context of \eqref{normal}, the mechanisms that generate these SAOs are ``bifurcation delay'' \cite{krupa2010local,de2016sector,letson2017analysis} and ``sector-type'' dynamics \cite{krupa2008mixed,de2016sector}; see Appendix~\ref{mechs} for details.

With regard to regions where normal hyperbolicity of $ \mathcal{M}_1 $ is lost, we reiterate that the dynamics of Equation~\eqref{normal} combines features of two-timescale slow-fast systems with either two slow variables and a fast one, or one fast variable and two slow ones. As shown in \secref{singpert1}, the corresponding mechanisms hence coexist and interact, giving rise to complex local dynamics in the vicinity of the fold lines $ \mathcal{L}^\mp $ in \eqref{normal}. We briefly sketched the implications of that interaction; in particular, we related the emergence of canard-type SAOs to the perturbation of an integrable system \cite{krupa2008mixed}. A more rigorous description of the resulting near-integrable system in the context of the Koper model, Equation~\eqref{koper1}, is part of work in progress. Of particular interest here is the investigation of Shilnikov-type homoclinic phenomena, as well as the further classification of MMOs with single epochs of SAOs; specifically, we conjecture that the bifurcation diagram in \figref{kl-plane} may be refined, in that one can identify regions of chaotic mixed-mode dynamics in dependence of the various parameters in the model, as well as of the ratio of $ \varepsilon $ and $ \delta $.

We emphasise that, strictly speaking, the MMO trajectories described in \secref{singpert1} cannot be considered as perturbations, for $ \varepsilon$ and $\delta$ positive, of the individual singular cycles constructed in \secref{singgeom}. Rather, the latter determine the qualitative properties of the former, for $\varepsilon$ and $\delta$ sufficiently small, as is evident from \figref{auto} in the context of the three-scale Koper model, where several periodic orbits seem to coexist for a given choice of $k$, $\lambda$, $\varepsilon$, and $\delta$. 

Finally, we emphasise that the geometric mechanism described in this article extends beyond the Koper model from chemical kinetics studied in \secref{koper}. One prominent example of a rich multiple-scale system that features similar geometric properties as our prototypical model, Equation~\eqref{normal}, is provided by a three-dimensional reduction of the famous Hodgkin-Huxley equations from mathematical neuroscience \cite{rubin2007giant},
\begin{subequations}\eqlab{hh}
\begin{align}
\epsilon\dot{v} &= \bar{I}-\lp v-\bar{E}_{Na}\rp m_\infty(v)^3h-\bar{g}_k\lp v-\bar{E}_k	\rp n^4-\bar{g}_l\lp v-\bar{E}_L\rp, \eqlab{hh-v}\\ 
\dot{h}&= \frac{1}{\tau_h{t}_h\lp v\rp} \lp h_\infty \lp v\rp - h\rp\\
\dot{n}&= \frac{1}{\tau_n{t}_n\lp v\rp} \lp n_\infty \lp v\rp - n\rp, \eqlab{hh-n}
\end{align}
\end{subequations}
where $v$ is the fast variable and $(h,n)$ are the slow ones.
Here, the functions $\frac1{t_h(v)}$, $\frac1{t_n(v)}$, and $x_\infty(v)$ ($x=m,h,n$) illustrated in \figref{xinf} are defined as in \cite{rubin2007giant}, as are the values of the various parameters in Equation~\eqref{hh}; see also \cite{doi2001complex}.
In particular, following \cite{doi2001complex}, we may set $\tau_n=1$ in \eqref{hh} and assume that $\tau_h\gg1$ is sufficiently large; alternatively, we may take $\tau_h=1$ and  $\tau_n\gg 1$. In either case, we obtain a three-timescale system,
where $v$ is the fast variable, with $n$ and $h$ being intermediate or slow, respectively.

\begin{figure}[ht!]
	\centering
	\begin{subfigure}[b]{0.45\textwidth}
		\centering
		\includegraphics[scale = 0.35]{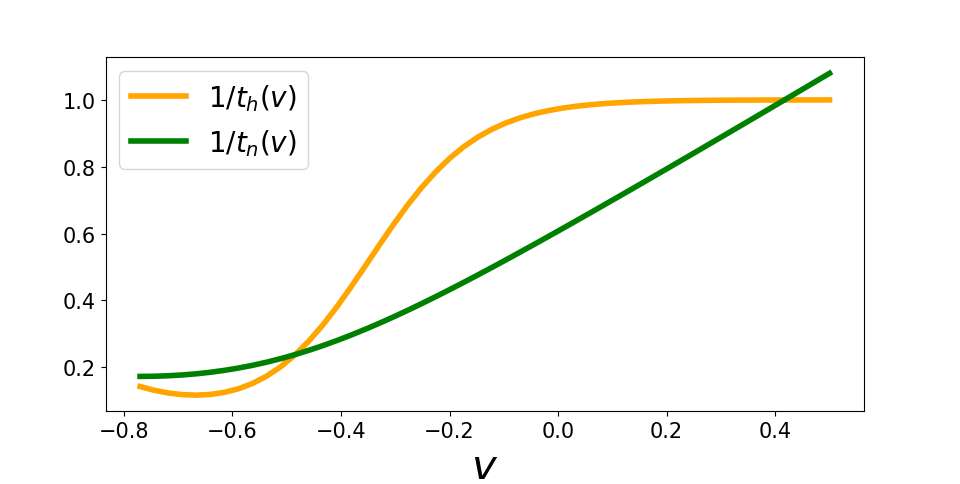}
		\caption{$\frac1{t_h(v)}$ and $\frac1{t_n(v)}$.}
	\end{subfigure}
	~
	\begin{subfigure}[b]{0.45\textwidth}
		\centering
		\includegraphics[scale = 0.35]{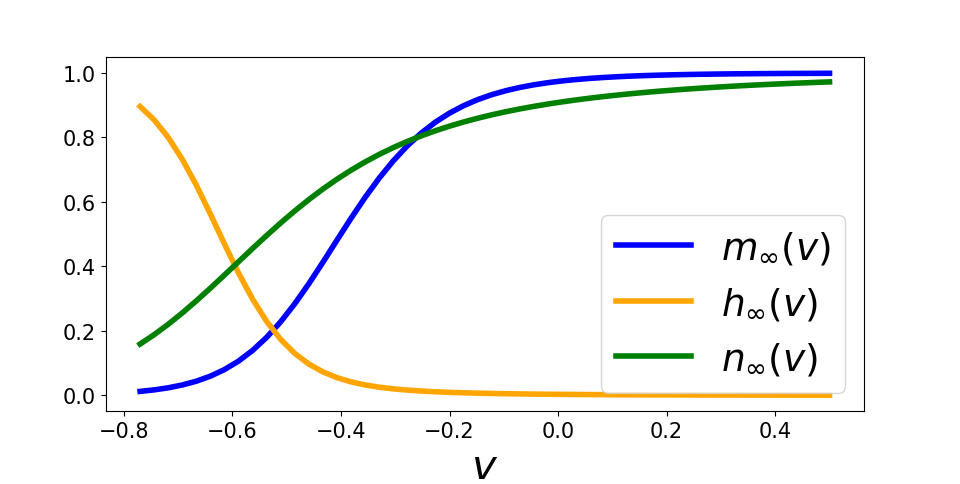}
		\caption{$x_\infty(v)$ ($x=m,h,n$).}
	\end{subfigure}
	\caption{Graphs of the nonlinear functions on the right-hand sides of the Hodgkin-Huxley equations in \eqref{hh}.}
	\figlab{xinf}
\end{figure}

\begin{figure}[ht!]
	\centering
	\begin{subfigure}[b]{0.45\textwidth}
		\centering
		\includegraphics[scale = 0.4]{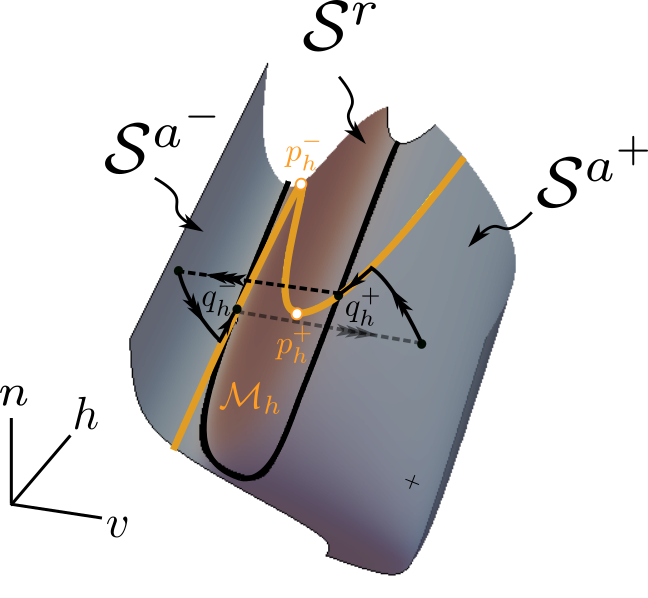}
		\caption{$\tau_h\gg 1$ and $\tau_n=1$.}
	\end{subfigure}
	~
	\begin{subfigure}[b]{0.45\textwidth}
		\centering
		\includegraphics[scale = 0.4]{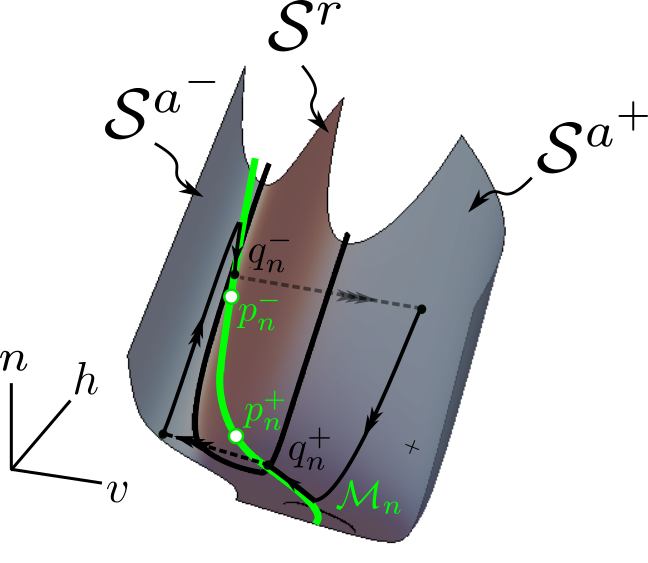}
		\caption{$\tau_h=1$ and $\tau_n\gg 1$.}
	\end{subfigure}
	\caption{The critical and supercritical manifolds of the three-dimensional, three-timescale Hodgkin-Huxley model, Equation~\eqref{hh}, when either $h$ or $n$ is taken to be the slowest variable; see panels (a) and (b), respectively.}
	\figlab{crimaHH}
\end{figure}

\figref{crimaHH} indicates that the resulting singular geometry of Equation~\eqref{hh} is analogous to that of our extended prototypical example, Equation~\eqref{normal}; recall \figref{singcycle}. MMO trajectories can hence again be constructed as outlined in \secref{singpert1}, by combining segments that evolve on different timescales. Upon variation of the parameter $\bar I$ -- the (rescaled) applied current in the Hodgkin-Huxley formalism -- transitions between MMOs with different qualitative properties occur via a mechanism that is similar to the one described for Equation~\eqref{normal} in \secref{singgeom} and \secref{singpert1}. 
For an in-depth geometric analysis of a novel, global three-dimensional reduction of the multiple-timescale Hodgkin-Huxley equations, rather than of Equation~\eqref{hh}, the reader is referred to the upcoming article \cite{hhgspt2019}.

\section*{Acknowledgements}

The authors thank Martin Krupa and Martin Wechselberger for their critical reading of previous versions of the manuscript and for constructive feedback, as well as for insightful discussions and relevant references.

PK was supported by the Principal's Career Development Scholarship for PhD Studies of the University of Edinburgh.

\appendix

\section{SAO-generating mechanisms}
\numberwithin{equation}{section}
\setcounter{equation}{0}
\label{mechs}

In this appendix, we briefly discuss the local, SAO-type dynamics of our prototypical model, Equation~\eqref{normal}; specifically, we give an overview of two SAO-generating mechanisms -- bifurcation delay and sector-type dynamics -- within the framework of \eqref{normal}.

\subsection{Local dynamics and SAOs}\seclab{local}

We begin by discussing the emergence of SAOs in a vicinity of $ \mathcal{L}^\mp $ in \eqref{normal} when trajectories are attracted to $ \mc{Z}^{\mp} $, respectively; we focus on describing the properties of $ \mathcal{Z}^{-} $ close to $ \mathcal{L}^- $ here, as the description of $ \mc{Z}^{+} $ near $ \mathcal{L}^+ $ is analogous.

We first consider the partially perturbed fast Equation~\eqref{norm12fast} with $\varepsilon$ sufficiently small and $ \delta =0 $:
\begin{subequations}\eqlab{fastsub}
	\begin{align}
	{x}' &= -y + f_2x^2+f_3x^3, \\
	{y}' &= \varepsilon \lp \alpha x+\beta y{-z}\rp, \\
	{z}' &=  0.
	\end{align}
\end{subequations}
By standard GSPT \cite{fenichel1979geometric,krupa2001extending}, we can define slow manifolds $ \mathcal{S}_{\varepsilon0}^{a,r} $ for \eqref{fastsub} as surfaces that are foliated by orbits within $\{z=z_0\}$, with $z_0$ constant. Since the steady states of \eqref{fastsub} correspond to portions of the supercritical manifold $ \mathcal{M}_2 $, it follows that $ \mathcal{Z}^{\mp,r}_{\varepsilon0} \equiv \mathcal{Z}^{\mp,r}$, i.e., that the geometry of $ \mathcal{Z}^{\mp,r}_{\varepsilon0} $ is, in fact, $\varepsilon$-independent. However, since it will become apparent that the stability properties of $ \mathcal{Z}^{\mp,r}_{\varepsilon0} $ do depend on $ \varepsilon $, we will not suppress the $ \varepsilon $-subscript in our notation.

\begin{figure}[ht!]
\centering
\includegraphics[scale=0.15]{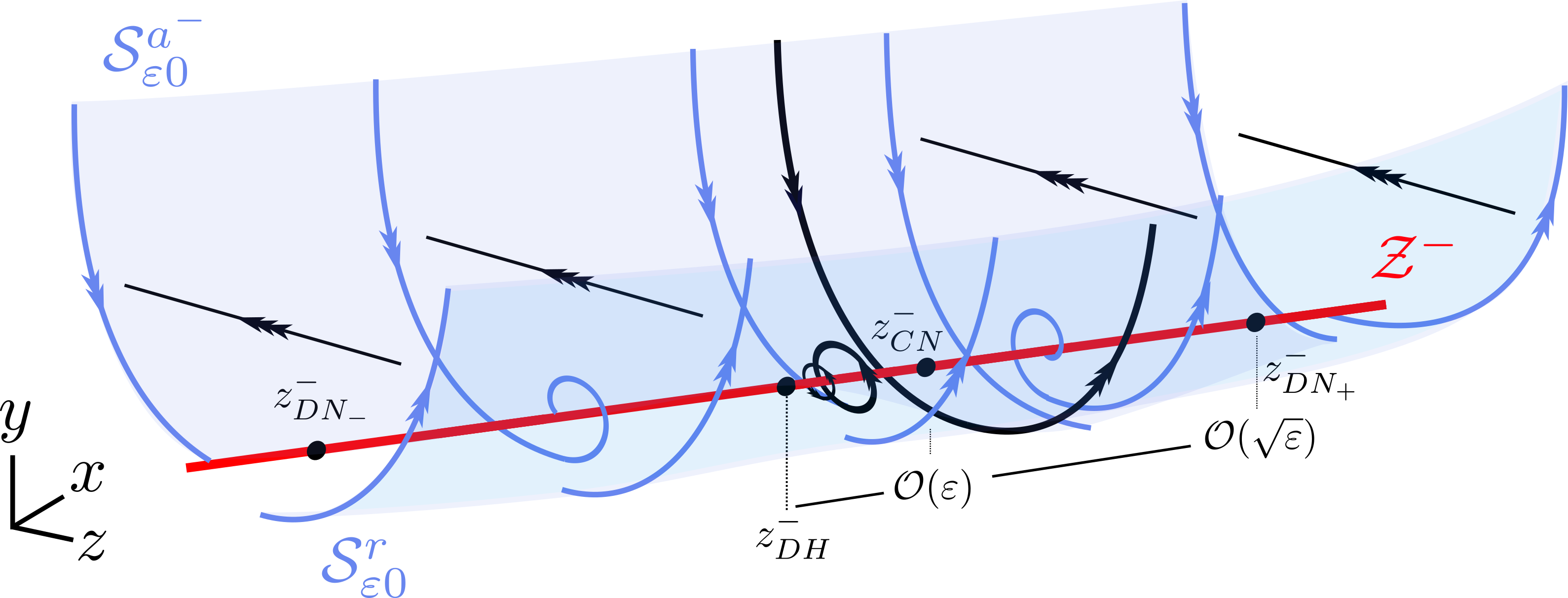}
\caption{Stability of the supercritical manifold $\mathcal{M}_2$ on the various portions of $\mathcal{Z}^{-}$, for $\varepsilon$ sufficiently small and $ \delta=0 $: at $z_{DN_\mp}^-$, the real eigenvalues of the linearisation of Equation~\eqref{fastsub} about $ \mathcal{M}_2 $ in \eqref{eigens} become complex, with a corresponding change from nodal to focal attraction or repulsion, and vice versa; at $ z_{DH}^- $, a Hopf bifurcation occurs which gives rise to small-amplitude periodic orbits. These orbits cease to exist at $ z_{CN}^- $, where a connecting trajectory between $ \mathcal{S}_{\varepsilon0}^{a^-}$ and $ \mathcal{S}_{\varepsilon0}^{r} $ is found. We note that the corresponding $z$-interval is of width $\mathcal{O}(\varepsilon)$, while the focal region is $\mathcal{O}(\sqrt{\varepsilon})$ wide overall.}
\figlab{M2bif}
\end{figure}

For $\varepsilon$ sufficiently small, Equation~\eqref{fastsub} undergoes a Hopf bifurcation at a point $ p_{DH}^- = \big( x_{DH}^-, y_{DH}^-,$ $z_{DH}^-\big)$; the periodic orbits that arise in that bifurcation cease to exist at $ z_{CN}^- $, where a connecting trajectory between the manifolds $ \mathcal{S}_{\varepsilon0}^{a^-}$ and $ \mathcal{S}_{\varepsilon0}^{r} $ is found. In other words, $\mathcal{S}_{\varepsilon0}^{a^-} $ and $ \mathcal{S}_{\varepsilon0}^{r} $ intersect transversely within the hyperplane $\mathcal{P}_{CN}^- : \{z=z_{CN}^-\}$ which lies $ \mathcal{O}(\varepsilon) $-close to $ p_{DH}^- $ in the $ z $-direction. Moreover, two degenerate nodes $ p_{DN_\mp}^- $ are located on $ \mathcal{M}_2 $ around $ p_{DH}^- $ at an $ \mathcal{O}(\sqrt{\varepsilon}) $-distance; see \figref{M2bif}. The asymptotics (in $\varepsilon$) of these objects is summarised below.
\begin{lemma}
A Hopf bifurcation of Equation~\eqref{fastsub} occurs at $p_{DH}^-:\ \lp x_{DH}^-, y_{DH}^-, z_{DH}^-\rp\in\mathcal{Z}_{\varepsilon0}^{-}$, where 
	\begin{gather}
	x_{DH}^-=-\frac{\beta}{2f_2}\varepsilon + \mathcal{O}(\varepsilon^2),\quad 
	y_{DH}^-=\frac{\beta^2}{4f_2}\varepsilon^2 + \mathcal{O}(\varepsilon^3),
	\quad\text{and}\quad
	z_{DH}^-={-\frac{\alpha\beta}{2f_2}}\varepsilon + \mathcal{O}(\varepsilon^2). \eqlab{DHB}
	\end{gather}
Two degenerate nodes $p_{DN_{\mp}}^-$ are located at
	\begin{gather}
	\begin{gathered}
	x_{DN_{\mp}}^-=\mp\lp\frac{\sqrt{\alpha}}{f_2}\sqrt{\varepsilon}+\frac{\beta}{2f_2}\varepsilon\rp+\mathcal{O}(\varepsilon^{\frac32}),\quad y_{DN_{\mp}}^-=\frac{\alpha}{f_2}\varepsilon+\mathcal{O}(\varepsilon^{\frac32}),\quad\text{and}\\ z_{DN_{\mp}}^-={\mp}\frac{{\alpha^{\frac32}}}{f^2_2}\sqrt{\varepsilon}{+}\frac{\alpha\beta}{f_2}\lp {1}\mp\frac{1}{2}\rp\varepsilon+\mathcal{O}\big(\varepsilon^{\frac32}\big),
	\end{gathered}
	\end{gather}
while a canard trajectory is contained in the hyperplane $\mathcal{P}_{CN}^- : \{z=z_{CN}^-\}$, with
\begin{gather}
\begin{gathered}
z_{CN}^-={z_{DH}^-}+{\alpha\beta\frac{ 5f_2-3(1-\alpha f_3)}{4\lp 1+f_2\rp f_2}}\varepsilon+\mathcal{O}(\varepsilon^2).
\end{gathered}
\eqlab{zcd}
\end{gather}
\lemmalab{delayedH}
\end{lemma}
\begin{proof}
The hyperplane $\{z=z_{CN}^-\}$, which contains the transverse intersection between $ \mathcal{S}^{a^-}_{\varepsilon0} $ and $ \mathcal{S}^{r}_{\varepsilon0} $, can be obtained by Melnikov-type calculations; see \cite{krupa2001extending, letson2017analysis}. The remaining estimates follow by considering the Jacobian matrix of the linearisation of \eqref{fastsub} along $\mathcal{M}_2$,
	\begin{align}
	J= \begin{pmatrix}
	x(2f_2+3f_3x) & -1 \\
	\varepsilon\alpha & \varepsilon\beta 
	\end{pmatrix},
	\eqlab{jac}
	\end{align}
the eigenvalues of which are
\begin{align}\eqlab{eigens}
\nu_{1,2}=\frac{1}{2}\left[\beta\varepsilon+2{f_2}x+3{f_3}x^2 \pm\sqrt{\big(\beta \varepsilon+2 {f_2} x+3 {f_3} x^2\big)^2-4 \big(\alpha\varepsilon+2\beta\varepsilon{f_2}x+3\beta\varepsilon {f_3}x^2\big)}\right].
\end{align}
\end{proof}

\begin{remark}
The Hopf bifurcation at $p_{DH}^-$ is ``inherited'' from the fact that $\mathcal{M}_2$ and $\mathcal{L}^-$ intersect in the folded singularity $q^-$: for $ \varepsilon=0 =\delta$, the trace of the Jacobian $J$ vanishes at that point.
\end{remark}

\begin{remark}
The estimate in \eqref{zcd} is a generalisation of the corresponding expression in \cite{letson2017analysis} for their canonical system, Equation~\eqref{canonic}; the $f_3$-dependence of \eqref{zcd} implies that the cubic $x$-terms in our Equation~\eqref{normal-a} do, in fact, contribute to the local dynamics.
\end{remark}


Motivated by \lemmaref{delayedH}, we introduce the following notation: for $\delta=0$, we define the intervals
\begin{gather}
\mathcal{I}_{\rm nod} = \lp -\infty, z_{DN_-}^- \rp,\quad 
\mathcal{I}_{\rm foc} = \lp z_{DN_-}^-, z_{DH}^-\rp, \quad\text{and}\quad\mathcal{I}_{\rm can}= 
\lp \min\lb z_{DH}^-,z_{CN}^-\rb, \max \lb z_{DH}^-,z_{CN}^-\rb \rp.
\eqlab{intervls}
\end{gather}
Then, it follows that
\begin{enumerate}
\item the manifold $ \mathcal{S}^{a^-}_{\varepsilon0} $ connects to $ \mathcal{Z}^{-} $ for $ z<\min \lb z_{DH}^-,z_{CN}^-\rb $, while $ \mathcal{S}^{r}_{\varepsilon0} $ connects to $ \mathcal{Z}^{-} $ for $ z>\max \lb z_{DH}^-,z_{CN}^-\rb $;
\item for $ f_2<\frac{3}{5}\lp 1-\alpha f_3\rp $,  i.e., for $ z_{CN}^->z_{DH}^- $, the Hopf bifurcation at $ p_{DH}^- $ is supercritical, with the resulting periodic orbits the $\omega$-limit sets of trajectories on $ \mathcal{S}^{a^-}_{\varepsilon0} $;
\item for $ f_2>\frac{3}{5}\lp 1-\alpha f_3\rp $, i.e., for $ z_{CN}^-<z_{DH}^- $, the Hopf bifurcation at $ p_{DH}^- $ is subcritical, with the resulting periodic orbits the $\alpha$-limit sets of trajectories on $ \mathcal{S}^{r}_{\varepsilon0} $.
\end{enumerate}

The corresponding geometry is illustrated in \figref{M2bif}; we emphasise that analogous objects $p_{DH}^+$, $\mathcal{P}_{CN}^+$, and $ p_{DN_{\pm}}^+ $, which are located symmetrically to the above, exist on $ \mc{Z}^{+} $. 

We define the \textit{canard point} $p_{CN}^- = (x_{CN}, y_{CN}, z_{CN})$ by
\begin{align*}
    p_{CN}^- = \mc{P}_{CN}^- \cap \mc{Z}^{-}.
\end{align*}
It has already been pointed out in \cite{letson2017analysis} that $p_{CN}^-$ and the Hopf point $p_{DH}^-$ on $\mathcal{Z}^{-}$ collapse to the origin in the limit of $\varepsilon=0$; correspondingly, the origin is referred to as the ``canard delayed Hopf singularity'' in the double singular limit of $ \varepsilon=0=\delta $. As a result, the folded singularity at $ q^- $ displays characteristics of both a Hopf point -- in that the trace of the Jacobian in \eqref{jac} vanishes -- and a canard point -- in that $ \mathcal{S}^{a^-} $ and $ \mathcal{S}^r $ meet along a fold. Moreover, we remark that an ``incomplete'' canard explosion occurs at $ z_{CN}^- $ in Equation~\eqref{canonic}, as the corresponding intermediate problem has two equilibria, with the equilibrium corresponding to $ \mathcal{Z}^r $ being a saddle forming a homoclinic connection to itself; see \cite{letson2017analysis} for details. On the other hand, Equation \eqref{normal} could feature either complete or incomplete canard explosion, depending on the relative position of $q^-$ and $p^+$; the implications for the global dynamics of the system are currently being investigated.

We briefly describe the associated two mechanisms -- bifurcation delay and sector-type dynamics -- in the following; we remark that the former is common in two-timescale systems with two fast variables, while the latter typically occurs in two-timescale systems with two slow variables. Therefore, the coexistence of these mechanisms in three-timescale systems is due to the fact that such systems can simultaneously be viewed as having two fast and one slow variables, as well as as one fast and two slow variables. {(For four-dimensional two-timescale systems with two fast and two slow variables, that interplay has been documented in \cite{curtu2011interaction}.)}

\subsection{Bifurcation delay} 

Bifurcation delay is typically encountered in two-timescale systems with two fast variables and one slow variable. In the context of Equation~\eqref{normal}, it is realised when trajectories are attracted to $ {\mathcal{Z}_{\varepsilon\delta}}\big\lvert_{\mathcal{I}_{\rm nod}+\mathcal{O}(\delta)} $ or $ {\mathcal{Z}_{\varepsilon\delta}}\big\lvert_{\mathcal{I}_{\rm foc}+\mathcal{O}(\delta)} $; recall \eqref{intervls} and \figref{M2bif}. Following the slow flow on $ {\mathcal{Z}_{\varepsilon\delta}^{-}}$, trajectories experience a delay in being repelled from $ {\mathcal{Z}_{\varepsilon\delta}^{-}} $ when crossing the Hopf bifurcation point $ p_{DH}^- $, as the accumulated contraction to $ {\mathcal{Z}_{\varepsilon\delta}^{-}}$ needs to be balanced by the total expansion from $ {\mathcal{Z}_{\varepsilon\delta}^{-}} $ \cite{krupa2010local}. Specifically, given some point 
$p_{\rm in}= (x_{\rm in},y_{\rm in},z_{\rm in})$ in an $\mathcal{O}(\delta)$-neighbourhood of ${\mathcal{Z}_{\varepsilon\delta}^{-}} $, one obtains the $x$-coordinate of the corresponding point $p_{\rm out}$ where the given trajectory through $p_{\rm in}$ exits an $\mc{O}(\delta)$-neighbourhood of ${\mathcal{Z}_{\varepsilon\delta}^{+}} $ from
\begin{align}
\int_{x_{\rm in}}^{x_{\rm out}} \frac{\Re \lb \nu_{1,2}\lp x\rp\rb}{\mu+\phi\lp x,F(x),G(x)\rp}\tn{d}x =0; 
\eqlab{inout}
\end{align}
here, $\nu_{1,2}$ are the eigenvalues of the linearisation of Equation~\eqref{fastsub} about $ {\mathcal{Z}_{\varepsilon\delta}^{-}} $, as defined in \eqref{eigens}. 
Trajectories that are attracted to $ \mathcal{Z}_{\varepsilon\delta}\big\lvert_{\mathcal{I}_{\rm foc}} $ typically exhibit ``dense" SAOs with initially decreasing and then increasing amplitude; see panel (e) of \figref{relax} for an illustration in the context of the Koper model, Equation~\eqref{koper1}. By contrast, trajectories that are attracted to $ \mathcal{Z}_{\varepsilon\delta}\big\lvert_{\mathcal{I}_{\rm nod}} $ are characterised by very few SAOs that are followed by a large excursion; cf.~\figref{relax}(c).
	
The case where trajectories enter the focally attracting region $\mathcal{I}_{\rm foc}$ is naturally studied in the ``rescaling chart" $ \kappa_2 $ which is introduced as part of a blow-up analysis in \cite{krupa2001extending,letson2017analysis}, since that region is bounded by the degenerate nodes $p_{DN_\mp}^-$ and, thus, of width $ \mathcal{O}(\sqrt{\varepsilon}) $. 
In that case, the eigenvalues $\nu_{1,2}$ in \eqref{eigens} are complex conjugates, which implies that the corresponding trajectory of \eqref{normal} undergoes damped oscillation towards $ \mathcal{Z}_{\varepsilon\delta}^{-} $.
	
On the other hand, when trajectories enter the nodally attracting region $\mathcal{I}_{\rm nod}$, the corresponding entry point is typically $ \mathcal{O}(\varepsilon^c) $ away from the folded singularity $q^-$, with $ c<1/2 $. One may therefore refer to the unscaled system, Equation~\eqref{normal}, for the study of that case. The eigenvalues $\nu_{1,2}$ in \eqref{eigens} correspond to strong and weak eigendirections: specifically, for $ z<z_{DN_-}^- $, the eigenvalue $ \nu_1 $ represents the weak eigendirection, while the eigenvalue $ \nu_2 $ corresponds to the strong eigendirection; that correspondence is reversed for $ z>z_{DN_-}^+ $. Due to the hierarchy of timescales in \eqref{normal}, trajectories are first attracted to $ \mathcal{S}_{\varepsilon\delta}^{a^-} $ and then to $ \mathcal{Z}_{\varepsilon\delta}^{-} $. Therefore, for initial conditions $ (x,y,z)\in \mathcal{S}_{\varepsilon\delta}^{a^-}$, trajectories approach $ \mathcal{Z}_{\varepsilon\delta}^{-} $ along the weak eigendirection, while for $ (x,y,z)\in \mathcal{S}_{\varepsilon\delta}^{r}$, trajectories are repelled from $ \mathcal{Z}_{\varepsilon\delta}^{-} $ along the strong eigendirection. It is hence reasonable to balance the accumulated contraction and expansion using solely $ \nu_1 $ in \eqref{inout}. Since the accumulated contraction on the intermediate timescale has to be balanced by expansion on the fast timescale, we have the following:
\begin{proposition}[\cite{hayes2016geometric,krupa2010local}]
Assume that \asuref{albeta} and \asuref{redflow} hold, and consider $ \lp x_{\rm in},y_{\rm in},z_{\rm in}\rp\in {\mathcal{Z}_{\varepsilon\delta}^{-}}\lvert_{\mathcal{I}_{\rm foc}\cup\mathcal{I}_{\rm nod}} $. Then, the exit point $ \lp x_{\rm out},y_{\rm out},z_{\rm out}\rp $ that is defined by \eqref{inout} satisfies
\begin{align*}
x_{\rm out} < x_{DN_+}^{-}+o(1),\quad y_{\rm out} < y_{DN_+}^{-}+o(1),\quad\text{and}\quad z_{\rm out} < z_{DN_+}^{-}+o(1).
\end{align*}
\proplab{nodesc}
\end{proposition}
\begin{remark}
The estimates on the entry point $p_{\rm out}$ in \propref{nodesc} can be refined under the additional assumption that the slow flow of Equation~\eqref{normal} is constant, i.e., that $ \phi(x,y,z) = 0 $: as in \cite{de2016sector}, for $p_{\textnormal{in}}\in \mc{I}^{\textnormal{in}}_{\textnormal{foc}}$ it then follows from \eqref{inout} that   $ x_{\rm out} = x_{DH}-x_{\rm in} $.
\end{remark}

\begin{remark}
In \cite{letson2017analysis}, for constant slow flow of Equation~\eqref{normal}, i.e., for $ \phi(x,y,z) = 0 $, the weak contraction towards ${\mathcal{Z}_{\varepsilon\delta}^{-}}$ is balanced by the weak expansion therefrom via
\begin{align*}
\int_{x_{\rm in}}^{x_{DN}^-} \Re \lb \nu_{1}\rb\tn{d}x +\int_{x_{DN}^+}^{x_{\rm out}} \Re \lb \nu_{2}\rb\,\tn{d}x=0.
\end{align*}
In that context, the fold point $p^-$ was in fact identified as the buffer point at which trajectories have to leave ${\mathcal{Z}_{\varepsilon\delta}^{-}}$, which allows them to account for maximal canard trajectories.
\end{remark}

\subsection{Sector-type dynamics} 
 
Sector-type dynamics is typically encountered in two-timescale systems with one fast variable and two slow variables; it can be described by exploiting the near-integrable structure of Equation~\eqref{normal} in a vicinity of the canard point $p_{CN}^-$ \cite{krupa2008mixed, de2016sector}. Sector-type dynamics is realised when trajectories are attracted to $ \mathcal{Z}\big\lvert_{\mathcal{I}_{\rm can}+\mathcal{O}(\delta)}$, where $ \mathcal{I}_{\rm can} $ is given by \eqref{intervls}. (We emphasise that, for $\delta$ sufficiently small, $\mathcal{S}_{\varepsilon\delta}^{a^-}$ and $\mathcal{S}_{\varepsilon\delta}^r$ intersect in a canard trajectory that provides a connection between the two manifolds; recall \secref{local}.)
For $\varepsilon$ and $\delta$ sufficiently small and $z_{\rm in}\in\mathcal{I}_{\rm can}+\mathcal{O}(\delta)$, trajectories remain ``trapped" and undergo SAOs (``loops"), taking $ \mathcal{O}(\mu\delta\sqrt{-\varepsilon\ln\varepsilon}) $ steps in the $ z $-direction until they reach a point $ p_{\rm out} $ at which they can escape following the fast flow of Equation~\eqref{normal}. The 
$ z $-coordinate of that point can hence be approximated by 
\begin{align}
z_{\rm out} = z_{CN}^- +o(1).
\eqlab{sectout}
\end{align}
The number of SAOs that is observed in the corresponding trajectory is determined by the passage thereof through sectors of rotation \cite{krupa2008mixed}, the boundaries of which are so-called ``secondary" canards. Trajectories that are attracted to this region typically exhibit few SAOs of near-constant amplitude; see panel (b) of  \figref{sechof}, where sector-type SAOs are seen in between delay-type segments. A detailed study of sector-type dynamics in Equation~\eqref{normal} is part of work in progress; see again \cite{krupa2008mixed} for an in-depth discussion in the context of their prototypical model, Equation~\eqref{prototypical}.


\bibliographystyle{siam}
\bibliography{refs}

\begin{thebibliography}{10}

\bibitem{cardin2017fenichel}
{\sc P.~T. Cardin and M.~A. Teixeira}, {\em Fenichel theory for multiple time
  scale singular perturbation problems}, SIAM Journal on Applied Dynamical
  Systems, 16 (2017), pp.~1425--1452.

\bibitem{curtu2011interaction}
{\sc R.~Curtu and J.~Rubin}, {\em Interaction of canard and singular {H}opf
  mechanisms in a neural model}, SIAM Journal on Applied Dynamical Systems, 10
  (2011), pp.~1443--1479.

\bibitem{de2014three}
{\sc P.~De~Maesschalck, E.~Kutafina, and N.~Popovi{\'c}}, {\em Three
  time-scales in an extended {B}onhoeffer--van der {P}ol oscillator}, Journal
  of Dynamics and Differential Equations, 26 (2014), pp.~955--987.

\bibitem{de2016sector}
{\sc P.~De~Maesschalck, E.~Kutafina, and N.~Popovi{\'c}}, {\em
  Sector-delayed-{H}opf-type mixed-mode oscillations in a prototypical
  three-time-scale model}, Applied Mathematics and Computation, 273 (2016),
  pp.~337--352.

\bibitem{desroches2012mixed}
{\sc M.~Desroches, J.~Guckenheimer, B.~Krauskopf, C.~Kuehn, H.~M. Osinga, and
  M.~Wechselberger}, {\em Mixed-mode oscillations with multiple time scales},
  SIAM Review, 54 (2012), pp.~211--288.

\bibitem{desroches2018spike}
{\sc M.~Desroches and V.~Kirk}, {\em Spike-adding in a canonical
  three-time-scale model: superslow explosion and folded-saddle canards}, SIAM
  Journal on Applied Dynamical Systems, 17 (2018), pp.~1989--2017.

\bibitem{doedel2007auto}
{\sc E.~J. Doedel, A.~R. Champneys, F.~Dercole, T.~F. Fairgrieve, Y.~A.
  Kuznetsov, B.~Oldeman, R.~Paffenroth, B.~Sandstede, X.~Wang, and C.~Zhang},
  {\em {AUTO-07P}: Continuation and bifurcation software for ordinary
  differential equations}.
\newblock \texttt{http://www.macs.hw.ac.uk/gabriel/auto07/auto.html}, 2007.
\newblock Accessed on 09/09/2020.

\bibitem{doi2001complex}
{\sc S.~Doi, S.~Nabetani, and S.~Kumagai}, {\em Complex nonlinear dynamics of
  the {H}odgkin--{H}uxley equations induced by time scale changes}, Biological
  Cybernetics, 85 (2001), pp.~51--64.

\bibitem{fenichel1979geometric}
{\sc N.~Fenichel}, {\em Geometric singular perturbation theory for ordinary
  differential equations}, Journal of Differential Equations, 31 (1979),
  pp.~53--98.

\bibitem{guckenheimer2008singular}
{\sc J.~Guckenheimer}, {\em Singular hopf bifurcation in systems with two slow
  variables}, SIAM Journal on Applied Dynamical Systems, 7 (2008),
  pp.~1355--1377.

\bibitem{hayes2016geometric}
{\sc M.~G. Hayes, T.~J. Kaper, P.~Szmolyan, and M.~Wechselberger}, {\em
  Geometric desingularization of degenerate singularities in the presence of
  fast rotation: A new proof of known results for slow passage through hopf
  bifurcations}, Indagationes Mathematicae, 27 (2016), pp.~1184--1203.

\bibitem{HHmain}
{\sc A.~L. Hodgkin and A.~F. Huxley}, {\em A quantitative description of
  membrane current and its application to conduction and excitation in nerve},
  The Journal of Physiology, 117 (1952), pp.~500--544.

\bibitem{hhgspt2019}
{\sc P.~Kaklamanos, N.~Popovi\'c, and K.~U. Kristiansen}, {\em Geometric
  singular perturbation analysis of the multiple-timescale {H}odgkin-{H}uxley
  equations}, in preparation,  (2021).

\bibitem{koper1995bifurcations}
{\sc M.~T. Koper}, {\em Bifurcations of mixed-mode oscillations in a
  three-variable autonomous {V}an der {P}ol-{D}uffing model with a cross-shaped
  phase diagram}, Physica D: Nonlinear Phenomena, 80 (1995), pp.~72--94.

\bibitem{krupa2008mixed}
{\sc M.~Krupa, N.~Popovi{\'c}, and N.~Kopell}, {\em Mixed-mode oscillations in
  three time-scale systems: a prototypical example}, SIAM Journal on Applied
  Dynamical Systems, 7 (2008), pp.~361--420.

\bibitem{krupa2001extending}
{\sc M.~Krupa and P.~Szmolyan}, {\em Extending geometric singular perturbation
  theory to nonhyperbolic points---fold and canard points in two dimensions},
  SIAM Journal on Mathematical Analysis, 33 (2001), pp.~286--314.

\bibitem{krupa2010local}
{\sc M.~Krupa and M.~Wechselberger}, {\em Local analysis near a folded
  saddle-node singularity}, Journal of Differential Equations, 248 (2010),
  pp.~2841--2888.

\bibitem{kuehn2011decomposing}
{\sc C.~Kuehn}, {\em On decomposing mixed-mode oscillations and their return
  maps}, Chaos: An Interdisciplinary Journal of Nonlinear Science, 21 (2011),
  p.~033107.

\bibitem{letson2017analysis}
{\sc B.~Letson, J.~E. Rubin, and T.~Vo}, {\em Analysis of interacting local
  oscillation mechanisms in three-timescale systems}, SIAM Journal on Applied
  Mathematics, 77 (2017), pp.~1020--1046.

\bibitem{nan2014dynamical}
{\sc P.~Nan}, {\em Dynamical Systems Analysis of Biophysical Models with
  Multiple Timescales}, PhD thesis, ResearchSpace@ Auckland, 2014.

\bibitem{rubin2007giant}
{\sc J.~Rubin and M.~Wechselberger}, {\em Giant squid-hidden canard: the 3{D}
  geometry of the {H}odgkin--{H}uxley model}, Biological Cybernetics, 97
  (2007), pp.~5--32.

\bibitem{szmolyan2001canards}
{\sc P.~Szmolyan and M.~Wechselberger}, {\em Canards in $\mathbb{R}^3$},
  Journal of Differential Equations, 177 (2001), pp.~419--453.

\bibitem{szmolyan2004relaxation}
\leavevmode\vrule height 2pt depth -1.6pt width 23pt, {\em Relaxation
  oscillations in r3}, Journal of Differential Equations, 200 (2004),
  pp.~69--104.

\bibitem{wechselberger2005existence}
{\sc M.~Wechselberger}, {\em Existence and bifurcation of canards in
  $\mathbb{R}^3$ in the case of a folded node}, SIAM Journal on Applied
  Dynamical Systems, 4 (2005), pp.~101--139.

\end{thebibliography}
\end{document}